\documentclass[11pt]{article}
\usepackage{amssymb}
\usepackage{amsthm}
\usepackage{amsmath}
\usepackage[dvips]{epsfig}
\usepackage{graphicx}
\usepackage{color}
\usepackage{url}
\usepackage{tikz}
\usepackage{caption}
\usepackage{subcaption}
\usepackage{epstopdf}

\makeatletter\@addtoreset {equation}{section}\makeatother

\usetikzlibrary{matrix,arrows}
\newtheorem{thm}{Theorem}
\newtheorem{cor}{Corollary}
\newtheorem{prop}{Proposition}
\newtheorem{lem}{Lemma}
\newtheorem{mydef}{Definition}
\newtheorem{remark}{Remark}

{\begin{trivlist} \item[]{\em Proof }}%
{\hspace*{\fill}$\rule{.3\baselineskip}{.35\baselineskip}$\end{trivlist}}

\begin{document}

\title{\bf Existence of global solutions to the derivative NLS equation
with the inverse scattering transform method}

\author{Dmitry E. Pelinovsky and Yusuke Shimabukuro \\
{\small \it Department of Mathematics and Statistics, McMaster
University, Hamilton, Ontario, Canada, L8S 4K1 } }

\date{\today}
\maketitle

\begin{abstract}
We address existence of global solutions to the derivative nonlinear Schr\"{o}dinger (DNLS) equation
without the small-norm assumption. By using the inverse scattering transform method without eigenvalues and resonances,
we construct a unique global solution in $H^2(\mathbb{R}) \cap H^{1,1}(\mathbb{R})$ which is also Lipschitz continuous with
respect to the initial data. Compared to the existing literature on the spectral problem for the DNLS equation,
the corresponding Riemann--Hilbert problem is defined in the complex plane with the jump on the real line.
\end{abstract}

\tableofcontents

\newpage

\section{Introduction}

We consider the Cauchy problem for the derivative nonlinear Schr\"{o}dinger (DNLS) equation
\begin{equation}\label{dnls}
\left\{ \begin{array}{l} iu_t+u_{xx} + i(|u|^2 u)_x = 0, \quad t > 0,\\
u|_{t=0} = u_0, \end{array} \right.
\end{equation}
where the subscripts denote partial derivatives and $u_0$ is defined in a suitable function space,
e.g., in Sobolev space $H^m(\mathbb{R})$ of distributions with square integrable derivatives up to the order $m$.

Local existence of solutions for $u_0 \in H^s(\mathbb{R})$ with $s > \frac{3}{2}$
was established by Tsutsumi \& Fukuda \cite{TF1} by using a parabolic regularization.
Later, the same authors \cite{TF2}
used the first five conserved quantities of the DNLS equation and established
the global existence of solutions for $u_0 \in H^2(\mathbb{R})$
provided the initial data is small in the $H^1(\mathbb{R})$ norm.

Using a gauge transformation of the DNLS equation to a system of two semi-linear NLS equations,
for which a contraction argument can be used in the space $L^2(\mathbb{R})$ with the help of the Strichartz estimates,
Hayashi \cite{Hayashi} proved local and global existence of solutions to the DNLS equation
for $u_0 \in H^1(\mathbb{R})$ provided that the initial data is small in the $L^2(\mathbb{R})$ norm.
More specifically, the initial data $u_0$ is required to satisfy the precise inequality:
\begin{equation}\label{cond}
\|u_0\|_{L^2} < \sqrt{2\pi}.
\end{equation}
The space $H^1(\mathbb{R})$ is referred to as the energy space
for the DNLS equation because its first three conserved quantities
having the meaning of the mass, momentum, and energy are well-defined in
the space $H^1(\mathbb{R})$:
\begin{eqnarray}
\label{I0}
I_0 & = & \int_{\mathbb{R}} |u|^2 dx, \\
\label{I1}
I_1 & = & i \int_{\mathbb{R}} (\bar{u} u_x - u \bar{u}_x ) dx - \int_{\mathbb{R}} |u|^4 dx, \\
\label{I2}
I_2 & = & \int_{\mathbb{R}} |u_x|^2 dx + \frac{3 i}{4} \int_{\mathbb{R}} |u|^2
(u \bar{u}_x - u_x \bar{u} ) dx + \frac{1}{2} \int_{\mathbb{R}} |u|^6 dx.
\end{eqnarray}
Using the gauge transformation $u = v e^{-\frac{3i}{4} \int_{-\infty}^x |v(y)|^2 dy}$ and
the Gagliardo--Nirenberg inequality \cite{Weinstein}
\begin{eqnarray}
\| u \|_{L^6}^6 \leq \frac{4}{\pi^2} \| u \|_{L^2}^4 \| u_x \|_{L^2}^2,
\end{eqnarray}
one can obtain
$$
I_2 = \| v_x \|_{L^2}^2 - \frac{1}{16} \| v \|_{L^6}^6 \geq \left( 1 - \frac{1}{4 \pi^2} \| v \|_{L^2}^4 \right) \| v_x \|_{L^2}^2.
$$
Under the small-norm assumption (\ref{cond}), the $H^1(\mathbb{R})$ norm of
the function $v$ (and hence, the $H^1(\mathbb{R})$ norm of the solution
$u$ to the DNLS equation) is controlled by the conserved quantities $I_0$ and
$I_2$, once the local existence of solutions in $H^1(\mathbb{R})$ is established.

Developing the approach based on the gauge transformation and a priori energy estimates,
Hayashi \& Ozawa \cite{H-O-1,H-O-2,Ozawa} considered global solutions to the DNLS equation
in weighted Sobolev spaces under the same small-norm assumption (\ref{cond}), e.g., for
$u_0 \in H^m(\mathbb{R}) \cap L^{2,m}(\mathbb{R})$, where $m \in \mathbb{N}$.
Here and in what follows, $L^{2,m}(\mathbb{R})$ denotes the weighted $L^2(\mathbb{R})$ space
with the norm
$$
\| u \|_{L^{2,m}} := \left( \int_{\mathbb{R}} (1 + x^2)^m |u|^2 dx \right)^{1/2} =
\left( \int_{\mathbb{R}} \langle x \rangle^{2m} |u|^2 dx \right)^{1/2},
$$
where $\langle x \rangle := (1 + x^2)^{1/2}$.

More recently, local well-posedness of solutions to the DNLS equation was
established in spaces of lower regularity, e.g., for $u_0 \in H^s(\mathbb{R})$
with $s \geq \frac{1}{2}$ by Takaoka \cite{Takaoka-1} who used the Fourier transform
restriction method. This result was shown to
be sharp in the sense that the flow map fails to be uniformly continuous for
$s < \frac{1}{2}$ \cite{B-L}. Global existence under the constraint (\ref{cond})
was established in $H^s(\mathbb{R})$ with subsequent generations of the
Fourier transform restriction method and the so-called I-method, e.g.,
for $s > \frac{32}{33}$ by Takaoka \cite{Takaoka-2},
for $s > \frac{2}{3}$ and $s > \frac{1}{2}$ by Colliander {\em et al.} \cite{C-K-S-T-T-1} and \cite{C-K-S-T-T-2}
respectively, and finally for $s = \frac{1}{2}$ by Miao, Wu and Xu \cite{M-W-X}.

The key question, which goes back to the paper of Hayashi \& Ozawa \cite{H-O-1},
is to find out {\em if the bound (\ref{cond}) is optimal for existence of global
solutions to the DNLS equation}. By analogy with the quintic nonlinear Schr\"{o}dinger (NLS)
and Korteweg--de Vries (KdV) equations, one can ask if solutions with the $L^2(\mathbb{R})$ norm exceeding
the threshold value in the inequality (\ref{cond}) can blow up in a finite time.

The threshold value $\sqrt{2\pi}$ for the $L^2(\mathbb{R})$ norm corresponds to
the constant value of the $L^2(\mathbb{R})$ norm of the stationary solitary wave
solutions to the DNLS equation. These solutions can be written in the explicit form:
\begin{equation}
\label{soliton}
u(x,t) = \phi_{\omega}(x) e^{i \omega^2 t - \frac{3i}{4} \int_{-\infty}^x |\phi_{\omega}(y)|^2 dy},
\quad \phi_{\omega}(x) = \sqrt{4 \omega \; {\rm sech}(2 \omega x)}, \quad \omega \in \mathbb{R}^+,
\end{equation}
from which we have $\| \phi_{\omega} \|_{L^2} = \sqrt{2\pi}$ for every $\omega \in \mathbb{R}^+$. Although the
solitary wave solutions are unstable in the quintic NLS and KdV
equations, it was proved by Colin \& Ohta \cite{CO06} that the solitary wave of the DNLS equation
is orbitally stable with respect to perturbations in $H^1(\mathbb{R})$.
This result indicates that there exist global solutions to the DNLS equation (\ref{dnls})
in $H^1(\mathbb{R})$ with the $L^2(\mathbb{R})$ norm exceeding the threshold value in (\ref{cond}).

Moreover, Colin \& Ohta \cite{CO06} proved that the moving solitary wave solutions
of the DNLS equation are also orbitally stable in $H^1(\mathbb{R})$.
Since the $L^2(\mathbb{R})$ norm of the moving solitary wave solutions
is bounded from above by $2 \sqrt{\pi}$, the orbital stability result
indicates that there exist global solutions to the DNLS equation (\ref{dnls})
if the initial data $u_0$ satisfies the inequality
\begin{equation}
\label{cond-new}
\| u_0 \|_{L^2} < 2 \sqrt{\pi}.
\end{equation}
These orbital stability results suggest that the inequality (\ref{cond})
is not sharp for the global existence in the DNLS equation (\ref{dnls}).
Furthermore, recent numerical explorations of the DNLS equation (\ref{dnls})
indicate no blow-up phenomenon for initial data with any large $L^2(\mathbb{R})$ norm \cite{Sulem1,Sulem2}.
The same conclusion is indicated by the asymptotic analysis in the recent work \cite{Sulem3}.

Towards the same direction, Wu \cite{W} proved that the solution to the DNLS equations with $u_0 \in H^1(\mathbb{R})$ does
not blow up in a finite time if the $L^2(\mathbb{R})$ norm of the initial data $u_0$ slightly exceed
the threshold value in (\ref{cond}). The technique used in \cite{W}
is a combination of a variational argument together with the mass, momentum and energy
conservation in (\ref{I0})--(\ref{I2}). On the other hand, the solution to the DNLS equation restricted on the half line $\mathbb{R}^+$
blows up in a finite time if the initial data $u_0 \in H^2(\mathbb{R}^+) \cap L^{2,1}(\mathbb{R}^+)$
yields the negative energy $I_2 < 0$ given by (\ref{I2}) \cite{W}.
Proceeding further with sharper Gagliardo--Nirenberg-type inequalities, Wu  \cite{W-2} proved very recently
that the global solutions to the DNLS equation exists in $H^1(\mathbb{R})$ if the
initial data $u_0 \in H^1(\mathbb{R})$ satisfies the inequality (\ref{cond-new}), which is
larger than the inequality (\ref{cond}).

Our approach to address the same question concerning global existence in the
Cauchy problem for the DNLS equation \eqref{dnls}
without the small $L^2(\mathbb{R})$-norm assumption relies on a different technique
involving the inverse scattering transform theory \cite{B-C-1,B-C-2}. As was shown
by Kaup \& Newell \cite{KN78}, the DNLS equation appears to be a compatibility
condition for suitable solutions to the linear system given by
\begin{equation} \label{laxeq1}
\partial_x \psi = \left[ -i\lambda^2\sigma_3+\lambda Q(u) \right] \psi
\end{equation}
and
\begin{equation}\label{laxeq2}
\partial_t \psi = \left[ - 2i\lambda^4\sigma_3 + 2 \lambda^3 Q(u) + i \lambda^2 |u|^2 \sigma_3
- \lambda |u|^2 Q(u) + i \lambda \sigma_3 Q(u_x) \right] \psi,
\end{equation}
where $\psi \in \mathbb{C}^2$ is assumed to be a $C^2$ function of $x$ and $t$,
$\lambda \in \mathbb{C}$ is the $(x,t)$-independent spectral parameter, and
$Q(u)$ is the $(x,t)$-dependent matrix potential given by
\begin{equation}
\label{lax-matrices}
Q(u) = \left[\begin{matrix}0&u\\ -\overline{u} & 0\end{matrix}\right].
\end{equation}
The Pauli matrices that include $\sigma_3$ are given by
\begin{equation}
\label{Pauli}
\sigma_1 = \left[ \begin{matrix} 0 & 1 \\ 1 & 0 \end{matrix} \right], \quad
\sigma_2 = \left[\begin{matrix} 0&i\\ -i & 0\end{matrix}\right], \quad
\sigma_3 = \left[\begin{matrix} 1&0\\ 0 & -1\end{matrix}\right].
\end{equation}
A long but standard computation shows that the compatibility condition
$\partial_t\partial_x \psi =\partial_x\partial_t \psi$
for eigenfunctions $\psi \in C^2(\mathbb{R} \times \mathbb{R})$
is equivalent to the DNLS equation
$i u_t + u_{xx} + i (|u|^2 u)_x = 0$ for classical solutions $u$. The linear equation (\ref{laxeq1})
is usually referred to as the Kaup--Newell spectral problem.

In a similar context of the cubic NLS equation, it is well known that
the inverse scattering transform technique applied to the linear system
(associated with the so-called Zakharov--Shabat spectral problem)
provides a rigorous framework to solve the Cauchy problem in weighted $L^2$ spaces,
e.g., for $u_0 \in H^1(\mathbb{R}) \cap L^{2,1}(\mathbb{R})$ \cite{D-J,D-Z-1,Z} or for
$u_0 \in H^1(\mathbb{R}) \cap L^{2,s}(\mathbb{R})$ with $s > \frac{1}{2}$ \cite{C-P}.
In comparison with the spectral problem (\ref{laxeq1}), the Zakharov--Shabat spectral
problem has no multiplication of matrix potential $Q(u)$ by $\lambda$.
As a result, Neumann series solutions for the Jost functions of the Zakharov--Shabat
spectral problem converge if $u$ belongs to the space $L^1(\mathbb{R})$, see, e.g., Chapter 2 in \cite{A-P-T}.
As was shown originally by Deift \& Zhou \cite{D-Z-1,Z}, the inverse scattering
problem based on the Riemann--Hilbert problem of complex analysis with a jump
along the real line can be solved uniquely if $u$ is defined in a subspace of $L^{2,1}(\mathbb{R})$,
which is continuously embedded into the space $L^1(\mathbb{R})$.
The time evolution of the scattering data is well defined
if $u$ is posed in space $H^1(\mathbb{R}) \cap L^{2,1}(\mathbb{R})$ \cite{D-J,D-Z-1}.

For the Kaup--Newell spectral problem (\ref{laxeq1}), the key feature is the presence of the
spectral parameter $\lambda$ that multiplies the matrix potential $Q(u)$. As a result,
Neumann series solutions for the Jost functions do not converge uniformly
if $u$ is only defined in the space $L^1(\mathbb{R})$. Although the Lax system
(\ref{laxeq1})--(\ref{laxeq2}) appeared long ago and was used many times
for formal methods, such as construction of soliton solutions \cite{KN78},
temporal asymptotics \cite{V1,V2}, and long-time asymptotic expansions \cite{X-F,X-F-C},
no rigorous results on the function spaces for the matrix $Q(u)$ have been obtained so far to ensure
bijectivity of the direct and inverse scattering transforms for the Kaup--Newell spectral
problem (\ref{laxeq1}).

In this connection, we mention the works of Lee \cite{Lee1,Lee2} on the local solvability of a generalized
Lax system with $\lambda^n$ dependence for an integer $n \geq 2$ and generic small initial
data $u_0$ in Schwarz class. In the follow-up paper \cite{Lee}, Lee also claimed existence of
a global solution to the Cauchy problem (\ref{dnls}) for large $u_0$ in Schwarz class, but the analysis
of \cite{Lee} relies on a ``Basic Lemma", where the Jost functions are claimed to be defined
for $u_0$ in $L^2(\mathbb{R})$. However, equation (\ref{laxeq1}) shows
that the condition $u_0 \in L^2(\mathbb{R})$ is insufficient
for construction of the Jost functions uniformly in $\lambda$.

We address the bijectivity of the direct and inverse scattering transform
for the Lax system (\ref{laxeq1})--(\ref{laxeq2}) in this work.
We show that the direct scattering transform for
the Jost functions of the Lax system (\ref{laxeq1})--(\ref{lax-jost-2})
can be developed under the requirement $u_0 \in L^1(\mathbb{R}) \cap L^{\infty}(\mathbb{R})$
and $\partial_x u_0 \in L^1(\mathbb{R})$. This requirement is satisfied if
$u_0$ is defined in the weighted Sobolev space $H^{1,1}(\mathbb{R})$ defined by
\begin{equation}
\label{H-1-1}
H^{1,1}(\mathbb{R}) = \left\{ u \in L^{2,1}(\mathbb{R}), \quad \partial_x u \in L^{2,1}(\mathbb{R}) \right\}.
\end{equation}
Note that it is quite common to use notation $H^{1,1}(\mathbb{R})$ to denote
$H^1(\mathbb{R}) \cap L^{2,1}(\mathbb{R})$ \cite{D-Z-1,Z}, which is not what is
used here in (\ref{H-1-1}). Moreover, we show that asymptotic expansions
of the Jost functions are well defined if $u_0 \in H^2(\mathbb{R}) \cap H^{1,1}(\mathbb{R})$,
which also provide a rigorous framework to study the inverse scattering transform based on
the Riemann--Hilbert problem of complex analysis. Finally, the time
evolution of the scattering data is well defined if $u_0 \in H^2(\mathbb{R}) \cap H^{1,1}(\mathbb{R})$.

We shall now define eigenvalues and resonances for the spectral problem (\ref{laxeq1})
and present the global existence result for the DNLS equation (\ref{dnls}).

\begin{mydef}
\label{definition-eigenvalue}
We say that $\lambda \in \mathbb{C}$ is an eigenvalue of the spectral problem (\ref{laxeq1})
if the linear equation (\ref{laxeq1}) with this $\lambda$ admits a solution in $L^2(\mathbb{R})$.
\end{mydef}

\begin{mydef}
\label{definition-resonance}
We say that $\lambda \in \mathbb{R} \cup i \mathbb{R}$ is a resonance of the spectral problem (\ref{laxeq1})
if the linear equation (\ref{laxeq1}) with this $\lambda$ admits a solution in $L^{\infty}(\mathbb{R})$
with the asymptotic behavior
$$
\psi(x) \sim \left\{ \begin{array}{l} a_+ e^{-i \lambda^2 x} e_1, \quad x \to -\infty, \\
a_- e^{+i \lambda^2 x} e_2, \quad x \to +\infty, \end{array} \right.
$$
where $a_+$ and $a_-$ are nonzero constant coefficients, whereas $e_1 = [1,0]^t$ and $e_2 = [0,1]^t$.
\end{mydef}

\begin{thm}\label{main}
For every $u_0 \in H^2(\mathbb{R}) \cap H^{1,1}(\mathbb{R})$ such that the linear equation (\ref{laxeq1}) admits no
eigenvalues or resonances in the sense of Definitions \ref{definition-eigenvalue} and \ref{definition-resonance},
there exists a unique global solution $u(t,\cdot) \in H^2(\mathbb{R}) \cap H^{1,1}(\mathbb{R})$
of the Cauchy problem \eqref{dnls} for every $t \in \mathbb{R}$. Furthermore, the map
$$
H^2(\mathbb{R}) \cap H^{1,1}(\mathbb{R}) \ni u_0 \mapsto u \in C(\mathbb{R};H^2(\mathbb{R}) \cap H^{1,1}(\mathbb{R}))
$$
is Lipschitz continuous.
\end{thm}

\begin{remark}
A sufficient condition that the spectral problem (\ref{laxeq1}) admits no eigenvalues was found in \cite{Pel-survey}.
This condition is satisfied under the small-norm
assumption on the $H^{1,1}(\mathbb{R})$ norm of the initial data $u_0$. See Remark \ref{eigfunc} below.
Although we believe that there exist initial data $u_0$ with large $H^{1,1}(\mathbb{R})$ norm that yield
no eigenvalues in the spectral problem (\ref{laxeq1}), we have no constructive examples of such initial data. 
Nevertheless, a finite number of eigenvalues $\lambda \in \mathbb{C}$ in the spectral problem (\ref{laxeq1})
can be included by using algebraic methods such as the Backl\"und, Darboux, or dressing transformations \cite{Contrera,C-P}.
\end{remark}

\begin{remark}
The condition that the spectral problem (\ref{laxeq1}) admits no resonance
is used to identify the so-called generic initial data $u_0$.
The non-generic initial data $u_0$ violating this condition are at the threshold case in the sense that
a small perturbation to $u_0$ may change the number of eigenvalues $\lambda$ in the linear equation (\ref{laxeq1}).
\end{remark}

\begin{remark}
Compared to the results of Hayashi \& Ozawa \cite{Hayashi,H-O-1,H-O-2,Ozawa},
where global well-posedness of the Cauchy problem for the DNLS equation (\ref{dnls})
was established in $H^2(\mathbb{R}) \cap H^{1,1}(\mathbb{R})$ under the small $L^2(\mathbb{R})$ norm assumption (\ref{cond}),
the inverse scattering transform theory is developed without the smallness assumption on the initial data $u_0$.
\end{remark}

\begin{remark}
An alternative proof of Theorem \ref{main} is developed in \cite{Perry} by using a different version
of the inverse scattering transform for the Lax system (\ref{laxeq1})--(\ref{laxeq2}). The results
of \cite{Perry} are formulated in space $H^2(\mathbb{R}) \cap L^{2,2}(\mathbb{R})$, which is embedded
into space $H^2(\mathbb{R}) \cap H^{1,1}(\mathbb{R})$.
\end{remark}

The paper is organized as follows. Section 2 reports the solvability results
on the direct scattering transform for the spectral problem (\ref{laxeq1}).
Section 3 gives equivalent formulations of the Riemann--Hilbert problem
associated with the spectral problem (\ref{laxeq1}). Section 4 is devoted to the solvability
results on the inverse scattering transform for the spectral problem (\ref{laxeq1}). Section 5 incorporates
the time evolution of the linear equation (\ref{laxeq2}) and contains the proof of
Theorem \ref{main}.

\section{Direct scattering transform}

The direct scattering transform is developed for the Kaup--Newell spectral problem
(\ref{laxeq1}), which we rewrite here for convenience:
\begin{equation}\label{lax1}
\partial_x \psi  = \left[ -i \lambda^2 \sigma_3 + \lambda Q(u) \right] \psi,
\end{equation}
where $\psi \in \mathbb{C}^2$, $\lambda \in \mathbb{C}$, and
the matrices $Q(u)$ and $\sigma_3$ are given by (\ref{lax-matrices}) and (\ref{Pauli}).

The formal construction of the Jost functions is based on the construction of
the fundamental solution matrices $\Psi_{\pm}(x;\lambda)$ of the linear equation
\eqref{lax1}, which satisfy the same asymptotic behavior at infinity
as the linear equation \eqref{lax1} with $Q(u) \equiv 0$:
\begin{equation}
\label{potential-free}
\Psi_{\pm}(x;\lambda) \to e^{-i\lambda^2x\sigma_3} \quad \mbox{as}\; x \to \pm \infty,
\end{equation}
where parameter $\lambda$ is fixed in an unbounded subset of $\mathbb{C}$. However,
the standard fixed point argument for Volterra's integral equations
associated with the linear equation (\ref{lax1}) is not uniform in $\lambda$ as $|\lambda| \to \infty$
if $Q(u) \in L^1(\mathbb{R})$.
Integrating by parts, it was suggested in \cite{Pel-survey}
that uniform estimates on the Jost functions of the linear equation (\ref{lax1})
can be obtained under the condition
$$
\|u\|_{L^1}(\|u\|_{L^{\infty}}+\|\partial_xu\|_{L^1})<\infty.
$$
Here we explore this idea further and introduce a transformation
of the linear equation (\ref{lax1}) to a spectral problem
of the Zakharov--Shabat type. This will allow us to adopt
the direct and inverse scattering transforms, which were previously used
for the cubic NLS equation \cite{D-Z-1,Z} (see also \cite{C-P,D-J} for review).
Note that the pioneer idea of a transformation
of the linear equation (\ref{lax1}) to a spectral problem
of the Zakharov--Shabat type can be found already in the formal work of Kaup \& Newell \cite{KN78}.

Let us define the transformation matrices for any $u \in L^{\infty}(\mathbb{R})$ and $\lambda \in \mathbb{C}$,
\begin{equation} \label{T}
T_1(x;\lambda) = \left[\begin{matrix} 1 & 0 \\-\overline{u}(x) & 2i\lambda \end{matrix}\right] \quad\mbox{and}
\quad T_2(x;\lambda) = \left[\begin{matrix} 2i\lambda & -u(x) \\0 & 1 \end{matrix}\right],
\end{equation}
If the vector $\psi \in \mathbb{C}^2$ is transformed by
$\psi_{1,2} = T_{1,2} \psi$, then straightforward computations show that
$\psi_{1,2}$ satisfy the linear equations
\begin{equation}\label{lax-jost-1}
\partial_x \psi_1 = \left[ -i \lambda^2 \sigma_3 + Q_1(u) \right] \psi_1, \quad
Q_1(u) = \frac{1}{2i} \left[\begin{matrix} |u|^2 & u \\ -2i \overline{u}_x - \overline{u}|u|^2 & -|u|^2 \end{matrix}\right]
\end{equation}
and
\begin{equation}\label{lax-jost-2}
\partial_x \psi_2 = \left[ -i \lambda^2 \sigma_3 + Q_2(u) \right] \psi_2, \quad
Q_2(u) = \frac{1}{2i} \left[\begin{matrix} |u|^2 & -2 i u_x + u|u|^2 \\  -\overline{u} &   -|u|^2 \end{matrix}\right].
\end{equation}
Note that $Q_{1,2}(u) \in L^1(\mathbb{R})$ if
$u \in L^1(\mathbb{R}) \cap L^3(\mathbb{R})$ and
$\partial_x u \in L^1(\mathbb{R})$. The linear equations (\ref{lax-jost-1}) and
(\ref{lax-jost-2}) are of the Zakharov--Shabat-type, after we introduce
the complex variable $z = \lambda^2$. In what follows, we study
the Jost functions and the scattering coefficients for the linear equations
(\ref{lax-jost-1}) and (\ref{lax-jost-2}).

\subsection{Jost functions}

Let us introduce the normalized Jost functions
from solutions $\psi_{1,2}$ of the linear equations
(\ref{lax-jost-1}) and (\ref{lax-jost-2}) with $z = \lambda^2$ in the form
\begin{equation}\label{lax-jost}
m_{\pm}(x;z) = \psi_1(x;z) e^{i x z}, \quad
n_{\pm}(x;z) = \psi_2(x;z) e^{-i x z},
\end{equation}
according to the asymptotic  behavior
\begin{equation}\label{lax-jost-infinity}
\left. \begin{array}{l}
m_{\pm}(x;z) \to e_1, \\
n_{\pm}(x;z) \to e_2, \end{array} \right\}
\quad \mbox{\rm as} \quad x \to \pm \infty,
\end{equation}
where $e_1 = [1,0]^t$ and $e_2 = [0,1]^t$.
The normalized Jost functions satisfy the following Volterra's
integral equations
\begin{equation}\label{int-jost-1}
m_{\pm}(x;z) = e_1 + \int_{\pm \infty}^{x}\left[\begin{matrix} 1 & 0 \\ 0 & e^{2iz (x-y)}
\end{matrix}\right] Q_1(u(y)) m_{\pm}(y;z) dy
\end{equation}
and
\begin{equation} \label{int-jost-2}
n_{\pm}(x;z) = e_2 + \int_{\pm \infty}^{x}\left[\begin{matrix} e^{-2iz (x-y)} & 0 \\ 0 & 1
\end{matrix}\right] Q_2(u(y)) n_{\pm}(y;z) dy.
\end{equation}
The next two lemmas describe properties of the Jost functions,
which are analogues to similar properties of the Jost functions
in the Zakharov--Shabat spectral problem (see, e.g., Lemma 2.1 in \cite{A-P-T}).

\begin{lem} \label{Jost}
Let $u \in L^1(\mathbb{R}) \cap L^3(\mathbb{R})$ and
$\partial_x u \in L^1(\mathbb{R})$. For every $z \in \mathbb{R}$,
there exist unique solutions $m_{\pm}(\cdot;z) \in L^{\infty}(\mathbb{R})$ and
$n_{\pm}(\cdot;z) \in L^{\infty}(\mathbb{R})$ satisfying the integral
equations (\ref{int-jost-1}) and (\ref{int-jost-2}).
Moreover, for every $x \in \mathbb{R}$,
$m_-(x;\cdot)$ and $n_+(x;\cdot)$ are continued analytically
in $\mathbb{C}^+$, whereas $m_+(x;\cdot)$ and $n_-(x;\cdot)$ are continued analytically
in $\mathbb{C}^-$. Finally, there exists a positive $z$-independent constant $C$ such that
\begin{equation}
\label{bounds-uniform-Jost}
\| m_{\mp}(\cdot;z) \|_{L^{\infty}} + \| n_{\pm}(\cdot;z) \|_{L^{\infty}} \leq C, \quad z \in \mathbb{C}^{\pm}.
\end{equation}
\end{lem}

\begin{proof}
It suffices to prove the statement for one Jost function, e.g., for $m_-$.
The proof for other Jost functions is analogous. Let us define the integral operator $K$ by
\begin{equation}\label{kernel-Q-1}
(K f)(x;z) := \frac{1}{2i} \int_{-\infty}^{x}
\left[\begin{matrix} 1 & 0 \\ 0 & e^{2iz (x-y)}
\end{matrix}\right] \left[\begin{matrix} |u(y)|^2 & u(y) \\ -2i \partial_y \overline{u}(y) - \overline{u(y)}
|u(y)|^2 & -|u(y)|^2 \end{matrix}\right] f(y) dy.
\end{equation}
For every $z \in \mathbb{C}^+$ and every $x_0 \in \mathbb{R}$, we have
$$
\| (K f)(\cdot;z)\|_{L^{\infty}(-\infty,x_0)} \leq
\frac{1}{2} \left[\begin{matrix} \|u\|_{L^2(-\infty,x_0)}^2 & \|u\|_{L^1(-\infty,x_0)} \\
2 \|\partial_x u\|_{L^1(-\infty,x_0)}+ \|u\|_{L^3(-\infty,x_0)}^3 & \|u\|_{L^2(-\infty,x_0)}^2 \end{matrix}\right]
\| f(\cdot;z) \|_{L^{\infty}(-\infty,x_0)}.
$$
The operator $K$ is a contraction from $L^{\infty}(-\infty,x_0)$ to
$L^{\infty}(-\infty,x_0)$ if the two eigenvalues of the matrix
$$
A = \frac{1}{2} \left[\begin{matrix} \|u\|_{L^2(-\infty,x_0)}^2 & \|u\|_{L^1(-\infty,x_0)} \\
2 \|\partial_x u\|_{L^1(-\infty,x_0)}+ \|u\|_{L^3(-\infty,x_0)}^3 & \|u\|_{L^2(-\infty,x_0)}^2 \end{matrix}\right]
$$
are located inside the unit circle. The two eigenvalues are given by
$$
\lambda_{\pm} = \frac{1}{2} \|u\|_{L^2(-\infty,x_0)}^2 \pm \frac{1}{2}
\sqrt{\|u\|_{L^1(-\infty,x_0)} (2 \|\partial_x u\|_{L^1(-\infty,x_0)} + \|u\|_{L^3(-\infty,x_0)}^3)},
$$
so that $|\lambda_-| < |\lambda_+|$.
Hence, the operator $K$ is a contraction if $x_0 \in \mathbb{R}$ is chosen
so that
\begin{equation}\label{estimate1}
\frac{1}{2} \|u\|_{L^2(-\infty,x_0)}^2 + \frac{1}{2}
\sqrt{\|u\|_{L^1(-\infty,x_0)} (2 \|\partial_x u\|_{L^1(-\infty,x_0)} + \|u\|_{L^3(-\infty,x_0)}^3)} < 1.
\end{equation}
By the Banach Fixed Point Theorem, for this $x_0$ and every $z \in \mathbb{C}^+$,
there exists a unique solution $m_-(\cdot;z) \in L^{\infty}(-\infty,x_0)$ of the integral equation (\ref{int-jost-1}).
To extend this result to $L^{\infty}(\mathbb{R})$, we can split $\mathbb{R}$ into a finite number of subintervals
such that the estimate \eqref{estimate1} is satisfied in each subinterval. Unique solutions in each subinterval
can be glued together to obtain the unique solution $m_-(\cdot;z) \in L^{\infty}(\mathbb{R})$ for every $z \in \mathbb{C}^+$.	

Analyticity of $m_-(x;\cdot)$ in $\mathbb{C}^+$ for every $x \in \mathbb{R}$
follows from the absolute and uniform convergence of the Neumann series of analytic functions in $z$. 	
Indeed, let us denote the $L^1$ matrix norm of the $2$-by-$2$ matrix function $Q$ as
$$
\| Q \|_{L^1} := \sum_{i,j = 1}^2 \| Q_{i,j} \|_{L^1}.
$$
If $u\in H^{1,1}(\mathbb{R})$, then $u \in L^1(\mathbb{R}) \cap L^3(\mathbb{R})$ and $\partial_x u \in L^1(\mathbb{R})$
so that $Q_1(u) \in L^1(\mathbb{R})$, where the matrix $Q_1(u)$ appears
in the integral kernel $K$ given by (\ref{kernel-Q-1}). For every $f(x;z) \in L^{\infty}(\mathbb{R}\times \mathbb{C}^+)$, we have
\begin{equation}
\label{bound-on-K-kernel}
\| (K^n f) \|_{L^{\infty}} \leq \frac{1}{n!} \| Q_1(u) \|_{L^1}^n \|f\|_{L^{\infty}}.
\end{equation}
As a result, the Neumann series for Volterra's integral equation (\ref{int-jost-1}) for $m_-$
converges absolutely and uniformly for every $x \in \mathbb{R}$ and $z \in \mathbb{C}^+$
and contains analytic functions of $z$ for $z \in \mathbb{C}^+$. Therefore, $m_-(x;\cdot)$ is analytic in $\mathbb{C}^+$
for every $x \in \mathbb{R}$ and it satisfies the bound (\ref{bounds-uniform-Jost}).
\end{proof}

\begin{remark} \label{eigfunc}
If $u$ is sufficiently small so that the estimate
\begin{equation}\label{estimate2}
\frac{1}{2} \|u\|_{L^2}^2 + \frac{1}{2}
\sqrt{\|u\|_{L^1} (2 \|\partial_x u\|_{L^1} + \|u\|_{L^3}^3)} < \frac{1}{2}
\end{equation}
holds on $\mathbb{R}$, then Banach Fixed Point Theorem
yields the existence of the unique solution $m_-(\cdot;z) \in L^{\infty}(\mathbb{R})$
of the integral equation (\ref{int-jost-1}) such that $\| m_-(\cdot;z) - e_1 \|_{L^{\infty}} < 1$.
This is in turn equivalent to the conditions that the linear equation (\ref{lax1})
has no $L^2(\mathbb{R})$ solutions for every $\lambda \in \mathbb{C}$
and the linear equation (\ref{lax1}) has no resonances for every $\lambda \in \mathbb{R} \cup i \mathbb{R}$
in the sense of Definitions \ref{definition-eigenvalue} and \ref{definition-resonance}.
Therefore, the small-norm constraint (\ref{estimate2}) is a sufficient condition
that the assumptions of Theorem \ref{main} are satisfied.
\end{remark}

\begin{lem} \label{asympt-mod}
Under the conditions of Lemma \ref{Jost}, for every $x \in \mathbb{R}$,
the Jost functions $m_{\pm}(x;z)$ and $n_{\pm}(x;z)$ satisfy the following limits
as $|{\rm Im}(z)| \to \infty$ along a contour in the domains of their analyticity:
\begin{equation}
\label{asymptotics-1-lim}
\lim_{|z| \to \infty} m_{\pm}(x;z) = m_{\pm}^{\infty}(x) e_1, \quad
m_{\pm}^{\infty}(x) := e^{\frac{1}{2i} \int_{\pm\infty}^x |u(y)|^2 dy}
\end{equation}
and
\begin{equation}
\label{asymptotics-2-lim}
\lim_{|z| \to \infty} n_{\pm}(x;z) = n_{\pm}^{\infty}(x) e_2, \quad
n_{\pm}^{\infty}(x) := e^{-\frac{1}{2i} \int_{\pm\infty}^x |u(y)|^2 dy}.
\end{equation}
If in addition, $u \in C^1(\mathbb{R})$, then for every $x \in \mathbb{R}$,
the Jost functions $m_{\pm}(x;z)$ and $n_{\pm}(x;z)$ satisfy the following limits
as $|{\rm Im}(z)| \to \infty$ along a contour in the domains of their analyticity:
\begin{equation}
\label{asymptotics-1}
\lim_{|z| \to \infty} z \left[ m_{\pm}(x;z) - m_{\pm}^{\infty}(x) e_1 \right] =
q_{\pm}^{(1)}(x) e_1 + q_{\pm}^{(2)}(x) e_2
\end{equation}
and
\begin{equation}
\label{asymptotics-2}
\lim_{|z| \to \infty} z \left[ n_{\pm}(x;z) - n_{\pm}^{\infty}(x) e_2 \right] =
s_{\pm}^{(1)}(x) e_1 + s_{\pm}^{(2)}(x) e_2,
\end{equation}
where
\begin{eqnarray*}
q_{\pm}^{(1)}(x) & := & - \frac{1}{4}  e^{\frac{1}{2i} \int_{\pm\infty}^x |u(y)|^2 dy}
\int_{\pm \infty}^{x} \left[ u(y) \partial_y \bar{u}(y) + \frac{1}{2i} |u(y)|^4 \right]  dy,\\
q_{\pm}^{(2)}(x) & := & \frac{1}{2i} \partial_x \left(\bar{u}(x) e^{\frac{1}{2i} \int_{\pm\infty}^x |u(y)|^2 dy} \right), \\
s_{\pm}^{(1)}(x) & := & - \frac{1}{2i} \partial_x \left( u(x) e^{-\frac{1}{2i} \int_{\pm\infty}^x |u(y)|^2 dy} \right),\\
s_{\pm}^{(2)}(x) & := & \frac{1}{4} e^{-\frac{1}{2i} \int_{\pm\infty}^x |u(y)|^2 dy}
\int_{\pm \infty}^{x} \left[ \bar{u}(y) \partial_y u(y) - \frac{1}{2i} |u(y)|^4 \right]  dy.
\end{eqnarray*}
\end{lem}

\begin{proof}
Again, we prove the statement for the Jost function $m_-$ only. The proof for other Jost functions is analogous.
Let $m_- = [m_-^{(1)},m_-^{(2)}]^t$ and rewrite the integral equation (\ref{int-jost-1}) in the component form:
\begin{eqnarray}\label{int-jost-asymp-1}
m_-^{(1)}(x;z) = 1 + \frac{1}{2i} \int_{-\infty}^{x} u(y) \left[\bar{u}(y) m_-^{(1)}(y;z) + m_-^{(2)}(y;z)\right] dy,
\end{eqnarray}
and
\begin{eqnarray}
\label{int-jost-asymp-2}
m_-^{(2)}(x;z) = \phantom{t} - \frac{1}{2i} \int_{-\infty}^{x} e^{2iz (x-y)}
\left[ (2 i \partial_y \bar{u}(y) + |u(y)|^2 \bar{u}(y)) m_-^{(1)}(y;z) + |u(y)|^2 m_-^{(2)}(y;z) \right] dy.
\end{eqnarray}
Recall that for every $x \in \mathbb{R}$, $m_-(x;\cdot)$ is analytic in $\mathbb{C}^+$.
By bounds (\ref{bounds-uniform-Jost}) in Lemma \ref{Jost}, for every $u \in L^1(\mathbb{R}) \cap L^3(\mathbb{R})$ and
$\partial_x u \in L^1(\mathbb{R})$, the integrand of the second equation (\ref{int-jost-asymp-2})
is bounded for every $z \in \mathbb{C}^+$ by an absolutely integrable $z$-independent function. Also,
the integrand converges to zero
for every $y \in (-\infty,x)$ as $|z| \to \infty$ in $\mathbb{C}^+$. By Lebesgue's Dominated Convergence Theorem,
we obtain $\lim_{|z| \to \infty} m_-^{(2)}(x;z) = 0$, hence $m_-^{\infty}(x) := \lim_{|z| \to \infty} m_-^{(1)}(x;z)$
satisfies the inhomogeneous integral equation
\begin{eqnarray}
m_-^{\infty}(x) = 1 + \frac{1}{2i} \int_{-\infty}^{x} |u(y)|^2 m_-^{\infty}(y) dy,
\label{int-jost-asymp-3}
\end{eqnarray}
with the unique solution $m_-^{\infty}(x) = e^{\frac{1}{2i} \int_{-\infty}^x |u(y)|^2 dy}$.
This proves the limit (\ref{asymptotics-1-lim}) for $m_-$.

We now add the condition $u \in C^1(\mathbb{R})$ and use the technique behind Watson's Lemma
related to the Laplace method of asymptotic analysis \cite{Miller}. For every $x \in \mathbb{R}$ and
every small $\delta > 0$, we split integration in the second equation (\ref{int-jost-asymp-2})
for $(-\infty,x-\delta)$ and $(x-\delta,x)$, rewriting it in the equivalent form:
\begin{eqnarray}
\nonumber
m_-^{(2)}(x;z) & = & \int_{-\infty}^{x-\delta} e^{2iz(x-y)} \phi(y;z) dy +
\phi(x;z) \int_{x-\delta}^x e^{2iz(x-y)}  dy \\
& \phantom{t} & + \int_{x-\delta}^x e^{2iz(x-y)}  \left[ \phi(y;z) - \phi(x;z) \right] dy \equiv I + II + III,
\label{int-jost-asymp-4}
\end{eqnarray}
where
$$
\phi(x;z) := - \frac{1}{2i} \left[ (2 i \partial_x \bar{u}(x) + |u(x)|^2 \bar{u}(x)) m_-^{(1)}(x;z) +
|u(x)|^2 m_-^{(2)}(x;z) \right].
$$
Since $\phi(\cdot;z) \in L^1(\mathbb{R})$, we have
$$
|I| \leq e^{-2 \delta {\rm Im}(z)} \| \phi(\cdot;z) \|_{L^1}.
$$
Since $\phi(\cdot;z) \in C^0(\mathbb{R})$, we have
$$
|III| \leq \frac{1}{2 {\rm Im}(z)} \| \phi(x - \cdot;z) - \phi(x;z) \|_{L^{\infty}(x-\delta,x)}.
$$
On the other hand, we have the exact value
$$
II = -\frac{1}{2iz} \left[ 1 - e^{2 i z \delta} \right] \phi(x;z).
$$
Let us choose $\delta := [{\rm Im}(z)]^{-1/2}$ such that $\delta \to 0$ as ${\rm Im}(z) \to \infty$.
Then, by taking the limit  along the contour in $\mathbb{C}^+$ such that ${\rm Im}(z) \to \infty$,
we obtain
\begin{eqnarray}
\lim_{|z| \to \infty} z m_-^{(2)}(x;z) = -\frac{1}{2i} \lim_{|z| \to \infty} \phi(x;z)
= - \frac{1}{4} (2 i \partial_x \bar{u}(x) + |u(x)|^2 \bar{u}(x)) m_-^{\infty}(x),
\label{int-jost-asymp-5}
\end{eqnarray}
which yields the limit (\ref{asymptotics-1}) for $m_-^{(2)}$. On the other hand,
the first equation \eqref{int-jost-asymp-1} can be rewritten as the differential equation
$$
\partial_x m_-^{(1)}(x;z) = \frac{1}{2i} |u(x)|^2 m_-^{(1)}(x;z) + \frac{1}{2i} u(x) m_-^{(2)}(x;z).
$$
Using $\bar{m}_-^{\infty}$ as the integrating factor,
$$
\partial_x(\overline{m}_-^{\infty}(x) m_-^{(1)}(x;z))=\frac{1}{2i} u(x) \overline{m}_-^{\infty}(x) m_-^{(2)}(x;z),
$$
we obtain another integral equation for $m_-^{(1)}$:
\begin{eqnarray}\label{int-jost-asymp-6}
m_-^{(1)}(x;z) = m_-^{\infty}(x) + \frac{1}{2i} m_-^{\infty}(x) \int_{-\infty}^{x} u(y) \overline{m}_-^{\infty}(y) m_-^{(2)}(y;z) dy,
\end{eqnarray}
Multiplying this equation by $z$ and taking the limit $|z| \to \infty$, we obtain
\begin{eqnarray}
\lim_{|z| \to \infty} z \left[ m_-^{(1)}(x;z) - m_-^{\infty}(x) \right] =
- \frac{1}{4} m_-^{\infty}(x) \int_{-\infty}^{x} \left[ u(y) \partial_y \bar{u}(y) + \frac{1}{2i} |u(y)|^4 \right]  dy,
\label{int-jost-asymp-7}
\end{eqnarray}
which yields the limit (\ref{asymptotics-1}) for $m_-^{(1)}$.
\end{proof}

We shall now study properties of the Jost functions
on the real axis of $z$. First, we note that following elementary
result from the Fourier theory. For notational convenience,
we use sometimes $\| f(z) \|_{L^2_z}$ instead of $\| f(\cdot) \|_{L^2}$.

\begin{prop}
\label{prop-Fourier-1}
If $w \in H^1(\mathbb{R})$, then
\begin{equation}
\label{Fourier-1b}
\sup_{x \in \mathbb{R}} \left\| \int_{- \infty}^x e^{2 i z (x-y)} w(y) dy \right\|_{L^2_z(\mathbb{R})} \leq
\sqrt{\pi} \| w \|_{L^2}.
\end{equation}
and
\begin{equation}
\label{Fourier-1a}
\sup_{x \in \mathbb{R}} \left\| 2 i z \int_{- \infty}^x e^{2 i z (x-y)} w(y) dy + w(x) \right\|_{L^2_z(\mathbb{R})} \leq
\sqrt{\pi} \| \partial_x w \|_{L^2}.
\end{equation}
Moreover, if $w \in L^{2,1}(\mathbb{R})$, then for every $x_0 \in \mathbb{R}^{-}$, we have
\begin{equation}
\label{Fourier-1}
\sup_{x \in (-\infty,x_0)} \left\| \langle x \rangle \int_{- \infty}^x e^{2 i z (x-y)} w(y) dy \right\|_{L^2_z(\mathbb{R})} \leq
\sqrt{\pi} \| w \|_{L^{2,1}(- \infty,x_0)},
\end{equation}
where $\langle x \rangle := (1+x^2)^{1/2}$.
\end{prop}

\begin{proof}
Here we give a quick proof based on Plancherel's theorem of Fourier analysis.
For every $x \in \mathbb{R}$ and every $z \in \mathbb{R}$, we write
$$
f(x;z) := \int_{- \infty}^x e^{2 i z (x-y)} w(y) dy = \int_{-\infty}^0 e^{-2 i z y} w(y+x) dy,
$$
so that
\begin{eqnarray}
\nonumber
\| f(x;\cdot) \|_{L^2}^2 & = & \int_{-\infty}^{\infty} \int_{-\infty}^0 \int_{-\infty}^0 \bar{w}(y_1+x) w(y_2+x) e^{2 i (y_1-y_2) z} dy_1 dy_2 dz \\
& = &  \pi \int_{-\infty}^0 |w(y+x)|^2 dy = \pi \int_{-\infty}^x |w(y)|^2 dy. \label{plancherel}
\end{eqnarray}
Bound (\ref{Fourier-1b}) holds if $w \in L^2(\mathbb{R})$.

If $y \leq x \leq 0$, we have $1 + y^2 \geq 1 + x^2$, so that
equation (\ref{plancherel}) implies
\begin{eqnarray*}
\| f(x;\cdot) \|_{L^2}^2 \leq \frac{\pi}{1+x^2} \int_{-\infty}^x (1 + y^2) |w(y)|^2 dy
\leq \frac{\pi}{1+x^2} \| w \|_{L^{2,1}(-\infty,x)}^2,
\end{eqnarray*}
which yields the bound (\ref{Fourier-1}) for any fixed $x_0 \in \mathbb{R}^-$.

To get the bound (\ref{Fourier-1a}), we note that if $w \in H^1(\mathbb{R})$, then
$w \in L^{\infty}(\mathbb{R})$ and $w(x) \to 0$ as $|x| \to \infty$. As a result, we have
$$
2 i z f(x;z) + w(x) = \int_{- \infty}^x e^{2 i z (x-y)} \partial_y w(y) dy.
$$
The bound (\ref{Fourier-1a}) follows from the computation similar to (\ref{plancherel}).
\end{proof}

Subtracting the asymptotic limits (\ref{asymptotics-1-lim}) and (\ref{asymptotics-2-lim}) in Lemma \ref{asympt-mod}
from the Jost functions $m_{\pm}$ and $n_{\pm}$ in Lemma \ref{Jost}, we prove that for every fixed $x \in \mathbb{R}^{\pm}$,
the remainder terms belongs to $H^1(\mathbb{R})$ with respect to the variable $z$
if $u$ belongs to the space $H^{1,1}(\mathbb{R})$ defined in (\ref{H-1-1}).
Moreover, subtracting also the $\mathcal{O}(z^{-1})$ terms as defined by (\ref{asymptotics-1}) and (\ref{asymptotics-2})
and multiplying the result by $z$,
we prove that the remainder term belongs to $L^{2}(\mathbb{R})$ if $u \in H^2(\mathbb{R}) \cap H^{1,1}(\mathbb{R})$.
Note that if $u \in H^{1,1}(\mathbb{R})$, then the conditions of Lemma \ref{Jost} are satisfied,
so that $u \in L^1(\mathbb{R}) \cap L^3(\mathbb{R})$ and
$\partial_x u \in L^1(\mathbb{R})$. Also if $u \in H^2(\mathbb{R}) \cap H^{1,1}(\mathbb{R})$,
then the additional condition $u \in C^1(\mathbb{R})$ of Lemma \ref{asympt-mod} is also satisfied.

\begin{lem} \label{regularity-m}
If $u \in H^{1,1}(\mathbb{R})$, then for every $x \in \mathbb{R}^{\pm}$, we have
\begin{equation}
\label{m-n-H-1}
m_{\pm}(x;\cdot) - m_{\pm}^{\infty}(x) e_1 \in H^1(\mathbb{R}),  \quad
n_{\pm}(x;\cdot) - n_{\pm}^{\infty}(x) e_2 \in H^1(\mathbb{R}).
\end{equation}
Moreover, if $u \in H^2(\mathbb{R}) \cap H^{1,1}(\mathbb{R})$, then
for every $x \in \mathbb{R}$, we have
\begin{equation}
\label{m-n-L-2-1a}
z \left[ m_{\pm}(x;z) - m_{\pm}^{\infty}(x) e_1 \right] - (q_{\pm}^{(1)}(x) e_1 + q_{\pm}^{(2)}(x) e_2) \in L^{2}_z(\mathbb{R})
\end{equation}
and
\begin{equation}
\label{m-n-L-2-1b}
z \left[ n_{\pm}(x;z) - n_{\pm}^{\infty}(x) e_2 \right] - (s_{\pm}^{(1)}(x) e_1 + s_{\pm}^{(2)}(x) e_2) \in L^{2}_z(\mathbb{R}).
\end{equation}
\end{lem}

\begin{proof}
Again, we prove the statement for the Jost function $m_-$. The proof for other Jost functions is analogous.
We write the integral equation \eqref{int-jost-1} for $m_-$ in the abstract form
\begin{equation} \label{int-m}
m_- = e_1 + K m_-,
\end{equation}
where the operator $K$ is given by (\ref{kernel-Q-1}).
Although equation (\ref{int-m}) is convenient for verifying the boundary condition
$m_-(x;z) \to e_1$ as $x \to -\infty$, we note that the asymptotic limit
as $|z| \to \infty$ is different by the complex exponential factor. Indeed,
for every $x \in \mathbb{R}$, the asymptotic limit \eqref{asymptotics-1-lim} is written as
\begin{equation*}
m_-(x;z) \to m_-^{\infty}(x) e_1 \quad \mbox{\rm as} \quad |z| \to \infty,
\quad \mbox{\rm where} \quad m_-^{\infty}(x) := e^{\frac{1}{2i} \int_{-\infty}^x |u(y)|^2 dy}.
\end{equation*}
Therefore, we rewrite equation \eqref{int-m} in the equivalent form
\begin{equation} \label{int-m-2}
(I-K)(m_- - m_-^{\infty} e_1) = h e_2,
\end{equation}
where we have used the integral equation (\ref{int-jost-asymp-3}) that yields
$e_1 - (I - K) m_-^{\infty} e_1 = h e_2$ with
\begin{equation}
\label{def-h}
h(x;z) = \int_{-\infty}^x e^{2iz(x-y)} w(y) dy, \quad
w(x) := - \partial_x \left( \overline{u}(x) e^{\frac{1}{2i} \int_{-\infty}^x  |u(y)|^2 dy} \right).
\end{equation}
If $u \in H^{1,1}(\mathbb{R})$, then $w\in L^2(\mathbb{R})$. By the bounds (\ref{Fourier-1b}) and (\ref{Fourier-1}) in
Proposition \ref{prop-Fourier-1}, we have $h(x;z) \in L^{\infty}_x(\mathbb{R};L^2_z(\mathbb{R}))$
and for every $x_0 \in \mathbb{R}^-$, the following bound is satisfied:
\begin{eqnarray}
\nonumber
\sup_{x \in (-\infty,x_0)} \| \langle x \rangle \; h(x;z) \|_{L^2_z(\mathbb{R})} & \leq & \sqrt{\pi}
\left( \| \partial_x u \|_{L^{2,1}} + \frac{1}{2} \| u^3 \|_{L^{2,1}} \right) \\
\label{bound-on-h} & \leq & C (\| u \|_{H^{1,1}} + \| u \|_{H^{1,1}}^3),
\end{eqnarray}
where $C$ is a positive $u$-independent constant and the Sobolev inequality
$\| u \|_{L^{\infty}} \leq \frac{1}{\sqrt{2}} \| u \|_{H^1}$ is used.

By using estimates similar to those in the derivation of the bound (\ref{bound-on-K-kernel}) in Lemma \ref{Jost},
we find that for every $f(x;z) \in L^{\infty}_x(\mathbb{R};L^2_z(\mathbb{R}))$, we have
\begin{equation}
\label{bound-on-K-kernel-L2}
\| (K^n f)(x;z) \|_{L^{\infty}_x L^2_z}\leq \frac{1}{n!} \| Q_1(u) \|_{L^1}^n \|f(x;z)\|_{L^{\infty}_xL^2_z}.
\end{equation}
Therefore, the operator $I - K$ is invertible on the space $L^{\infty}_x(\mathbb{R}; L^2_z(\mathbb{R}))$
and a bound on the inverse operator is given by
\begin{equation}\label{I-K-estimate}
\|(I-K)^{-1}\|_{L^{\infty}_xL^2_z \to L^{\infty}_xL^2_z} \leq \sum_{n=0}^{\infty} \frac{1}{n!} \| Q_1(u) \|_{L^1}^n = e^{\| Q_1(u)\|_{L^1}}.
\end{equation}
Moreover, the same estimate (\ref{I-K-estimate}) can be obtained
in the norm $L^{\infty}_x((-\infty,x_0); L^2_z(\mathbb{R}))$ for every $x_0 \in \mathbb{R}$.
By using \eqref{int-m-2}, \eqref{bound-on-h}, and \eqref{I-K-estimate},
we obtain the following estimate for every $x_0 \in \mathbb{R}^-$:
\begin{equation} \label{estimate-m}
\sup_{x \in (-\infty,x_0)} \left\| \langle x \rangle \left( m_-(x;z) - m_-^{\infty}(x) e_1 \right) \right\|_{L^2_z(\mathbb{R})}
\leq C e^{\| Q_1(u)\|_{L^1}} \left( \|u\|_{H^{1,1}} + \|u\|_{H^{1,1}}^3 \right).
\end{equation}

Next, we want to show $\partial_z m_-(x;z) \in L^{\infty}_x((-\infty,x_0); L^2_z(\mathbb{R}))$ for
every $x_0 \in \mathbb{R}^-$. We differentiate the integral equation \eqref{int-m} in $z$
and introduce the vector $v = [v^{(1)},v^{(2)}]^t$ with the components
$$
v^{(1)}(x;z) := \partial_z m_-^{(1)}(x;z) \quad \mbox{\rm and} \quad
v^{(2)}(x;z) := \partial_z m_-^{(2)}(x;z) - 2 i x m_-^{(2)}(x;z).
$$
Thus, we obtain from (\ref{int-m}):
\begin{equation} \label{int-m-z}
(I-K) v = h_1 e_1 + h_2 e_2 + h_3 e_2,
\end{equation}
where
\begin{eqnarray*}
h_1(x;z) & = & \int_{-\infty}^x y u(y) m_-^{(2)}(y;z) dy, \\
h_2(x;z) & = & \int_{-\infty}^x y e^{2 i z (x-y)} (2i \overline{u}_y(y) + |u(y)|^2 \overline{u}(y)) (m_-^{(1)}(y;z)-m_-^{\infty}(y)) dy, \\
h_3(x;z) & = & \int_{-\infty}^x y e^{2 i z (x-y)} (2i \overline{u}_y(y) + |u(y)|^2 \overline{u}(y)) m_-^{\infty}(y) dy.
\end{eqnarray*}
For every $x_0 \in \mathbb{R}^-$, each inhomogeneous term of the integral equation (\ref{int-m-z}) can be estimated by
using H\"{o}lder's inequality and the bound (\ref{Fourier-1b}) of Proposition \ref{prop-Fourier-1}:
\begin{eqnarray*}
\sup_{x \in (-\infty,x_0)} \|h_1(x;z) \|_{L_z^2(\mathbb{R})} & \leq & \|u\|_{L^1}
\sup_{x \in (-\infty,x_0)} \| \langle x \rangle \; m_-^{(2)}(x;z)\|_{L_z^2(\mathbb{R})}, \\
\sup_{x \in (-\infty,x_0)} \|h_2(x;z) \|_{L_z^2(\mathbb{R})} & \leq & \left( 2\| \partial_x u\|_{L^1}
+ \| u^3 \|_{L^1} \right) \sup_{x \in (-\infty,x_0)}  \left\| \langle x \rangle
\left( m_-^{(1)}(x;z) - m_-^{\infty}(x)\right) \right\|_{L_z^2(\mathbb{R})},  \\
\sup_{x \in (-\infty,x_0)} \|h_3(x;z) \|_{L_z^2(\mathbb{R})} & \leq & \sqrt{\pi} \left( 2\| \partial_x u\|_{L^{2,1}}
+ \| u^3 \|_{L^{2,1}} \right).
\end{eqnarray*}
The upper bounds in the first two inequalities are finite due to estimate \eqref{estimate-m} and the embedding
of $L^{2,1}(\mathbb{R})$ into $L^1(\mathbb{R})$.
Using the bounds (\ref{I-K-estimate}), (\ref{estimate-m}), and the integral equation (\ref{int-m-z}), we conclude that
$v(x;z) \in L^{\infty}_x((-\infty,x_0); L^2_z(\mathbb{R}))$ for every $x_0 \in \mathbb{R}^-$.
Since $x m_-^{(2)}(x;z)$ is bounded in $L^{\infty}_x((-\infty,x_0); L^2_z(\mathbb{R}))$ by the same estimate \eqref{estimate-m},
we finally obtain $\partial_z m_-(x;z) \in L^{\infty}_x((-\infty,x_0); L^2_z(\mathbb{R}))$ for every $x_0 \in \mathbb{R}^-$.
This completes the proof of (\ref{m-n-H-1}) for $m_-$.

To prove (\ref{m-n-L-2-1a}) for $m_-$, we subtract the $\mathcal{O}(z^{-1})$ term as defined by (\ref{asymptotics-1})
from the integral equation (\ref{int-m-2}) and multiply the result by $z$. Thus, we obtain
\begin{equation} \label{int-m-3}
(I-K)\left[ z \left( m_- - m_-^{\infty} e_1 \right) -( q_-^{(1)} e_1 + q_-^{(2)} e_2) \right] =
z h e_2 - (I-K) ( q_-^{(1)} e_1 + q_-^{(2)} e_2),
\end{equation}
where the limiting values $q_-^{(1)}$ and $q_-^{(2)}$ are defined in Lemma \ref{asympt-mod}.
Using the integral equation (\ref{int-jost-asymp-6}), we obtain cancelation of the first component of the source term,
so that
$$
z h e_2 - (I-K) ( q_-^{(1)} e_1 + q_-^{(2)} e_2) = \tilde{h} e_2
$$
with
\begin{eqnarray*}
\tilde{h}(x;z) & = & z \int_{-\infty}^x e^{2iz(x-y)} w(y) dy + \frac{1}{2i} w(x)\\
& \phantom{t} & - \frac{1}{2i} \int_{-\infty}^x e^{2iz(x-y)} \left[ (2i \partial_y \bar{u}(y) + \bar{u}(y) |u(y)|^2 )
q_-^{(1)}(y) + |u(y)|^2 q_-^{(2)}(y) \right] dy,
\end{eqnarray*}
where $w$ is the same as in (\ref{def-h}). By using bounds (\ref{Fourier-1b}) and (\ref{Fourier-1a}) in Proposition \ref{prop-Fourier-1},
we have $\tilde{h}(x;z) \in L^{\infty}_x(\mathbb{R};L^2_z(\mathbb{Z}))$ if $w \in H^1(\mathbb{R})$ in
addition to $u \in H^{1,1}(\mathbb{R})$, that is, if $u \in H^2(\mathbb{R}) \cap H^{1,1}(\mathbb{R})$.
Inverting $(I-K)$ on $L^{\infty}_x(\mathbb{R};L^2_z(\mathbb{Z}))$,
we finally obtain (\ref{m-n-L-2-1a}) for $m_-$.
\end{proof}

The following result is deduced from Lemma \ref{regularity-m} to show that
the mapping
\begin{equation}
\label{map-u-Jost}
H^{1,1}(\mathbb{R}) \ni u \to [m_{\pm}(x;z)-m_{\pm}^{\infty}(x) e_1,
n_{\pm}(x;z) - n_{\pm}^{\infty}(x)] \in L_x^{\infty}(\mathbb{R}^{\pm};H^1_z(\mathbb{R}))
\end{equation}
is Lipschitz continuous. Moreover, by restricting the potential to $H^2(\mathbb{R}) \cap H^{1,1}(\mathbb{R})$,
subtracting $\mathcal{O}(z^{-1})$ terms from the Jost functions, and multiplying them by $z$, we also have Lipschitz continuity
of remainders of the Jost functions in function space $L_x^{\infty}(\mathbb{R};L^{2}_z(\mathbb{R}))$.

\begin{cor}
\label{cor-regularity-m}
Let $u, \tilde{u} \in H^{1,1}(\mathbb{R})$ satisfy $\| u \|_{H^{1,1}}, \| \tilde{u} \|_{H^{1,1}} \leq U$ for some $U > 0$.
Denote the corresponding Jost functions by $[m_{\pm},n_{\pm}]$ and $[\tilde{m}_{\pm},\tilde{n}_{\pm}]$ respectively.
Then, there is a positive $U$-dependent constant $C(U)$ such that
for every $x \in \mathbb{R}^{\pm}$, we have
\begin{equation}
\label{LC-m}
\| m_{\pm}(x;\cdot) - m_{\pm}^{\infty}(x) e_1
- \tilde{m}_{\pm}(x;\cdot) + \tilde{m}_{\pm}^{\infty}(x) e_1 \|_{H^1} \leq C(U) \| u - \tilde{u} \|_{H^{1,1}}
\end{equation}
and
\begin{equation}
\label{LC-n}
\| n_{\pm}(x;\cdot) - n_{\pm}^{\infty}(x) e_2
- \tilde{n}_{\pm}(x;\cdot) + \tilde{n}_{\pm}^{\infty}(x) e_2 \|_{H^1} \leq C(U) \| u - \tilde{u} \|_{H^{1,1}}.
\end{equation}
Moreover, if $u, \tilde{u} \in H^2(\mathbb{R}) \cap H^{1,1}(\mathbb{R})$
satisfy $\| u \|_{H^2 \cap H^{1,1}}, \| \tilde{u} \|_{H^2 \cap H^{1,1}} \leq U$,
then for every $x \in \mathbb{R}$, there is a positive $U$-dependent constant $C(U)$ such that
\begin{equation}
\label{LC-mn}
\| \hat{m}_{\pm}(x;\cdot) - \hat{\tilde{m}}_{\pm}(x;\cdot) \|_{L^{2}}
+ \| \hat{n}_{\pm}(x;\cdot) - \hat{\tilde{n}}_{\pm}(x;\cdot) \|_{L^{2}}
\leq C(U) \| u - \tilde{u} \|_{H^2 \cap H^{1,1}}.
\end{equation}
where
\begin{eqnarray*}
\hat{m}_{\pm}(x;z) & := & z \left[ m_{\pm}(x;z) - m_{\pm}^{\infty}(x) e_1 \right] - (q_{\pm}^{(1)}(x) e_1 + q_{\pm}^{(2)}(x) e_2), \\
\hat{n}_{\pm}(x;z) & := & z \left[ n_{\pm}(x;z) - n_{\pm}^{\infty}(x) e_2 \right] - (s_{\pm}^{(1)}(x) e_1 + s_{\pm}^{(2)}(x) e_2).
\end{eqnarray*}
\end{cor}

\begin{proof}
Again, we prove the statement for the Jost function $m_-$. The proof for other Jost functions is analogous.
First, let us consider the limiting values of $m_-$ and $\tilde{m}_-$ given by
$$
m_{-}^{\infty}(x) := e^{\frac{1}{2i} \int_{-\infty}^x |u(y)|^2 dy}, \quad
\tilde{m}_{-}^{\infty}(x) := e^{\frac{1}{2i} \int_{-\infty}^x |\tilde{u}(y)|^2 dy}
$$
Then, for every $x \in \mathbb{R}$, we have
\begin{eqnarray}
\nonumber
|m_{-}^{\infty}(x) - \tilde{m}_{-}^{\infty}(x)| & = & \left| e^{\frac{1}{2i} \int_{-\infty}^x (|u(y)|^2 - |\tilde{u}(y)|^2) dy} - 1\right| \\
\nonumber
& \leq & C_1(U) \int_{-\infty}^x (|u(y)|^2 - |\tilde{u}(y)|^2) dy \\
\label{LC-first}
& \leq & 2 U C_1(U) \| u - \tilde{u} \|_{L^2},
\end{eqnarray}
where $C_1(U)$ is a $U$-dependent positive constant. Using the integral equation \eqref{int-m-2}, we obtain
\begin{eqnarray}
\nonumber
(m_- - m_-^{\infty} e_1) - (\tilde{m}_- - \tilde{m}_-^{\infty}e_1) &=& (I-K)^{-1}he_2-(I-\tilde{K})^{-1}\tilde{h}e_2\\
\nonumber
&=&(I-K)^{-1}(h-\tilde{h})e_2+[(I-K)^{-1}-(I-\tilde{K})^{-1}]\tilde{h}e_2\\
\label{LC-third}
&=&(I-K)^{-1}(h-\tilde{h})e_2+(I-K)^{-1}(K-\tilde{K})(I-\tilde{K})^{-1}\tilde{h}e_2, \phantom{texttext}
\end{eqnarray}
where $\tilde{K}$ and $\tilde{h}$ denote the same as $K$ and $h$ but with $u$ being replaced by $\tilde{u}$.
To estimate the first term, we write
\begin{eqnarray}
\label{LC-second}
h(x;z) - \tilde{h}(x;z) = \int_{-\infty}^x e^{2iz(x-y)} \left[ w(y) - \tilde{w}(y) \right] dy,
\end{eqnarray}
where
\begin{eqnarray*}
w - \tilde{w} & = &  \left( \partial_x \bar{\tilde{u}} + \frac{1}{2i} |\tilde{u}|^2 \bar{\tilde{u}} \right) \tilde{m}_-^{\infty}
-\left( \partial_x \bar{u} + \frac{1}{2i} |u|^2 \bar{u} \right) m_-^{\infty}.
\end{eqnarray*}
By using (\ref{LC-first}), we obtain $\| w - \tilde{w} \|_{L^{2,1}} \leq C_2(U) \| u - \tilde{u} \|_{H^{1,1}}$,
where $C_2(U)$ is another $U$-dependent positive constant. By using (\ref{LC-second}) and Proposition \ref{prop-Fourier-1},
we obtain for every $x_0 \in \mathbb{R}^-$:
\begin{eqnarray}
\label{LC-fourth}
\sup_{x \in (-\infty,x_0)} \left\| \langle x \rangle \left( h(x;z) - \tilde{h}(x;z) \right) \right\|_{L^2_z(\mathbb{R})}
& \leq & \sqrt{\pi} C_2(U) \| u - \tilde{u} \|_{H^{1,1}}.
\end{eqnarray}
This gives the estimate for the first term in (\ref{LC-third}). To estimate the second term,
we use (\ref{kernel-Q-1}) and observe that $K$ is a Lipschitz continuous operator
from $L^{\infty}_x(\mathbb{R};L^2_z(\mathbb{R}))$ to  $L^{\infty}_x(\mathbb{R};L^2_z(\mathbb{R}))$ in the sense
that for every $f \in L^{\infty}_x(\mathbb{R};L^2_z(\mathbb{R}))$, we have
\begin{equation}
\label{LC-fifth}
\| (K - \tilde{K}) f \|_{L^{\infty}_x L^2_z} \leq
C_3(U) \| u - \tilde{u} \|_{H^{1,1}} \| f \|_{L^{\infty}_x L^2_z},
\end{equation}
where $C_3(U)$ is another $U$-dependent positive constant that is independent of $f$.
By using (\ref{bound-on-h}), (\ref{I-K-estimate}), (\ref{LC-third}), (\ref{LC-fourth}), and (\ref{LC-fifth}),
we obtain for every $x_0 \in \mathbb{R}^-$:
$$
\sup_{x \in (-\infty,x_0)}  \left\| \langle x \rangle \left( m_-(x;\cdot) - m_-^{\infty}(x) e_1
- \tilde{m}_-(x;\cdot) + \tilde{m}_-^{\infty}(x) e_1 \right) \right\|_{L^2_z(\mathbb{R})}
\leq  C(U) \|u-\tilde{u}\|_{H^{1,1}}.
$$
This yields the first part of the bound (\ref{LC-m}) for $m_-$ and $\tilde{m}_-$. The other part of the bound (\ref{LC-m})
and the bound (\ref{LC-mn}) for $m_-$ and $\tilde{m}_-$ follow by repeating the same analysis to
the integral equations (\ref{int-m-z}) and (\ref{int-m-3}).
\end{proof}

\subsection{Scattering coefficients}

Let us define the Jost functions of the original Kaup--Newell spectral problem (\ref{lax1}).
These Jost functions are related to the Jost functions of the Zakharov--Shabat spectral problems (\ref{lax-jost-1}) and (\ref{lax-jost-2})
by using the matrix transformations (\ref{T}). To be precise, we define
\begin{equation}
\label{original-Jost}
\varphi_{\pm}(x;\lambda) = T_1^{-1}(x;\lambda) m_{\pm}(x;z), \quad
\phi_{\pm}(x;\lambda) = T_2^{-1}(x;\lambda) n_{\pm}(x;z),
\end{equation}
where the inverse matrices are given by
\begin{equation}
\label{T-inverse}
T_1^{-1}(x;\lambda) = \frac{1}{2i\lambda} \left[\begin{matrix} 2 i \lambda & 0 \\ \overline{u}(x) & 1 \end{matrix}\right] \quad\mbox{and}
\quad T_2^{-1}(x;\lambda) = \frac{1}{2i\lambda} \left[\begin{matrix} 1 & u(x) \\ 0 & 2i\lambda  \end{matrix}\right].
\end{equation}
It follows from the integral equations (\ref{int-jost-1})--(\ref{int-jost-2}) and the transformation
(\ref{original-Jost}) that the original Jost functions $\varphi_{\pm}$ and $\phi_{\pm}$ satisfy the
following Volterra's integral equations
\begin{equation}\label{int-jost-1-orig}
\varphi_{\pm}(x;\lambda) = e_1 + \lambda \int_{\pm \infty}^{x} \left[\begin{matrix} 1 & 0 \\ 0 & e^{2i \lambda^2 (x-y)}
\end{matrix}\right] Q(u(y)) \varphi_{\pm}(y;\lambda) dy,
\end{equation}
and
\begin{equation} \label{int-jost-2-orig}
\phi_{\pm}(x;\lambda) = e_2 + \lambda \int_{\pm \infty}^{x}\left[\begin{matrix} e^{-2i \lambda^2 (x-y)} & 0 \\ 0 & 1
\end{matrix}\right] Q(u(y)) \phi_{\pm}(y;\lambda) dy.
\end{equation}
The following corollary is obtained from Lemma \ref{Jost}
and the representations (\ref{original-Jost})--(\ref{T-inverse}).

\begin{cor}
Let $u \in L^1(\mathbb{R}) \cap L^{\infty}(\mathbb{R})$ and
$\partial_x u \in L^1(\mathbb{R})$. For every $\lambda^2 \in \mathbb{R} \backslash \{0\}$,
there exist unique functions $\varphi_{\pm}(\cdot;\lambda) \in L^{\infty}(\mathbb{R})$ and
$\phi_{\pm}(\cdot;\lambda) \in L^{\infty}(\mathbb{R})$ such that
\begin{equation}\label{jost-infinity-original}
\left. \begin{array}{l}
\varphi_{\pm}(x;\lambda) \to e_1, \\
\phi_{\pm}(x;\lambda) \to e_2, \end{array} \right\}
\quad \mbox{\rm as} \quad x \to \pm \infty.
\end{equation}
Moreover, $\varphi_{\pm}^{(1)}(x;\lambda)$ and $\phi_{\pm}^{(2)}(x;\lambda)$ are
even in $\lambda$, whereas $\varphi_{\pm}^{(2)}(x;\lambda)$ and $\phi_{\pm}^{(1)}(x;\lambda)$ are
odd in $\lambda$.
\label{Jost-original}
\end{cor}

\begin{proof}
To the conditions of Lemma \ref{Jost}, we added the condition $u \in L^{\infty}(\mathbb{R})$,
which ensures that $T_{1,2}^{-1}(x;\lambda)$ are bounded for every $x \in \mathbb{R}$ and
for every $\lambda \in \mathbb{C} \backslash \{0\}$.
Then, the existence and uniqueness of the functions  $\varphi_{\pm}(\cdot;\lambda) \in L^{\infty}(\mathbb{R})$ and
$\phi_{\pm}(\cdot;\lambda) \in L^{\infty}(\mathbb{R})$, as well as the limits (\ref{jost-infinity-original})
follow by the representation (\ref{original-Jost})--(\ref{T-inverse}) and
by the first assertion of Lemma \ref{Jost}. The parity argument for components of
$\varphi_{\pm}(x;\lambda)$ and $\psi_{\pm}(x;\lambda)$ in $\lambda$ follow from the representation (\ref{original-Jost})--(\ref{T-inverse})
and the fact that $m_{\pm}(x;z)$ and $n_{\pm}(x;z)$ are even in $\lambda$ since $z = \lambda^2$.
\end{proof}

\begin{remark}\label{remark-zero}
There is no singularity in the definition of Jost functions at the value $\lambda = 0$.
The integral equations (\ref{int-jost-1-orig}) and (\ref{int-jost-2-orig})
with $\lambda = 0$ admit unique Jost functions
$\varphi_{\pm}(x;0) = e_1$ and $\phi_{\pm}(x;0) = e_2$,
which yield unique definitions for $m_{\pm}(x;0)$ and $n_{\pm}(x;0)$:
$$
m_{\pm}(x;0) = \left[ \begin{matrix} 1 \\ -\bar{u}(x) \end{matrix} \right], \quad
n_{\pm}(x;0) = \left[ \begin{matrix} -u(x) \\ 1 \end{matrix} \right],
$$
which follow from the unique solutions to the integral equations (\ref{int-jost-1}) and (\ref{int-jost-2}) at $z = 0$.
\end{remark}

\begin{remark}
The only purpose in the definition of the original Jost functions (\ref{original-Jost}) is
to introduce the standard form of the scattering relations, similar to the one used
in the literature \cite{KN78}. After introducing the scattering data for $\lambda \in \mathbb{R} \cup i \mathbb{R}$,
we analyze their behavior in the complex $z$-plane, instead of the complex $\lambda$-plane,
where $z = \lambda^2$.
\end{remark}

Analytic properties of the Jost functions $\varphi_{\pm}(x;\cdot)$ and $\psi_{\pm}(x;\cdot)$
for every $x \in \mathbb{R}$ are summarized in the following result. The result is a corollary
of Lemmas \ref{Jost} and \ref{regularity-m}.

\begin{cor}
Under the same assumption as Corollary \ref{Jost-original}, for every $x \in \mathbb{R}$, the Jost functions
$\varphi_{-}(x;\cdot)$ and $\phi_{+}(x;\cdot)$ are analytic in the first and third quadrant of the $\lambda$ plane
(where ${\rm Im}(\lambda^2) > 0$), whereas the Jost functions
$\varphi_{+}(x;\cdot)$ and $\phi_{-}(x;\cdot)$ are analytic in the second and fourth quadrant of the $\lambda$ plane
(where ${\rm Im}(\lambda^2) < 0$). Moreover, if $u \in H^{1,1}(\mathbb{R})$, then for every $x \in \mathbb{R}^{\pm}$,
we have
\begin{equation}
\label{varphi-H-1}
\varphi_{\pm}^{(1)}(x;\lambda) - m_{\pm}^{\infty}(x), \;\;
2 i \lambda \varphi_{\pm}^{(2)}(x;\lambda) - \bar{u}(x)m_{\pm}^{\infty}(x), \;\; \lambda^{-1} \varphi_{\pm}^{(2)}(x;\lambda)
\in H^1_z(\mathbb{R})
\end{equation}
and
\begin{equation}
\label{phi-H-1}
\lambda^{-1} \phi_{\pm}^{(1)}(x;\lambda), \;\;
2 i \lambda \phi_{\pm}^{(1)}(x;\lambda) - u(x) n_{\pm}^{\infty}(x), \;\;
\phi_{\pm}^{(2)}(x;\lambda) - n_{\pm}^{\infty}(x) \in H^1_z(\mathbb{R}),
\end{equation}
where $m_{\pm}^{\infty}$ and $n_{\pm}^{\infty}$ are the same as in Lemma \ref{asympt-mod}.
\label{lemma-Jost-original}
\end{cor}

\begin{proof}
By chain rule, we obtain
$$
\frac{\partial}{\partial \bar{\lambda}} = 2 \bar{\lambda} \frac{\partial}{\partial \bar{z}}.
$$
As a result, the analyticity result for the Jost functions $\varphi_{\pm}$ and $\phi_{\pm}$ follows
from the corresponding result of Lemma \ref{Jost}. With the transformation
(\ref{original-Jost})--(\ref{T-inverse}) and the result of Lemma \ref{regularity-m},
we obtain (\ref{varphi-H-1}) and (\ref{phi-H-1}) for $\varphi_{\pm}^{(1)}$, $\lambda \varphi_{\pm}^{(2)}$,
$\lambda \phi_{\pm}^{(1)}$, and $\phi_{\pm}^{(2)}$.

It remains to consider $\lambda^{-1} \varphi_{\pm}^{(2)}$ and $\lambda^{-1} \phi_{\pm}^{(1)}$. Although the result
also follows from Remark \ref{remark-zero}, we will give a direct proof.
We write explicitly from the integral equation (\ref{int-jost-1-orig}):
\begin{equation}
\label{varphi-int-1}
\lambda^{-1} \varphi_{\pm}^{(2)}(x;\lambda) =
- \int_{\pm \infty}^x e^{2 i z (x-y)} \overline{u}(y) m_{\pm}^{\infty}(y) dy
- \int_{\pm \infty}^x e^{2 i z (x-y)} \overline{u}(y) \left( m_{\pm}^{(1)}(y;z) -  m_{\pm}^{\infty}(y) \right) dy,
\end{equation}
where $m_{\pm}^{\infty} = e^{\frac{1}{2i} \int_{\pm \infty}^x |u(y)|^2 dy}$ and $z = \lambda^2$
as the same as in Lemma \ref{regularity-m}.
By using Proposition \ref{prop-Fourier-1} in the same way as it was used in the proof of Lemma \ref{regularity-m},
we obtain $\lambda^{-1} \varphi_{\pm}^{(2)}(x;\lambda) \in H^1_z(\mathbb{R})$ for every $x \in \mathbb{R}^{\pm}$.
The proof of $\lambda^{-1} \phi_{\pm}^{(1)}(x;\lambda) \in H^1_z(\mathbb{R})$ is similar.
\end{proof}

We note that $\psi(x) := \varphi_{\pm}(x;\lambda) e^{-i \lambda^2 x}$ and $\psi(x) := \phi_{\pm}(x;\lambda) e^{i \lambda^2 x}$
satisfies the Kaup--Newell spectral problem (\ref{lax1}), see asymptotic limits (\ref{potential-free}) and
(\ref{jost-infinity-original}). By the ODE theory for the second-order differential systems,
only two solutions are linearly independent. Therefore,
for every $x \in \mathbb{R}$ and every $\lambda^2 \in \mathbb{R} \backslash\{0\}$,
we define the scattering data according to the following transfer matrix
\begin{equation}\label{linear}
\left[\begin{matrix}  \varphi_-(x;\lambda) \\ \phi_-(x;\lambda) \end{matrix} \right]
= \left[\begin{matrix}  a(\lambda) & b(\lambda) e^{2 i \lambda^2 x} \\ c(\lambda) e^{-2 i \lambda^2 x} & d(\lambda) \end{matrix} \right]
\left[\begin{matrix}  \varphi_+(x;\lambda) \\ \phi_+(x;\lambda) \end{matrix} \right].
\end{equation}
By Remark \ref{remark-zero}, the transfer matrix is extended to $\lambda = 0$
with $a(0) = d(0) = 1$ and $b(0) = c(0) = 0$.

Since the coefficient matrix in the Kaup--Newell spectral problem \eqref{lax1} has zero trace,
the Wronskian determinant, denoted by $W$, of two solutions
to the differential system (\ref{lax1}) for any $\lambda \in \mathbb{C}$ is independent of $x$.
As a result, we verify that the scattering coefficients $a$, $b$, $c$, and $d$ are independent of $x$:
\begin{eqnarray}
\label{W-a-b-1}
a(\lambda) & = & W(\varphi_-(x;\lambda) e^{-i \lambda^2 x},\phi_+(x;\lambda) e^{+i \lambda^2 x}) = W(\varphi_-(0;\lambda),\phi_+(0;\lambda)), \\
\label{W-a-b-2}
b(\lambda) & = & W(\varphi_+(x;\lambda) e^{-i \lambda^2 x},\varphi_-(x;\lambda) e^{-i \lambda^2 x}) = W(\varphi_+(0;\lambda),\varphi_-(0;\lambda)),
\end{eqnarray}
where we have used the Wronskian relation $W(\varphi_+,\phi_+) = 1$, which follows
from the boundary conditions (\ref{jost-infinity-original}) as $x \to +\infty$.

Now we note the symmetry on solutions to the linear equation (\ref{lax1}).
If $\psi$ is a solution for any $\lambda \in \mathbb{C}$,
then $\sigma_1 \sigma_3 \overline{\psi}$ is also a solution for $\bar{\lambda} \in \mathbb{C}$,
where $\sigma_1$ and $\sigma_3$ are Pauli matrices in (\ref{Pauli}).
As a result, using the boundary conditions for the normalized Jost functions, we obtain the following relations:
$$
\phi_{\pm}(x;\lambda) = \sigma_1 \sigma_3 \overline{\varphi_{\pm}(x;\overline{\lambda})},
$$
where $\overline{\varphi_{\pm}(x;\overline{\lambda})}$ means that we take complex conjugation of
$\varphi_{\pm}$ constructed from the system of integral equations (\ref{int-jost-1-orig}) for $\bar{\lambda}$.
By applying complex conjugation to the first equation in system (\ref{linear}) for $\bar{\lambda}$,
multiplying it by $\sigma_1 \sigma_3$, and using the relations
$\sigma_1 \sigma_3 = - \sigma_3 \sigma_1$ and $\sigma_1^2 = \sigma_3^2 = 1$,
we obtain the second equation in system (\ref{linear}) with the correspondence
\begin{eqnarray}
\label{W-a-b-3}
c(\lambda)=-\overline{b(\overline{\lambda})},\quad d(\lambda)=\overline{a(\overline{\lambda})}, \quad \lambda \in \mathbb{R} \cup i \mathbb{R}.
\end{eqnarray}

From the Wronskian relation $W(\varphi_-,\phi_-) = 1$, which can be established
from the boundary conditions (\ref{jost-infinity-original}) as $x \to -\infty$, we verify that the
transfer matrix in system (\ref{linear}) has the determinant equals to unity.
In view  of the correspondence (\ref{W-a-b-3}), this yields the result
\begin{eqnarray}
\label{W-a-b-4}
a(\lambda)\overline{a(\overline{\lambda})} + b(\lambda)\overline{b(\overline{\lambda})} = 1, \quad \lambda \in \mathbb{R} \cup i \mathbb{R}.
\end{eqnarray}

We now study properties of the scattering coefficients $a$ and $b$ in suitable function spaces.
We prove that
\begin{equation}
\label{asymptotics-a}
a(\lambda) \to a_{\infty} := e^{\frac{1}{2i} \int_{\mathbb{R}} |u|^2 dx} \quad \mbox{\rm as} \quad |\lambda| \to \infty,
\end{equation}
whereas $a(\lambda) - a_{\infty}$, $\lambda b(\lambda)$, and $\lambda^{-1} b(\lambda)$
are $H^1_z(\mathbb{R})$ functions with respect to $z$ if $u$ belongs to $H^{1,1}(\mathbb{R})$ defined
in (\ref{H-1-1}). Moreover, we show that $\lambda b(\lambda)$ is also in $L^{2,1}_z(\mathbb{R})$ if
$u \in H^2(\mathbb{R}) \cap H^{1,1}(\mathbb{R})$.

\begin{lem}\label{regularity-a-b}
If $u\in H^{1,1}(\mathbb{R})$, then the functions $a(\lambda)$ and $\overline{a(\overline{\lambda})}$
are continued analytically in $\mathbb{C}^+$ and $\mathbb{C}^-$ with respect to $z$, and, in addition,
\begin{equation}
\label{scat-data}
a(\lambda) - a_{\infty}, \; \lambda b(\lambda), \; \lambda^{-1} b(\lambda) \in H^1_z(\mathbb{R}),
\end{equation}
where $a_{\infty} := e^{\frac{1}{2i} \int_{\mathbb{R}} |u|^2 dx}$.
Moreover, if $u\in H^2(\mathbb{R}) \cap H^{1,1}(\mathbb{R})$, then
\begin{equation}
\label{scat-data-H-2}
\lambda b(\lambda), \; \lambda^{-1} b(\lambda) \in L^{2,1}_z(\mathbb{R}).
\end{equation}
\end{lem}

\begin{proof}
We consider the integral equations (\ref{int-jost-1-orig}) and (\ref{int-jost-2-orig}).
By taking the limit $x \to +\infty$, which is justified due to Corollary \ref{Jost-original}
and Remark \ref{remark-zero} for every
$\lambda \in \mathbb{R} \cup i \mathbb{R}$, and using the scattering relation (\ref{linear}) and
the transformation (\ref{original-Jost})--(\ref{T-inverse}), we obtain
\begin{eqnarray}
\label{int-a-b-1}
a(\lambda) = 1 + \lambda \int_{\mathbb{R}} u(x) \varphi_-^{(2)}(x;\lambda) dx
\end{eqnarray}
and
\begin{eqnarray}
\label{int-a-b-2}
\overline{a(\overline{\lambda})} = 1 - \lambda \int_{\mathbb{R}} \overline{u}(x) \phi_-^{(1)}(x;\lambda) dx.
\end{eqnarray}
It follows from the representations (\ref{int-a-b-1}) and (\ref{int-a-b-2}), as well as Corollary \ref{lemma-Jost-original},
that $a(\lambda)$ is continued analytically in $\mathbb{C}^+$ with respect to $z$, whereas
$\overline{a(\overline{\lambda})}$ is continued analytically in $\mathbb{C}^-$ with respect to $z$.
Using limits (\ref{asymptotics-1-lim}) in Lemma \ref{asympt-mod} and transformation (\ref{T-inverse}),
we obtain the following limit for
the scattering coefficient $a(\lambda)$ as $|{\rm Im}(z)| \to \infty$ along a contour in $\mathbb{C}^+$:
$$
\lim_{|z| \to \infty} a(\lambda) = 1 + \frac{1}{2i} \int_{\mathbb{R}} |u(x)|^2 e^{\frac{1}{2i} \int_{-\infty}^x |u(y)|^2 dy} dx =
e^{\frac{1}{2i} \int_{\mathbb{R}} |u(x)|^2 dx} =: a_{\infty}.
$$

In order to prove that $a(\lambda) - a_{\infty}$ is a $H^1_z(\mathbb{R})$ function,
we use the Wronskian representation (\ref{W-a-b-1}).
Recall from the transformation (\ref{original-Jost})--(\ref{T-inverse}) that
$$
\varphi_{\pm}^{(1)}(x;\lambda) = m_{\pm}^{(1)}(x;z) \quad \mbox{\rm and} \quad
\phi_{\pm}^{(2)}(x;\lambda) = n_{\pm}^{(2)}(x;z).
$$
Subtracting the limiting values for $a$ and the normalized Jost functions $m_{\pm}$ and $n_{\pm}$,
we rewrite the Wronskian representation (\ref{W-a-b-1}) explicitly
\begin{eqnarray} \nonumber
a(\lambda) - a_{\infty} & = & (m_-^{(1)}(0;z)- m_-^{\infty}(0)) (n_+^{(2)}(0;z)-n_+^{\infty}(0)) +
m_-^{\infty}(0) (n_+^{(2)}(0;z)-n_+^{\infty}(0)) \\
& \phantom{t} & + n_+^{\infty}(0) (m_-^{(1)}(0;z)-m_-^{\infty}(0))
- \varphi_-^{(2)}(0;\lambda) \phi_+^{(1)}(0;\lambda).
\label{a-det}
\end{eqnarray}
By (\ref{m-n-H-1}) in Lemma \ref{regularity-m}, all but the last term in (\ref{a-det}) belong to $H^1_z(\mathbb{R})$.
Furthermore, $\lambda^{-1} \varphi_{\pm}^{(2)}(0;\lambda)$ and
$2 i \lambda \phi_{\pm}^{(1)}(0;\lambda) - u(0) n_{\pm}^{\infty}(0)$ also belong to $H^1_z(\mathbb{R})$
by Corollary \ref{lemma-Jost-original}. Using the representation (\ref{a-det}) and the Banach algebra property of $H^1(\mathbb{R})$,
we conclude that $a(\lambda) - a_{\infty} \in H^1_z(\mathbb{R})$.

Next, we analyze the scattering coefficient $b$. By using the representation (\ref{original-Jost})--(\ref{T-inverse})
and the Wronskian representation (\ref{W-a-b-2}), we write
\begin{equation}\label{b-det}
2 i \lambda b(\lambda) = m_+^{(1)}(0;z) m_-^{(2)}(0;z) - m_+^{(2)}(0;z) m_-^{(1)}(0;z).
\end{equation}
By (\ref{m-n-H-1}) in Lemma \ref{regularity-m} (after the corresponding limiting values are subtracted from
$m_{\pm}^{(1)}(0;z)$), we establish that $\lambda b(\lambda) \in H^1_z(\mathbb{R})$.
On the other hand, the same Wronskian representation (\ref{W-a-b-2}) can also be
written in the form
\begin{equation}
\label{bb}
\lambda^{-1} b(\lambda) = m_+^{(1)}(0;z) \lambda^{-1} \varphi_-^{(2)}(0;\lambda) -
m_-^{(1)}(0;z) \lambda^{-1} \varphi_+^{(2)}(0;\lambda).
\end{equation}
Recalling that $\lambda^{-1} \varphi_{\pm}^{(2)}(0;\lambda)$ belongs to $H^1_z(\mathbb{R})$ by Corollary
\ref{lemma-Jost-original}, we obtain $\lambda^{-1} b(\lambda) \in H^1_z(\mathbb{R})$. The first assertion (\ref{scat-data})
of the lemma is proved.

To prove the second assertion (\ref{scat-data-H-2}) of the lemma,
we note that $\lambda^{-1} b(\lambda) \in L^{2,1}_z(\mathbb{R})$ because
$z \lambda^{-1} b(\lambda) = \lambda b(\lambda) \in H^1_z(\mathbb{R})$.
On the other hand, to show that $\lambda b(\lambda)\in L^{2,1}_z(\mathbb{R})$, we multiply equation \eqref{b-det} by $z$
and write the resulting equation in the form
\begin{eqnarray}
\nonumber
2i \lambda z b(\lambda) & = & m_+^{(1)}(0;z) \left( z m_-^{(2)}(0;z) - q_-^{(2)}(0) \right)
-  m_-^{(1)}(0;z) \left( z m_+^{(2)}(0;z) - q_+^{(2)}(0) \right) \\ \label{bbb}
& \phantom{t} & + q_-^{(2)}(0) \left( m_+^{(1)}(0;z) - m_+^{\infty}(0) \right)  - q_+^{(2)}(0) \left(m_-^{(1)}(0;z) - m_-^{\infty}(0) \right),
\end{eqnarray}
where we have used the identity $q_-^{(2)}(0) m_+^{\infty}(0)- q_+^{(2)}(0) m_-^{\infty}(0) = 0$
that follows from limits (\ref{asymptotics-1-lim}) and (\ref{asymptotics-1}).
By (\ref{m-n-H-1}) and (\ref{m-n-L-2-1a}) in Lemma \ref{regularity-m},
all the terms in the representation (\ref{bbb}) are in $L^2_z(\mathbb{R})$,
hence $\lambda b(\lambda) \in L^{2,1}_z(\mathbb{R})$. The second assertion (\ref{scat-data-H-2})
of the lemma is proved.
\end{proof}

We show that the mapping
\begin{equation}
\label{map-u-a-b}
H^{1,1}(\mathbb{R}) \ni u \to a(\lambda)- a_{\infty}, \lambda b(\lambda), \lambda^{-1} b(\lambda) \in H^1_z(\mathbb{R})
\end{equation}
is Lipschitz continuous. Moreover, we also have Lipschitz continuity of the mapping
\begin{equation}
\label{map-u-a-b-weighted}
H^2(\mathbb{R}) \cap H^{1,1}(\mathbb{R}) \ni u \to \lambda b(\lambda), \lambda^{-1} b(\lambda) \in L^{2,1}_z(\mathbb{R}).
\end{equation}
The corresponding result is deduced from Lemma \ref{regularity-a-b} and Corollary \ref{cor-regularity-m}.

\begin{cor}
\label{cor-regularity-a-b}
Let $u, \tilde{u} \in H^{1,1}(\mathbb{R})$ satisfy $\| u \|_{H^{1,1}}, \| \tilde{u} \|_{H^{1,1}} \leq U$ for some $U > 0$.
Denote the corresponding scattering coefficients by $(a,b)$ and $(\tilde{a},\tilde{b})$ respectively.
Then, there is a positive $U$-dependent constant $C(U)$ such that
\begin{equation}
\label{LC-a}
\| a(\lambda) - a_{\infty} - \tilde{a}(\lambda) + \tilde{a}_{\infty} \|_{H^1_z} +
\| \lambda b(\lambda) - \lambda \tilde{b}(\lambda) \|_{H^1_z} +
\| \lambda^{-1} b(\lambda) - \lambda^{-1} \tilde{b}(\lambda) \|_{H^1_z} \leq C(U) \| u - \tilde{u} \|_{H^{1,1}}.
\end{equation}
Moreover, if $u, \tilde{u} \in H^2(\mathbb{R}) \cap H^{1,1}(\mathbb{R})$
satisfy $\| u \|_{H^2 \cap H^{1,1}}, \| \tilde{u} \|_{H^2 \cap H^{1,1}} \leq U$,
then there is a positive $U$-dependent constant $C(U)$ such that
\begin{equation}
\label{LC-b}
\| \lambda b(\lambda) - \lambda \tilde{b}(\lambda) \|_{L^{2,1}_z} +
\| \lambda^{-1} b(\lambda) - \lambda^{-1} \tilde{b}(\lambda) \|_{L^{2,1}_z} \leq C(U) \| u - \tilde{u} \|_{H^2 \cap H^{1,1}}.
\end{equation}
\end{cor}

\begin{proof}
The assertion follows from the representations (\ref{a-det}),
(\ref{b-det}), (\ref{bb}), and (\ref{bbb}), as well as the Lipschitz continuity of the Jost functions
$m_{\pm}$ and $n_{\pm}$ established in Corollary \ref{cor-regularity-m}.
\end{proof}

\begin{remark}
Since Corollary \ref{cor-regularity-a-b} yields Lipschitz continuity of the mappings (\ref{map-u-a-b}) 
and (\ref{map-u-a-b-weighted}) for every $u$, $\tilde{u}$ in a ball of a fixed (but possibly large) radius $U$, 
the mappings (\ref{map-u-a-b}) and (\ref{map-u-a-b-weighted}) are one-to-one for every $u$ in the ball.  \label{remark-direct}
\end{remark}

Another result, which follows from Lemma \ref{regularity-a-b}, is the parity property of the scattering coefficients $a$
and $b$ with respect to $\lambda$. The corresponding result is given in the following corollary.

\begin{cor}
\label{corollary-a-b}
The scattering coefficients $a$ and $b$ are even and odd functions in $\lambda$ for $\lambda \in \mathbb{R} \cup i \mathbb{R}$.
Moreover, they satisfy the following scattering relation
\begin{equation}
\label{scattering-relation-modified}
\left\{ \begin{array}{l}
|a(\lambda)|^2 + |b(\lambda)|^2 = 1, \quad \lambda \in \mathbb{R}, \\
|a(\lambda)|^2 - |b(\lambda)|^2 = 1, \quad \lambda \in i\mathbb{R}.
\end{array} \right.
\end{equation}
\end{cor}

\begin{proof}
Because $a(\lambda)$ and $\lambda^{-1} b(\lambda)$ are functions of $z = \lambda^2$, as follows from
Lemma \ref{regularity-a-b}, we have $a(-\lambda) = a(\lambda)$ and $b(-\lambda) = -b(\lambda)$
for all $\lambda \in \mathbb{R} \cup i\mathbb{R}$. For $\lambda \in \mathbb{R}$, the
scattering relation (\ref{W-a-b-4}) yields the first line of
(\ref{scattering-relation-modified}). For $\lambda = i \gamma$ with $\gamma \in \mathbb{R}$,
the parity properties of $a$ and $b$ imply
$$
\overline{a(\bar{\lambda})} = \overline{a(-i\gamma)} = \overline{a(i\gamma)} = \overline{a(\lambda)} \quad
\mbox{\rm and} \quad
\overline{b(\overline{\lambda})} = \overline{b(-i\gamma)} = -\overline{b(i\gamma)} = -\overline{b(\lambda)}.
$$
Substituting these relations to the scattering relation (\ref{W-a-b-4}), we obtain the second line of (\ref{scattering-relation-modified})
\end{proof}

\section{Formulations of the Riemann--Hilbert problem}

We deduce the Riemann--Hilbert problem of complex analysis from the jump condition for normalized Jost functions
on $\mathbb{R} \cup i \mathbb{R}$ in the $\lambda$ plane, which corresponds to $\mathbb{R}$ in the $z$ plane,
where $z = \lambda^2$. The jump condition yields boundary conditions for the Jost functions
extended to sectionally analytic functions in different domains of the corresponding complex plane.
In the beginning, we derive the jump condition in the $\lambda$ plane by using the Jost functions of the original
Kaup--Newell spectral problem (\ref{lax1}).

Let us define the reflection coefficient by
\begin{equation}
r(\lambda) := \frac{b(\lambda)}{a(\lambda)}, \quad \lambda \in \mathbb{R} \cup i \mathbb{R}.
\end{equation}
Each zero of $a$ on $\mathbb{R} \cup i \mathbb{R}$ corresponds to the resonance,
according to Definition \ref{definition-resonance}. By the assumptions of Theorem \ref{main},
the spectral problem (\ref{lax1}) admits no resonances, therefore, there exists a positive number $A$ 
such that 
\begin{equation}
\label{constraint-on-a}
|a(\lambda)| \geq A > 0, \quad \lambda \in \mathbb{R} \cup i \mathbb{R}.
\end{equation}
Thus, $r(\lambda)$ is well-defined for every $\lambda \in \mathbb{R} \cup i \mathbb{R}$.

Under the condition (\ref{constraint-on-a}),
the scattering relations (\ref{linear}) with (\ref{W-a-b-3}) can be rewritten in the equivalent form:
\begin{equation}
\label{linear-1}
\frac{\varphi_-(x;\lambda)}{a(\lambda)} - \varphi_+(x;\lambda) = r(\lambda) e^{2 i \lambda^2 x}\phi_+(x;\lambda)
\end{equation}
and
\begin{equation}
\label{linear-2}
\frac{\phi_-(x;\lambda)}{\overline{a(\bar{\lambda})}} - \phi_+(x;\lambda) =
-\overline{r(\bar{\lambda})} e^{-2 i \lambda^2 x} \varphi_+(x;\lambda),
\end{equation}
where $\lambda \in \mathbb{R} \cup i \mathbb{R}$.

By Lemma \ref{regularity-a-b}, $a(\lambda)$
is continued analytically in the first and third quadrants of the $\lambda$ plane,
where ${\rm Im}(\lambda^2) >0$. Also $a(\lambda)$ approaches to a finite limit
$a_{\infty} \neq 0$ as $|\lambda| \to \infty$. By a theorem of complex analysis on zeros
of analytic functions, $a$ has at most finite number of zeros in
each quadrant of the $\lambda$ plane. Each zero of $a$ corresponds to an eigenvalue
of the spectral problem (\ref{lax1}) with the $L^2(\mathbb{R})$ solution $\psi(x)$ decaying to
zero exponentially fast as $|x| \to \infty$. Indeed, this follows from the Wronskian relation (\ref{W-a-b-1})
between the Jost functions $\varphi_-$ and $\psi_+$ extended to the first and third quadrant of
the $\lambda$ plane by Corollary \ref{lemma-Jost-original}. By the assumptions of Theorem \ref{main},
the spectral problem (\ref{lax1}) admits no eigenvalues, hence the bound (\ref{constraint-on-a}) 
is extended to the first and third quadrants of the $\lambda$ plane. Therefore, the functions
$\frac{\varphi_-(x;\lambda)}{a(\lambda)}$ and $\frac{\phi_-(x;\lambda)}{\overline{a(\bar{\lambda})}}$
are analytic in the corresponding domains of the $\lambda$ plane.

From the scattering relations (\ref{linear-1}) and (\ref{linear-2}), we can define the complex functions
\begin{equation} \label{lambda-jump}
\Phi_+(x;\lambda) := \left[\frac{\varphi_-(x;\lambda)}{a(\lambda)}, \phi_+(x;\lambda) \right],
\quad
\Phi_-(x;\lambda) := \left[ \varphi_+(x;\lambda), \frac{\phi_-(x;\lambda)}{\overline{a(\bar{\lambda})}}\right].
\end{equation}
By Corollary \ref{lemma-Jost-original}, Lemma \ref{regularity-a-b}, and the condition (\ref{constraint-on-a})
on $a$, for every $x \in \mathbb{R}$, the function $\Phi_+(x;\cdot)$ 
is analytic in the first and third quadrants of the $\lambda$ plane, whereas
the function $\Phi_-(x;\cdot)$ is analytic in the second and fourth quadrants of the $\lambda$ plane.
For every $x \in \mathbb{R}$ and $\lambda \in \mathbb{R} \cup i \mathbb{R}$,
the two functions are related by the jump condition
\begin{equation} \label{S}
\Phi_+(x;\lambda) - \Phi_-(x;\lambda) = \Phi_-(x;\lambda) S(x;\lambda),
\end{equation}
where
\begin{equation}
\label{S1}
S(x;\lambda) := \left[\begin{matrix} |r(\lambda)|^2 &
\overline{r(\lambda)} e^{-2i\lambda^2x} \\ r(\lambda)e^{2i\lambda^2x} &0 \end{matrix}\right], \quad \lambda \in \mathbb{R}
\end{equation}
and
\begin{equation}
\label{S2}
S(x;\lambda) := \left[\begin{matrix} -|r(\lambda)|^2 &
-\overline{r(\lambda)} e^{-2i\lambda^2x} \\ r(\lambda)e^{2i\lambda^2x} &0 \end{matrix}\right], \quad \lambda \in i \mathbb{R}.
\end{equation}
Note that $r(-\lambda) = -r(\lambda)$ by Corollary \ref{corollary-a-b}, so that $r(0) = 0$. By Corollary \ref{lemma-Jost-original},
the functions $\Phi_{\pm}(x;\lambda)$ satisfy the limiting behavior as $|\lambda| \to \infty$
along a contour in the corresponding domains of their analyticity in the $\lambda$ plane:
\begin{equation}
\label{jump-bc}
\Phi_{\pm}(x;\lambda) \to \Phi_{\infty}(x) :=
\left[e^{\frac{1}{2i} \int_{+\infty}^x |u(y)|^2 dy} e_1, \quad e^{-\frac{1}{2i}\int_{+\infty}^x |u(y)|^2 dy} e_2\right]
\quad \mbox{\rm as} \quad  |\lambda| \to \infty.
\end{equation}

The jump conditions (\ref{S}) and the boundary conditions (\ref{jump-bc}) set up a
Riemann--Hilbert problem to find sectionally analytic functions $\Phi(x;\cdot)$ for every $x \in \mathbb{R}$.
It is quite remarkable that the matrix $S$ is Hermitian for $\lambda \in \mathbb{R}$. In this case, we
can use the theory of Zhou \cite{ZhouSIMA} to obtain a unique solution to the
Riemann--Hilbert problem (\ref{S}), (\ref{S1}), and (\ref{jump-bc}). However,
the matrix $S$ is not Hermitian for $\lambda \in i \mathbb{R}$. Nevertheless,
the second scattering relation \eqref{scattering-relation-modified} yields a useful constraint:
\begin{equation}
\label{smallness-defocusing}
1 - |r(\lambda)|^2 = \frac{1}{|a(\lambda)|^2} \geq c_0^2 > 0, \quad \lambda \in i \mathbb{R},
\end{equation}
where $c_0 := \sup_{\lambda \in i \mathbb{R}} |a(\lambda)|$. The constraint (\ref{smallness-defocusing})
will be used to obtain a unique solution to the
Riemann--Hilbert problem (\ref{S}), (\ref{S2}), and (\ref{jump-bc}).

We note that only the latter case (\ref{S2}), which is relevant to the imaginary values of $\lambda$, was considered
in the context of the Kaup--Newell spectral problem by Kitaev \& Vartanian \cite{V1}, who studied the long time asymptotic solution
of the derivative NLS equation (\ref{dnls}, also in the case of no solitons. The smallness condition (\ref{smallness-defocusing})
does not need to be assumed a priori, as it is done in Lemma 2.2 in \cite{V1}, but appears naturally from the second
scattering relation (\ref{scattering-relation-modified}). The Hermitian case of real values of $\lambda$ was missed in \cite{V1}.

We also note that the scattering matrix $S(x;\lambda)$ is analogous to the one known for the focusing NLS equation
if $\lambda \in \mathbb{R}$ and the one known for the defocusing NLS equation if $\lambda \in i \mathbb{R}$.
As a result, the inverse scattering transform
for the derivative NLS equation combines elements of the inverse scattering transforms developed
for the focusing and defocusing cubic NLS equations \cite{D-J,D-Z-1,Z}.

In the rest of this section, we reformulate the jump condition in the $z$ plane and introduce
two scattering coefficients $r_{\pm}$, which are defined on the real line in the function
space $H^1(\mathbb{R}) \cap L^{2,1}(\mathbb{R})$.  The scattering coefficients $r_{\pm}$ allow us to recover
a potential $u$ in the function space $H^2(\mathbb{R}) \cap H^{1,1}(\mathbb{R})$ (in Section 4).

\subsection{Reformulation of the Riemann--Hilbert problem (\ref{S})}

Using transformation matrices in (\ref{original-Jost})--(\ref{T-inverse}),
we can rewrite the scattering relations (\ref{linear-1}) and (\ref{linear-2})
in terms of the $z$-dependent Jost functions $m_{\pm}$ and $n_{\pm}$:
\begin{equation}\label{linear-3}
\frac{m_-(x;z)}{a(\lambda)} - m_+(x;z) = \frac{2 i \lambda b(\lambda)}{a(\lambda)} e^{2 i z x} p_+(x;z)
\end{equation}
and
\begin{equation}
\label{linear-4}
\frac{p_-(x;z)}{\overline{a(\bar{\lambda})}} - p_+(x;z) =
-\frac{\overline{b(\bar{\lambda})}}{2 i \overline{\lambda} \overline{a(\bar{\lambda})}}
e^{-2 i z x} m_+(x;z),
\end{equation}
where $z \in \mathbb{R}$, $m_{\pm}$ are defined by Lemma \ref{Jost}, and $p_{\pm}$ are given explicitly by
\begin{equation}
\label{f}
p_{\pm}(x;z) = \frac{1}{2 i \lambda} T_1(x;\lambda) T_2^{-1}(x;\lambda) n_{\pm}(x;z) = -\frac{1}{4z}
\left[ \begin{matrix} 1 & u(x) \\ -\bar{u}(x) & -|u(x)|^2 - 4 z \end{matrix} \right] n_{\pm}(x;z).
\end{equation}
Properties of the new functions $p_{\pm}$ are summarized in the following result.

\begin{lem}
\label{lemma-p}
Under the conditions of Lemma \ref{Jost}, for every $x \in \mathbb{R}$,
the functions $p_{\pm}(x;z)$ are continued analytically in $\mathbb{C}^{\pm}$ and satisfy
the following limits as $|{\rm Im}(z)| \to \infty$ along a contour in the domains of their analyticity:
\begin{equation}
\label{asymptotic-p}
\lim_{|z| \to \infty} p_{\pm}(x;z) = n_{\pm}^{\infty}(x) e_2,
\end{equation}
where $n_{\pm}^{\infty}$ are the same as in the limits (\ref{asymptotics-2-lim}).
\end{lem}

\begin{proof}
The asymptotic limits (\ref{asymptotic-p}) follow from the representation
(\ref{f}) and the asymptotic limits (\ref{asymptotics-2-lim}) for $n_{\pm}(x;z)$
as $|z| \to \infty$ in Lemma \ref{asympt-mod}. Using the transformation
(\ref{original-Jost})--(\ref{T-inverse}), functions $p_{\pm}$ can be written in the
equivalent form
\begin{equation}
\label{p-form}
p_{\pm}(x;z) = n_{\pm}^{(2)}(x;z) e_2 + \frac{1}{2 i \lambda}
\left[ \begin{matrix} 1 \\ -\bar{u}(x) \end{matrix} \right] \phi_{\pm}^{(1)}(x;\lambda),
\end{equation}
where both $n_{\pm}^{(2)}(x;z)$ and $\lambda^{-1} \phi_{\pm}^{(1)}(x;\lambda)$ are continued analytically in $\mathbb{C}^{\pm}$
with respect to $z$ by Lemma \ref{Jost} and Corollary \ref{lemma-Jost-original}.
From the Volterra integral equation (\ref{int-jost-2-orig}), we also obtain
\begin{equation}
\label{varphi-3}
\lambda^{-1} \phi_{\pm}^{(1)}(x;\lambda) = \int_{\pm \infty}^{x} e^{-2i z (x-y)} u(y) n_{\pm}^{(2)}(y;z) dy,
\end{equation}
therefore, $p_{\pm}(x;0)$ exists for every $x \in \mathbb{R}$.
Thus, for every $x \in \mathbb{R}$, the analyticity properties of $p_{\pm}(x;\cdot)$
are the same as those of $n_{\pm}(x;\cdot)$.
\end{proof}

Let us now introduce the new scattering data:
\begin{equation}
\label{r-definition}
r_+(z) := -\frac{b(\lambda)}{2 i \lambda a(\lambda)}, \quad
r_-(z) := \frac{2i \lambda b(\lambda)}{a(\lambda)}, \quad z \in \mathbb{R}.
\end{equation}
which satisfy the relation
\begin{equation} \label{r1-r2}
r_-(z) = 4 z r_+(z), \quad z \in \mathbb{R}.
\end{equation}
It is worthwhile noting that
\begin{equation} \label{r1-r2-other}
\left\{ \begin{array}{l}
\overline{r}_+(z) r_-(z) = |r(\lambda)|^2, \quad \quad z \in \mathbb{R}^+, \quad \lambda \in \mathbb{R}, \\
\overline{r}_+(z) r_-(z) = -|r(\lambda)|^2, \quad \; z \in \mathbb{R}^-, \quad \lambda \in i\mathbb{R}.
\end{array} \right.
\end{equation}
The scattering data $r_{\pm}$ satisfy the following properties, which are derived from the previous results.

\begin{lem}
\label{lemma-scattering-data}
Assume the condition (\ref{constraint-on-a}) on $a$. 
If $u \in H^{1,1}(\mathbb{R})$, then $r_{\pm} \in H^1(\mathbb{R})$, whereas if
$u \in H^2(\mathbb{R}) \cap H^{1,1}(\mathbb{R})$, then $r_{\pm} \in L^{2,1}(\mathbb{R})$.
Moreover, the mapping
\begin{equation}
\label{map-u-r}
H^2(\mathbb{R}) \cap H^{1,1}(\mathbb{R}) \ni u \to (r_+, r_-) \in H^1(\mathbb{R}) \cap L^{2,1}(\mathbb{R})
\end{equation}
is Lipschitz continuous.
\end{lem}

\begin{proof}
The first assertion on $r_{\pm}$ follows from Lemma \ref{regularity-a-b}.
To prove Lipschitz continuity of the mapping (\ref{map-u-r}), we use
the following representation for $r_-$ and $\tilde{r}_-$
that correspond to two potentials $u$ and $\tilde{u}$,
\begin{equation}
\label{continuity-r}
r_- - \widetilde{r}_- = \frac{2i\lambda(b-\widetilde{b})}{a} +
\frac{2i\lambda \widetilde{b}}{a\widetilde{a}}[ (\widetilde{a}-\widetilde{a}_{\infty})-(a-a_{\infty})]+\frac{2i\lambda \widetilde{b}}{a\widetilde{a}}(\widetilde{a}_{\infty}-a_{\infty}).
\end{equation}
Lipschitz continuity of the mapping (\ref{map-u-r}) for $r_-$ follows from
the representation (\ref{continuity-r}) and
Corollary \ref{cor-regularity-a-b}. Lipschitz continuity of the mapping (\ref{map-u-r})
for $r_+$ is studied by using a
representation similar to (\ref{continuity-r}).
\end{proof}

\begin{remark}
By Corollary \ref{corollary-a-b}, $a(-\lambda) = a(\lambda)$ for every $\lambda \in \mathbb{R} \cup i \mathbb{R}$.
Therefore, when we introduce $z = \lambda^2$ and start considering functions of $z$, it makes sense to introduce
${\bf a}(z) := a(\lambda)$ for every $z \in \mathbb{R}$. In what follows, we drop the bold notations
in the definition of $a(z)$.
\end{remark}

For every $x \in \mathbb{R}$ and $z \in \mathbb{R}$, we define two matrices $P_+(x;z)$ and $P_-(x;z)$ by
\begin{equation}
\label{correspondence-M}
P_+(x;z) := \left[ \frac{m_-(x;z)}{a(z)}, \; p_+(x;z) \right], \quad
P_-(x;z) := \left[ m_+(x;z), \; \frac{p_-(x;z)}{\overline{a}(z)} \right].
\end{equation}
By Lemmas \ref{Jost}, \ref{regularity-a-b}, and \ref{lemma-p}, as well as the condition (\ref{constraint-on-a}) on $a$, 
the functions $P_{\pm}(x;\cdot)$ for every $x \in \mathbb{R}$
are continued analytically in $\mathbb{C}^{\pm}$.
The scattering relations (\ref{linear-3}) and (\ref{linear-4}) are now rewritten
as the jump condition between functions $P_{\pm}(x;z)$ across the real axis in $z$ for every $x \in \mathbb{R}$:
\begin{equation} \label{jump-P}
P_+(x;z) - P_-(x;z) = P_-(x;z) R(x;z), \quad
R(x;z) := \left[\begin{matrix} \overline{r}_+(z) r_-(z) & \overline{r}_+(z) e^{-2izx} \\
r_-(z) e^{2izx} & 0 \end{matrix} \right] \quad z\in \mathbb{R}.
\end{equation}
By Lemmas \ref{asympt-mod}, \ref{regularity-a-b}, and \ref{lemma-p}, the functions
$P_{\pm}(x;\cdot)$ satisfy the limiting behavior as $|z| \to \infty$
along a contour in the domain of their analyticity in the $z$ plane:
\begin{equation}
\label{jump-bc-z}
P_{\pm}(x;z) \to \Phi_{\infty}(x) \quad \mbox{\rm as} \quad |z| \to \infty,
\end{equation}
where $\Phi_{\infty}$ is the same as in (\ref{jump-bc}).
The boundary conditions (\ref{jump-bc-z}) depend on $x$, which represents an obstacle in the inverse scattering transform,
where we reconstruct the potential $u(x)$ from the behavior of the analytic continuations
of the Jost functions $P_{\pm}(x;\cdot)$ for $x \in \mathbb{R}$.
Therefore, we fix the boundary conditions to the identity matrix by defining
new matrices
\begin{equation}
\label{formula-P-M}
M_{\pm}(x;z) := \left[ \Phi_{\infty}(x) \right]^{-1} P_{\pm}(x;z), \quad x \in \mathbb{R}, \quad z \in \mathbb{C}^{\pm}.
\end{equation}
As a result, we obtain the Riemann--Hilbert problem for analytic functions $M_{\pm}(x;\cdot)$ in $\mathbb{C}^{\pm}$,
which is given by the jump condition equipped with the uniform boundary conditions:
\begin{equation} \label{jump}
\left\{ \begin{matrix} M_+(x;z) - M_-(x;z) = M_-(x;z) R(x;z), & z\in \mathbb{R}, \\
M_{\pm}(x;z) \to I & \mbox{\rm as} \; |z| \to \infty. \end{matrix} \right.
\end{equation}
The scattering data $r_{\pm} \in H^1(\mathbb{R}) \cap L^{2,1}(\mathbb{R})$ are defined in Lemma \ref{lemma-scattering-data}.

Figure \ref{Jost-func} shows the regions of analyticity of functions $\Phi_{\pm}$ in the $\lambda$ plane (left) 
and those of functions $M_{\pm}$ in the $z$ plane (right).

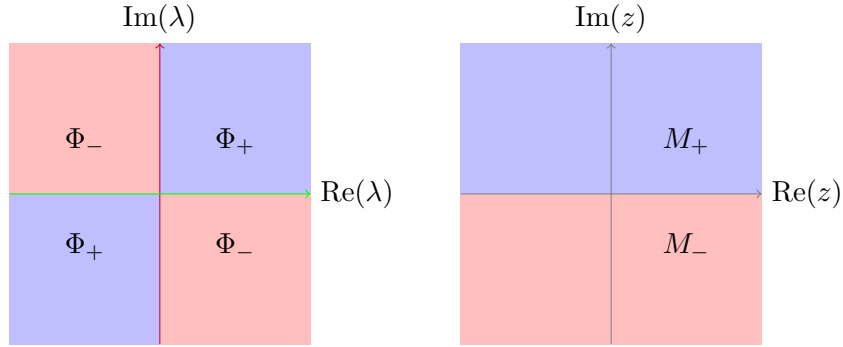
\begin{figure}[htbp] 
   \centering
 \begin{tikzpicture}
\fill[blue!25!white] (-4,0) rectangle (-2,2);
\fill[red!25!white] (-6,0) rectangle (-4,2);
\fill[blue!25!white] (-6,-2) rectangle (-4,0);
\fill[red!25!white] (-4,-2) rectangle (-2,0);

 \draw (-3,0.7) node
      {$\Phi_+$};
 \draw (-5,0.7) node
      {$\Phi_-$};
 \draw (-3,-0.7) node
      {$\Phi_-$};
 \draw (-5,-0.7) node
      {$\Phi_+$};

    \draw [purple, ->] (-4,-2) -- (-4,2)      
        node [above, black] {$\mbox{Im}(\lambda) $};              

    \draw [green, ->] (-6,0) -- (-2,0)      
        node [right, black] {$\mbox{Re}(\lambda) $}; 
        
\fill[blue!25!white] (2,0) rectangle (4,2);
\fill[blue!25!white] (0,0) rectangle (2,2);
\fill[red!25!white] (0,-2) rectangle (2,0);
\fill[red!25!white] (2,-2) rectangle (4,0);

 \draw (3,0.7) node
      {$M_+$};
 \draw (3,-0.7) node
      {$M_-$};

    \draw [gray,->] (2,-2) -- (2,2)      
        node [above, black] {$\mbox{Im}(z) $};              

    \draw [gray,->] (0,0) -- (4,0)      
        node [right, black] {$\mbox{Re}(z) $};

         \end{tikzpicture}
   \caption{Blue and red regions mark domains of analyticity of $\Phi_{\pm}$ in the $\lambda$ plane (left) 
   and those of $M_{\pm}$ in the $z$ plane (right).}
   \label{Jost-func}
\end{figure}

The scattering matrix $R$ in the Riemann--Hilbert problem (\ref{jump}) is not Hermitian. As a result, it is difficult to
use the theory of Zhou \cite{ZhouSIMA} in order to construct a unique solution for $M_{\pm}$
in the Riemann--Hilbert problem (\ref{jump}) without restricting the scattering data $r_{\pm}$ to be small in their norms.
On the other hand, the original Riemann--Hilbert problem (\ref{S}) in the $\lambda$ plane
does not have these limitations. Therefore, in the following subsection,
we consider two equivalent reductions of the Riemann--Hilbert problem (\ref{jump})
in the $z$ plane to those related with the scattering matrix $S$ instead of the scattering matrix $R$.

\subsection{Two transformations of the Riemann-Hilbert problem \eqref{jump}}

For every $\lambda \in \mathbb{C} \backslash \{0\}$, we denote
\begin{equation}
\label{RH-factors}
\tau_1(\lambda) := \left[\begin{matrix} 1 & 0\\0 & 2 i \lambda \end{matrix}\right], \quad
\tau_2(\lambda) := \left[\begin{matrix} (2 i \lambda)^{-1} & 0\\0 & 1 \end{matrix}\right]
\end{equation}
and observe that
$$
\tau_1^{-1}(\lambda) R(x;z) \tau_1(\lambda) = \tau_2^{-1}(\lambda) R(x;z) \tau_2(\lambda) = S(x;\lambda), \quad
z \in \mathbb{R}, \quad \lambda \in \mathbb{R} \cup i \mathbb{R},
$$
where $S(x;\lambda)$ is defined in (\ref{S1}) and (\ref{S2}), whereas $R(x;z)$ is defined in (\ref{jump-P}).
Using these properties, we introduce two formally equivalent reformulations
of the Riemann--Hilbert problem (\ref{jump}):
\begin{equation} \label{jump-2}
\left\{ \begin{array}{l} G_{+ 1,2}(x;\lambda) - G_{- 1,2}(x;\lambda) = G_{- 1,2}(x;\lambda) S(x;\lambda) + F_{1,2}(x;\lambda), \quad
\lambda \in \mathbb{R} \cup i \mathbb{R}, \\
\lim_{|\lambda|\rightarrow\infty} G_{\pm 1,2}(x;\lambda) = 0, \end{array} \right.
\end{equation}
where
\begin{equation}
\label{correspondence-G-F}
G_{\pm 1,2}(x;\lambda) := M_{\pm}(x;z) \tau_{1,2}(\lambda) - \tau_{1,2}(\lambda), \quad F_{1,2}(x;\lambda) := \tau_{1,2}(\lambda)  S(x;\lambda).
\end{equation}
The functions $G_{+ 1,2}(x;\lambda)$ are analytic in the first and third quadrants of the $\lambda$ plane,
whereas the functions $G_{- 1,2}(x;\lambda)$  are analytic in the second and fourth quadrants of the $\lambda$ plane.
Although the behavior of functions $M_{\pm}(x;z) \tau_{1,2}(\lambda)$ may become singular as
$\lambda \to 0$, we prove in Corollary \ref{RH-general} below
that $G_{\pm 1,2}(x;\lambda)$ are free of singularities as $\lambda \to 0$. 

Figure \ref{RH-diagram} summarizes on the transformations of the Riemann--Hilbert problems.

\begin{figure}[htbp] 
   \centering
  \begin{tikzpicture}
  \node at (0,2) (RH1) {$\Phi_{\pm}$};
  \node at (0,0) (RH2) {$M_{\pm}$};
  \node at ( -2,-2) (RH2-1) {$G_{\pm1}$};
  \node at (2,-2) (RH2-2)  {$G_{\pm2}$};
 \draw[<->] (RH2) to node[right]{$T_{1,2}$} (RH1);
  \draw[<->] (RH2-1) to node[left]{$\tau_1$} (RH2);
  \draw[<->] (RH2-2) to node[right] {$\tau_2$} (RH2);
  \draw[<->] (RH2-1) to node[below]{$2i\lambda$} (RH2-2);
\end{tikzpicture}
   \caption{A useful diagram showing transformations of the Riemann--Hilbert problems}
   \label{RH-diagram}
\end{figure}
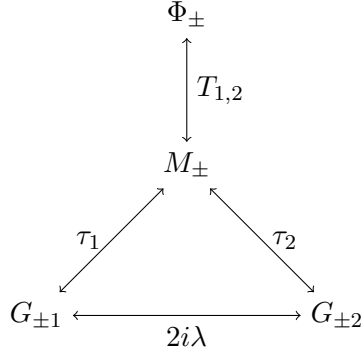

Solvability of the Riemann--Hilbert problem \eqref{jump-2} is obtained in Section 4.1. 
Then, in Section 4.2, we show that the solution to the two related Riemann-Hilbert problems
\eqref{jump-2} can be used to obtain the solution to the Riemann-Hilbert problem \eqref{jump}. 
In Section 4.3, we show how this procedure defines the inverse 
scattering transform to recover the potential $u$ of the Kaup--Newell spectral problem (\ref{lax1}) 
from the scattering data $r_{\pm}$.

\section{Inverse scattering transform}

We are now concerned with the solvability of the Riemann-Hilbert problem (\ref{jump})
for the given scattering data $r_+,r_- \in H^1(\mathbb{R}) \cap L^{2,1}(\mathbb{R})$
satisfying the constraint (\ref{r1-r2}). We are looking
for analytic matrix functions $M_{\pm}(x;\cdot)$ in $\mathbb{C}^{\pm}$ for every
$x \in \mathbb{R}$. Let us introduce the following notations for the column vectors
of the matrices $M_{\pm}$ as
\begin{equation}
\label{definitions-M}
M_{\pm}(x;z) = [ \mu_{\pm}(x;z), \eta_{\pm}(x;z) ].
\end{equation}
Before we proceed, let us inspect regularity of the reflection coefficient 
$r(\lambda)$ as a function of $z$ on $\mathbb{R}$.

\begin{prop} \label{r-regularity}
If $r_{\pm}(z) \in H^1_z(\mathbb{R})\cap L^{2,1}_z(\mathbb{R})$, then $r(\lambda) \in L^{2,1}_z(\mathbb{R})\cap L^{\infty}_z(\mathbb{R})$.
\end{prop}

\begin{proof}
Since $r_{\pm} \in L^{2,1}(\mathbb{R})$ and $|r(\lambda)|^2 = {\rm sign}(z) \; \overline{r}_+(z) r_-(z)$ for every $z \in \mathbb{R}$,
we have $r(\lambda) \in L^{2,1}_z(\mathbb{R})$ by Cauchy--Schwarz inequality.

To show that $r(\lambda) \in L^{\infty}_z(\mathbb{R})$, we notice that $r(\lambda)$ can be defined equivalently
from (\ref{r-definition}) in the following form:
$$
r(\lambda) = \left\{ \begin{matrix} -2i \lambda r_+(z) & |\lambda| \leq 1 \\ (2i \lambda)^{-1} r_-(z) & |\lambda|\geq 1.\end{matrix}\right.
$$
Since $r_{\pm} \in L^{\infty}(\mathbb{R})$ as it follows from $r_{\pm} \in H^1(\mathbb{R})$, then we have $r(\lambda) \in L^{\infty}_z(\mathbb{R})$.
\end{proof}

\begin{remark}
We do not expect generally that $r(\lambda)$ belongs to $H^1_z(\mathbb{R})$. For instance,
if
$$
h(\lambda) := \frac{\lambda}{(1+\lambda^4)^s}, \quad s > \frac{5}{4},
$$
then $\lambda h(\lambda), \lambda^{-1} h(\lambda) \in H^1_z(\mathbb{R}) \cap L^{2,1}_z(\mathbb{R})$,
$h(\lambda) \in L^{2,1}_z(\mathbb{R}) \cap L^{\infty}_z(\mathbb{R})$
but $h(\lambda) \notin H^1_z(\mathbb{R})$.
\end{remark}

We also note another useful elementary result.

\begin{prop} \label{r-boundness}
If $r_-(z) \in H^1_z(\mathbb{R})\cap L_z^{2,1}(\mathbb{R})$, then $\| \lambda r_-(z) \|_{L^{\infty}_z} \leq \| r_- \|_{H^1 \cap L^{2,1}}$.
\end{prop}

\begin{proof}
The result follows from the representation
$$
z r_-(z)^2 = \int_0^z \left( r_-(z)^2 + 2 z r_-(z) r_-'(z) \right) dz.
$$
Using Cauchy--Schwarz inequality for $r_-(z) \in H^1_z(\mathbb{R})\cap L_z^{2,1}(\mathbb{R})$,
we obtain the desired bound.
\end{proof}

\subsection{Solution to the Riemann--Hilbert problems (\ref{jump-2})}

Let us start with the definition of the Cauchy operator, which can be found in
many sources, e.g., in \cite{D-Z-1}. For any function $h \in L^p(\mathbb{R})$  with $1 \leq p < \infty$,
the Cauchy operator denoted by $\mathcal{C}$ is given by
\begin{equation}
\label{Cauchy}
\mathcal{C}(h)(z) := \frac{1}{2\pi i} \int_{\mathbb{R}} \frac{h(s)}{s-z}ds, \quad
z\in \mathbb{C}\setminus \mathbb{R}.
\end{equation}
The function $\mathcal{C}(h)$ is analytic off the real line such that $\mathcal{C}(h)(\cdot + i y)$
is in $L^p(\mathbb{R})$ for each $y \neq 0$.
When $z$ approaches to a point on the real line transversely from the upper and lower half planes,
that is, if $y \to \pm 0$,
the Cauchy operator $\mathcal{C}$ becomes the Plemelj projection operators,
denoted respectively by $\mathcal{P}^{\pm}$. These projection operators are given explicitly by
\begin{equation} \label{C-pm}
\mathcal{P}^{\pm}(h)(z) :=
\lim_{\epsilon\downarrow 0}  \frac{1}{2\pi i}\int_{\mathbb{R}}\frac{h(s)}{s-(z\pm \epsilon i)}ds , \quad z\in \mathbb{R}.
\end{equation}
The following proposition summarizes the basic properties of the Cauchy and projection operators.

\begin{prop}
For every $h \in L^p(\mathbb{R})$, $1 \leq p < \infty$,
the Cauchy operator $\mathcal{C}(h)$ is analytic off the real line,
decays to zero as $|z| \to \infty$, and approaches to $\mathcal{P}^{\pm}(h)$
almost everywhere, when a point $z \in \mathbb{C}^{\pm}$ approaches to
a point on the real axis by any non-tangential contour from $\mathbb{C}^{\pm}$.
If $1 < p < \infty$, then there exists a positive constant $C_p$ (with $C_{p=2} = 1$) such that
\begin{equation}
\label{bound-projection-operator}
\| \mathcal{P}^{\pm}(h) \|_{L^p} \leq C_p \| h \|_{L^p}.
\end{equation}
If $h \in L^1(\mathbb{R})$, then the Cauchy operator admits
the following asymptotic limit in either $\mathbb{C}^+$ or $\mathbb{C}^-$:
\begin{equation}
\label{limit-Cauchy-operator}
\lim_{|z| \to \infty} z \mathcal{C}(h)(z) = -\frac{1}{2\pi i} \int_{\mathbb{R}} h(s) ds.
\end{equation}
\label{prop-RHP}
\end{prop}

\begin{proof}
Analyticity, decay, and boundary values of $\mathcal{C}$ on the real axis follow
from Theorem 11.2 and Corollary 2 on pp. 190--191 in \cite{Duren}.
By Sokhotski--Plemelj theorem, we have the relations
\begin{equation}
\label{plemelj}
\mathcal{P}^{\pm}(h)(z) = \pm \frac{1}{2} h(z) - \frac{i}{2} \mathcal{H}(h)(z), \quad z \in \mathbb{R},
\end{equation}
where $\mathcal{H}$ is the Hilbert transform given by
$$
\mathcal{H}(h)(z) := \frac{1}{\pi}\lim_{\epsilon \downarrow 0} \left( \int_{-\infty}^{z-\epsilon}
+ \int_{z+\epsilon}^{\infty} \right) \frac{h(s)}{s-z} ds, \quad z\in \mathbb{R}.
$$
By Riesz's theorem (Theorem 3.2 in \cite{Duo}),
$\mathcal{H}$ is a bounded operator from $L^p(\mathbb{R})$ to $L^p(\mathbb{R})$
for every $1 < p < \infty$, so that the bound (\ref{bound-projection-operator})
holds with $C_2 = 1$ and $C_p \to +\infty$ as $p \to 1$ and $p \to \infty$.
Finally, the asymptotic limit (\ref{limit-Cauchy-operator}) is justified by Lebesgue's
dominated convergence theorem if $h \in L^1(\mathbb{R})$.
\end{proof}

We recall the scattering matrix $S(x;\lambda)$ given explicitly by (\ref{S1})
and (\ref{S2}). The following proposition states that 
if $r(\lambda)$ is bounded and satisfies (\ref{smallness-defocusing}), then 
the quadratic form associated with the matrix $I+S(x;\lambda)$ is strictly positive for every $x \in \mathbb{R}$ and
every $\lambda \in \mathbb{R} \cup i \mathbb{R}$, whereas the matrix $I + S(x;\lambda)$ is bounded. In 
what follows, $\| \cdot \|$ denotes the Euclidean norm of vectors in $\mathbb{C}^2$. 

\begin{prop} \label{positivity}
For every $r(\lambda) \in L^{\infty}_z(\mathbb{R})$ satisfying (\ref{smallness-defocusing}),
there exist positive constants $C_-$ and $C_+$ such that
for every $x \in \mathbb{R}$ and every column-vector $g\in \mathbb{C}^2$, we have
\begin{equation}
\label{positivity-I-S}
{\rm Re} \; g^t \left(I + S(x;\lambda)\right) g \geq C_- g^t g, \quad \lambda \in \mathbb{R} \cup i \mathbb{R}
\end{equation}
and
\begin{equation}
\label{boundness-I-S}
\left\| \left(I + S(x;\lambda)\right) g \right\| \leq C_+ \|g\|, \quad \lambda \in \mathbb{R} \cup i \mathbb{R}.
\end{equation}
\end{prop}

\begin{proof}
For $\lambda \in \mathbb{R}$, we use representation (\ref{S1}). Since $I+S(x;\lambda)$ is Hermitian for every $x \in \mathbb{R}$
and $\lambda \in \mathbb{R}$, we compute the two real eigenvalues of $I + S(x;\lambda)$ given by
$$
\mu_{\pm}(\lambda) = 1 + \frac{1}{2} |r(\lambda)|^2 \pm |r(\lambda)| \sqrt{1 + \frac{1}{4} |r(\lambda)|^2}
= \left( \sqrt{1 + \frac{1}{4} |r(\lambda)|^2} \pm \frac{1}{2} |r(\lambda)| \right)^2 > 0.
$$
Note that
$$
\frac{1}{(1 + |r(\lambda)|)^2} \leq \mu_-(\lambda) \leq \mu_+(\lambda) \leq (1 + |r(\lambda)|)^2, \quad \lambda \in \mathbb{R}.
$$
It follows from the above inequalities that the bounds (\ref{positivity-I-S}) and (\ref{boundness-I-S})
for $\lambda \in \mathbb{R}$ hold with
$$
C_- := \frac{1}{(1 + \sup_{\lambda \in \mathbb{R}} |r(\lambda)|)^2} > 0 \quad \mbox{\rm and} \quad
C_+ := (1 + \sup_{\lambda \in \mathbb{R}} |r(\lambda)|)^2 < \infty.
$$

For $\lambda \in i\mathbb{R}$, we use representation (\ref{S2}). Since $I+S(x;\lambda)$ is no longer Hermitian,
we define the Hermitian part of $S(x;\lambda)$ by
$$
S_H(\lambda) := \frac{1}{2} S(x;\lambda) + \frac{1}{2} S^*(x;\lambda) = \left[\begin{matrix} -|r(\lambda)|^2 &
0 \\ 0 &0 \end{matrix}\right],
$$
where the asterisk denotes Hermite conjugate (matrix transposition and complex conjugate).
It follows from (\ref{smallness-defocusing}) that $\sup_{\lambda \in i \mathbb{R}} r(\lambda) \leq 1 - c_0^2 < 1$
so that the diagonal matrix $I + S_H(\lambda)$ is positive definite for every $\lambda \in i \mathbb{R}$.
The bound (\ref{positivity-I-S}) for $\lambda \in i \mathbb{R}$
follows from this estimate with $C_- := 1 - \sup_{\lambda \in i \mathbb{R}} |r(\lambda)|^2 \geq c_0^2 > 0$.
Finally, estimating componentwise
\begin{eqnarray*}
\left\| (I + S(x;\lambda)) g \right\|^2 & \leq & (1 + |r(\lambda)|^2) \|g\|^2 + |r(\lambda)|^2 \left( r(\lambda) g^{(1)} \overline{g^{(2)}}
+ \overline{r(\lambda)} \overline{g^{(1)}} g^{(2)} \right) \\
& \leq & \left( 1 + |r(\lambda)|^2 \right) \left(1 + \frac{1}{2} |r(\lambda)|^2 \right) \|g\|^2,
\end{eqnarray*}
we obtain the bound (\ref{boundness-I-S}) for $\lambda \in i \mathbb{R}$ with $C_+ := (1 + \sup_{\lambda \in i\mathbb{R}} |r(\lambda)|^2) < \infty$.
\end{proof}

Thanks to the result of Proposition \ref{positivity}, we shall prove solvability of the two related Riemann--Hilbert problems
(\ref{jump-2}) by using the method of Zhou \cite{ZhouSIMA}. Dropping the subscripts, 
we rewrite the two related Riemann--Hilbert problems (\ref{jump-2}) in the following abstract form
\begin{equation} \label{jump-3}
\left\{ \begin{matrix} G_+(x;\lambda) - G_-(x;\lambda) = G_-(x;\lambda) S(x;\lambda) + F(x;\lambda), &
\lambda \in \mathbb{R} \cup i \mathbb{R}, \\
G_{\pm}(x,\lambda) \to 0 & \mbox{\rm as} \; |\lambda| \rightarrow \infty. \end{matrix} \right.
\end{equation}
If $r_{\pm} \in H^1_z(\mathbb{R}) \cap L^{2,1}(\mathbb{R})$, then Proposition \ref{r-regularity}
implies that
$S(x;\lambda) \in L^1_z(\mathbb{R}) \cap L^{\infty}_z(\mathbb{R})$ and $F(x;\lambda) \in L^2_z(\mathbb{R})$
for every $x \in \mathbb{R}$. We consider the class of solutions to the Riemann--Hilbert problem (\ref{jump-3}) such that
for every $x \in \mathbb{R}$,
\begin{itemize}
\item $G_{\pm}(x;\lambda)$ are analytic functions of $z = \lambda^2$ in $\mathbb{C}^{\pm}$
\item $G_{\pm}(x;\lambda) \in L^2_z(\mathbb{R})$
\item The same columns of $G_{\pm}(x;\lambda)$, $G_-(x;\lambda) S(x;\lambda)$, and $F(x;\lambda)$
are either even or odd in $\lambda$.
\end{itemize}

By Proposition \ref{prop-RHP} with $p = 2$, for every $x \in \mathbb{R}$,
the Riemann-Hilbert problem (\ref{jump-3}) has a solution given by the Cauchy operator
\begin{equation}\label{RH-1}
G_{\pm}(x;\lambda) = \mathcal{C} \left(G_-(x;\lambda) S(x;\lambda) + F(x;\lambda) \right)(z), \quad z\in \mathbb{C}^{\pm}
\end{equation}
if and only if there is a solution $G_-(x;\lambda) \in L^2_z(\mathbb{R})$ of the Fredholm integral equation:
\begin{equation}\label{RH-2}
G_-(x;\lambda) = \mathcal{P}^- \left(G_-(x;\lambda) S(x;\lambda) + F(x;\lambda) \right)(z), \quad z\in \mathbb{R}.
\end{equation}
Once $G_-(x;\lambda) \in L^2_z(\mathbb{R})$ is found from the Fredholm integral equation (\ref{RH-2}),
then $G_+(x;\lambda) \in L^2_z(\mathbb{R})$ is obtained from the projection formula
\begin{equation}\label{RH-4}
G_+(x;\lambda) = \mathcal{P}^+ \left(G_-(x;\lambda) S(x;\lambda) + F(x;\lambda) \right)(z), \quad z\in \mathbb{R}.
\end{equation}

\begin{remark}
The complex integrals in $\mathcal{C}$ and $\mathcal{P}^{\pm}$ over the real line $z = \lambda^2$ can be parameterized 
by $\lambda$ on $\mathbb{R}^+ \cup i \mathbb{R}^+$. Extensions of integral
representations (\ref{RH-1}), (\ref{RH-2}), and (\ref{RH-4}) for $\lambda \in \mathbb{R}^- \cup i \mathbb{R}^-$ is performed
with the account of parity symmetries of the corresponding columns of $G_{\pm}(x;\lambda)$, $G_-(x;\lambda) S(x;\lambda)$, and $F(x;\lambda)$.
See Proposition \ref{prop-redundancy}, Corollary \ref{cor-redundancy}, and Remark \ref{remark-redundancy} below.
\end{remark}

The following lemma relies on the positivity result of Proposition \ref{positivity} and
states solvability of the integral equation (\ref{RH-2}) in $L^2_z(\mathbb{R})$. For simplicity
of notations, we drop dependence of $S$, $F$ and $G_{\pm}$ from the variable $x$.

\begin{lem} \label{solve-RH-integral}
For every $r(\lambda) \in L^2_z(\mathbb{R}) \cap L^{\infty}_z(\mathbb{R})$ satisfying (\ref{smallness-defocusing})
and every $F(\lambda) \in L^2_z(\mathbb{R})$, there is a unique solution $G(\lambda) \in L^2_z(\mathbb{R})$ of the
linear inhomogeneous equation
\begin{equation}
\label{RH-3}
(I - \mathcal{P}^-_S) G(\lambda) = F(\lambda), \quad \lambda \in \mathbb{R} \cup i \mathbb{R},
\end{equation}
where $\mathcal{P}^-_S G := \mathcal{P}^-(G S)$.
\end{lem}

\begin{proof}
The operator $I-\mathcal{P}^-_S$ is known to be a Fredholm operator of the index zero \cite{B-C-1,B-C-2,ZhouSIMA}.
By Fredholm's alternative, a unique solution to the linear integral equation (\ref{RH-3}) exists for
$G(\lambda) \in L^2_z(\mathbb{R})$ if and only if the zero solution to the homogeneous equation
$(I - \mathcal{P}_S^-) g = 0$ is unique in $L^2_z(\mathbb{R})$.

Suppose that there exists nonzero $g \in L^2_z(\mathbb{R})$ such that $(I - \mathcal{P}_S^-) g = 0$.
Since $S(\lambda) \in L^2_z(\mathbb{R}) \cap L^{\infty}_z(\mathbb{R})$, we define
two analytic functions in $\mathbb{C}\setminus\mathbb{R}$ by
$$
g_1(z) := \mathcal{C}(gS)(z) \quad \mbox{\rm and} \quad g_2(z) := \mathcal{C}(g S)^*(z),
$$
where the asterisk denotes Hermite conjugate.
We multiply the two functions by each other and integrate along the semi-circle of radius $R$ centered at zero in $\mathbb{C}^+$.
Because $g_1$ and $g_2$ are analytic functions in $\mathbb{C}^+$,
the Cauchy--Goursat theorem implies that
$$
0 = \oint g_1(z) g_2(z) dz.
$$
Because $g(\lambda), S(\lambda) \in L^2_z(\mathbb{R})$, we have $g(\lambda) S(\lambda) \in L^1_z(\mathbb{R})$, so that
the asymptotic limit (\ref{limit-Cauchy-operator}) in Proposition \ref{prop-RHP} implies that
$g_{1,2}(z) = \mathcal{O}(z^{-1})$ as $|z| \to \infty$. Therefore,
the integral on arc  goes to zero as $R \to \infty$, so that we obtain
\begin{eqnarray*}
0 & = & \int_{\mathbb{R}} g_1(z) g_2(z) dz \\
& = & \int_{\mathbb{R}} \mathcal{P}^+ (g S) \; [\mathcal{P}^-(g S)]^* dz \\
 & = & \int_{\mathbb{R}} \left[ \mathcal{P}^-(g S) + g S \right] [\mathcal{P}^- (g S)]^* dz,
\end{eqnarray*}
where we have used the identity $\mathcal{P}^+ - \mathcal{P}^- = I$ following from
relations (\ref{plemelj}). Since $\mathcal{P}^-(g S) = g$, we finally obtain
\begin{eqnarray}
\label{empty-equation}
0 =  \int_{\mathbb{R}} g (I+S) g^* dz.
\end{eqnarray}
By bound (\ref{positivity-I-S}) in Proposition \ref{positivity}, the real part of the
quadratic form associated with the matrix $I + S$
is strictly positive definite for every $z \in \mathbb{R}$. Therefore, equation (\ref{empty-equation}) 
implies that $g=0$ is the only solution to the homogeneous equation $(I - \mathcal{P}_S^-) g = 0$
in $L^2_z(\mathbb{R})$.
\end{proof}

As a consequence of Lemma \ref{solve-RH-integral}, we obtain
solvability of the two related Riemann--Hilbert problems (\ref{jump-2}).

\begin{cor}\label{RH-general}
Let $r_{\pm}\in H^1(\mathbb{R})\cap L^{2,1}(\mathbb{R})$ such that the inequality (\ref{smallness-defocusing})
is satisfied. There exists a unique solution to the Riemann--Hilbert problems (\ref{jump-2}) for every $x \in \mathbb{R}$
such that the functions
$$
G_{\pm 1,2}(x;\lambda) := M_{\pm}(x;z) \tau_{1,2}(\lambda) - \tau_{1,2}(\lambda)
$$
are analytic functions of $z$ in $\mathbb{C}^{\pm}$ and $G_{\pm 1,2}(x;\lambda) \in L^2_z(\mathbb{R})$.
\end{cor}

\begin{proof}
For every $x \in \mathbb{R}$, the two related Riemann--Hilbert problems (\ref{jump-2}) are rewritten for $G_{\pm 1,2}$ and $F_{1,2}$
given by (\ref{correspondence-G-F}) in the form (\ref{jump-3}). By Proposition \ref{r-regularity},
we have $S(x;\lambda) \in L^1_z(\mathbb{R}) \cap L^{\infty}_z(\mathbb{R})$ and
$F_{1,2}(x;\lambda) \in L^2_z(\mathbb{R})$, hence $\mathcal{P}^-(F_{1,2}) \in L^2_z(\mathbb{R})$.
By Lemma \ref{solve-RH-integral}, equation (\ref{RH-2}) admits a unique solution for $G_{-1,2}(x;\lambda) \in L^2_z(\mathbb{R})$
for every $x \in \mathbb{R}$.
Then, we define a unique solution for $G_{+1,2}(x;\lambda) \in L^2_z(\mathbb{R})$ by equation (\ref{RH-4}).
Analytic extensions of $G_{\pm 1,2}(x;\lambda)$ as functions of $z$ in $\mathbb{C}^{\pm}$ are
defined by the Cauchy integrals (\ref{RH-1}). These functions solve the Riemann--Hilbert problem
(\ref{jump-3}) by Proposition \ref{prop-RHP} with $p = 2$.
\end{proof}

For further estimates, we modify the method of Lemma \ref{solve-RH-integral} and
prove that the operator $(I-\mathcal{P}^-_S)^{-1}$ in the
integral Fredholm equation (\ref{RH-3}) is invertible with a bounded inverse
in space $L^2_z(\mathbb{R})$.

\begin{lem} \label{inverse-fredholm}
For every $r(\lambda) \in L^2_z(\mathbb{R}) \cap L^{\infty}_z(\mathbb{R})$ satisfying (\ref{smallness-defocusing}),
the inverse operator $(I-\mathcal{P}^-_S)^{-1}$ is a bounded operator from $L^2_z(\mathbb{R})$ to $L^2_z(\mathbb{R})$.
In particular, there is a positive constant $C$ that only depends on $\| r(\lambda) \|_{L^{\infty}_z}$
such that for every row-vector $f \in L^2_z(\mathbb{R})$, we have
\begin{equation}
\label{bound-on-inverse-lemma}
\| (I-\mathcal{P}^-_S)^{-1} f \|_{L^2_z} \leq C \| f \|_{L^2_z}.
\end{equation}
\end{lem}

\begin{proof}
We consider the linear inhomogeneous equation (\ref{RH-3}) with $F \in L^2_z(\mathbb{R})$.
Recalling that $\mathcal{P}^+ - \mathcal{P}^- = I$, we write $G = G_+ - G_-$, where
$G_+$ and $G_-$ satisfy the inhomogeneous equations
\begin{equation}
\label{inhomogeneous-fredholm-plus-minus}
G_- - \mathcal{P}^- (G_- S) = \mathcal{P}^-(F), \quad G_+ - \mathcal{P}^- (G_+ S) = \mathcal{P}^+(F).
\end{equation}
By Lemma \ref{solve-RH-integral}, since $\mathcal{P}^{\pm}(F) \in L^2_z(\mathbb{R})$,
there are unique solutions to the inhomogeneous equations
(\ref{RH-3}) and (\ref{inhomogeneous-fredholm-plus-minus}), so that the decomposition $G = G_+ - G_-$ is unique.
Therefore, we only need to find the estimates of $G_+$ and $G_-$ in $L^2_z(\mathbb{R})$.

To deal with $G_-$, we define two analytic functions in $\mathbb{C}\setminus\mathbb{R}$ by
$$
g_1(z) := \mathcal{C}(G_- S)(z) \quad \mbox{\rm and} \quad
g_2(z) := \mathcal{C}(G_- S + F)^*(z),
$$
similarly to the proof of Lemma \ref{solve-RH-integral}. By Proposition \ref{prop-RHP},
$g_1(z) = \mathcal{O}(z^{-1})$ and $g_2(z) \to 0$ as $|z| \to \infty$, since
$F \in L^2_z(\mathbb{R})$, $G_- \in L^2_z(\mathbb{R})$, and $S(\lambda) \in L^2_z(\mathbb{R}) \cap L^{\infty}_z(\mathbb{R})$.
Therefore, the integral on the semi-circle of radius $R > 0$ in the upper half-plane
still goes to zero as $R \to \infty$ by Lebesgue's dominated convergence theorem.
Performing the same manipulations as in the proof of Lemma \ref{solve-RH-integral},
we obtain
\begin{eqnarray*}
0 & = & \oint g_1(z) g_2(z) dz  \\
& = & \int_{\mathbb{R}} \mathcal{P}^+ (G_- S) \left[ \mathcal{P}^- (G_- S + F) \right]^* dz \\
& = & \int_{\mathbb{R}} \left[ \mathcal{P}^- (G_- S) + G_- S \right] \left[ \mathcal{P}^- (G_- S + F) \right]^* dz \\
& = &  \int_{\mathbb{R}} \left[ G_- - \mathcal{P}^-(F) + G_- S \right] G_-^* dz,
\end{eqnarray*}
where we have used the first inhomogeneous equation in system (\ref{inhomogeneous-fredholm-plus-minus}).
By the bound (\ref{positivity-I-S}) in Proposition \ref{positivity}, there is a positive constant $C_-$ such that
$$
C_- \|G_-\|_{L^2}^2 \leq {\rm Re}\; \int_{\mathbb{R}} G_- (I + S) G_-^* dz =
{\rm Re} \int_{\mathbb{R}} \mathcal{P}^-(F) G_-^* dz \leq \| F \|_{L^2} \| G_- \|_{L^2},
$$
where we have used the Cauchy--Schwarz inequality and bound
(\ref{bound-projection-operator}) with $C_{p=2} = 1$.
Note that the above estimate holds independently for the corresponding row-vectors of the matrices
$G_-$ and $F$. Since $G_- = (I-\mathcal{P}^-_S)^{-1}\mathcal{P}^- F$,
for every row-vector $f \in L^2_z(\mathbb{R})$ of the matrix $F \in L^2_z(\mathbb{R})$,
the above inequality yields
\begin{equation}
\label{G-minus-bound}
\|(I-\mathcal{P}^-_S)^{-1} \mathcal{P}^- f \|_{L^2_z} \leq C_-^{-1} \| f \|_{L^2_z}.
\end{equation}

To deal with $G_+$, we use $\mathcal{P}^+ - \mathcal{P}^- = I$ and rewrite
the second inhomogeneous equation in system (\ref{inhomogeneous-fredholm-plus-minus}) as
follows:
\begin{equation}
\label{inhomogeneous-fredholm-last}
G_+ (I + S) - \mathcal{P}^+ (G_+ S) = \mathcal{P}^+(F).
\end{equation}
We now define two analytic functions in $\mathbb{C}\setminus\mathbb{R}$ by
$$
g_1(z) := \mathcal{C}(G_+ S)(z) \quad \mbox{\rm and} \quad
g_2(z) := \mathcal{C}(G_+ S + F)^*(z)
$$
and integrate the product of $g_1$ and $g_2$ on the semi-circle of radius $R > 0$ in the lower half-plane.
Performing the same manipulations as above, we obtain
\begin{eqnarray*}
0 & = & \oint g_1(z) g_2(z) dz \\
& = & \int_{\mathbb{R}} \mathcal{P}^- (G_+ S) \left[ \mathcal{P}^+ (G_+ S + F) \right]^* dz \\
& = &  \int_{\mathbb{R}} \left[ G_+ - \mathcal{P}^+(F) \right] \left[ G_+ (I + S) \right]^* dz,
\end{eqnarray*}
where we have used equation (\ref{inhomogeneous-fredholm-last}).

By the bounds (\ref{positivity-I-S}) and (\ref{boundness-I-S}) in
Proposition \ref{positivity}, there are positive constants $C_+$ and $C_-$ such that
$$
C_- \| G_+\|_{L^2}^2 \leq {\rm Re}\; \int_{\mathbb{R}} G_+ (I + S)^* G_+^* dz =
{\rm Re} \int_{\mathbb{R}} \mathcal{P}^+(F) (I+S)^* G_+^* dz \leq C_+ \| F \|_{L^2} \| G_+ \|_{L^2},
$$
where we have used the Cauchy--Schwarz inequality and bound
(\ref{bound-projection-operator}) with $C_{p=2} = 1$.
Again, the above estimate holds independently for the corresponding row-vectors of the matrices
$G_+$ and $F$. Since $G_+ = (I-\mathcal{P}^-_S)^{-1}\mathcal{P}^+ F$,
for every row-vector $f \in L^2_z(\mathbb{R})$ of the matrix $F \in L^2_z(\mathbb{R})$,
the above inequality yields
\begin{equation}
\label{G-plus-bound}
\|(I-\mathcal{P}^-_S)^{-1} \mathcal{P}^+ f \|_{L^2_z} \leq C_-^{-1} C_+ \| f \|_{L^2_z}.
\end{equation}
The assertion of the lemma is proved with bounds (\ref{G-minus-bound}), (\ref{G-plus-bound}), and
the triangle inequality.
\end{proof}

\subsection{Estimates on solutions to the Riemann-Hilbert problem (\ref{jump})}

Using Corollary \ref{RH-general}, we obtain solvability of the Riemann--Hilbert problem (\ref{jump}).
Indeed, the abstract Riemann--Hilbert problem (\ref{jump-3}) is derived for two
versions of $G_{\pm}$ and $F_{\pm}$ given by (\ref{correspondence-G-F}).
For the first version, we have
\begin{equation}
\label{G-1}
G_{\pm 1}(x;\lambda) := M_{\pm}(x;z) \tau_1(\lambda) - \tau_1(\lambda) = \left[ \mu_{\pm}(x;z) - e_1,
2 i \lambda \left( \eta_{\pm}(x;z) - e_2 \right) \right]
\end{equation}
and
\begin{equation}
\label{F-1}
F_1(x;\lambda) := \tau_1(\lambda)  S(x;\lambda) = R(x;z) \tau_1(\lambda).
\end{equation}
By Corollary \ref{RH-general}, there is a solution $G_{\pm 1}(x;\lambda) \in L^2_z(\mathbb{R})$
of the integral Fredholm equations
\begin{equation} \label{projection-int-type-1}
G_{\pm 1}(x;\lambda) = \mathcal{P}^{\pm} \left( G_{- 1}(x;\lambda) S(x;\lambda) + F_1(x;\lambda) \right)(z), \quad  z\in \mathbb{R}.
\end{equation}
Using equation (\ref{projection-int-type-1}) for the first column of $G_{\pm}$,
we obtain
\begin{equation} \label{projection-int-type-2}
\mu_{\pm}(x;z) - e_1 = \mathcal{P}^{\pm} \left( M_-(x;\cdot) R(x;\cdot) \right)^{(1)}(z), \quad  z\in \mathbb{R},
\end{equation}
where we have used the following identities:
\begin{eqnarray*}
(G_{- 1} S + F_1)^{(1)} = (M_- \tau_1 S)^{(1)} = (M_- R \tau_1)^{(1)} = (M_- R)^{(1)}.
\end{eqnarray*}

For the second version of the abstract Riemann--Hilbert problem (\ref{jump-3}), we have
\begin{equation}
\label{G-2}
G_{\pm 2}(x;\lambda) := M_{\pm}(x;z) \tau_2(\lambda) - \tau_2(\lambda) = \left[
(2i \lambda)^{-1} \left( \mu_{\pm}(x;z) - e_1 \right), \eta_{\pm}(x;z) - e_2 \right]
\end{equation}
and
\begin{equation}
\label{F-2}
F_2(x;\lambda) := \tau_2(\lambda)  S(x;\lambda) = R(x;z) \tau_2(\lambda).
\end{equation}
Again by Corollary \ref{RH-general}, there is a solution $G_{\pm 2}(x;\lambda) \in L^2_z(\mathbb{R})$
of the integral Fredholm equations (\ref{projection-int-type-1}), where $G_{\pm 1}$ and $F_1$
are replaced by $G_{\pm 2}$ and $F_2$. Using equation (\ref{projection-int-type-1})
for the second column of $G_{\pm 2}$, we obtain
\begin{equation} \label{projection-int-type-3}
\eta_{\pm}(x;z) - e_2 = \mathcal{P}^{\pm} \left( M_-(x;\cdot) R(x;\cdot) \right)^{(2)}(z), \quad  z \in \mathbb{R}.
\end{equation}
where we have used the following identities:
\begin{eqnarray*}
(G_{-2} S + F_2)^{(2)} = (M_- \tau_2 S)^{(2)} = (M_- R \tau_2)^{(2)} = (M_- R)^{(2)}.
\end{eqnarray*}

Equations (\ref{projection-int-type-2}) and (\ref{projection-int-type-3}) can be written in the form
\begin{equation}
\label{RH-M}
M_{\pm}(x;z) = I + \mathcal{P}^{\pm} \left( M_-(x;\cdot) R(x;\cdot) \right)(z), \quad z\in \mathbb{R},
\end{equation}
which represents the solution to the Riemann--Hilbert problem (\ref{jump}) on the real line.
The analytic continuation of functions $M_{\pm}(x;\cdot)$ in $\mathbb{C}^{\pm}$ is given by
the Cauchy operators
\begin{equation}
\label{RH-M-complex}
M_{\pm}(x;z) = I + \mathcal{C} \left( M_-(x;\cdot) R(x;\cdot) \right)(z), \quad z\in \mathbb{C}^{\pm}.
\end{equation}
The corresponding result on solvability of the integral equations (\ref{RH-M}) is given by the following lemma.

\begin{lem} \label{inverse-fredholm-cor}
Let $r_{\pm}\in H^1(\mathbb{R})\cap L^{2,1}(\mathbb{R})$ such that the inequality (\ref{smallness-defocusing})
is satisfied. There is a positive constant $C$ that only depends on $\| r_{\pm} \|_{L^{\infty}}$
such that the unique solution to the integral equations (\ref{RH-M}) enjoys the estimate for every $x \in \mathbb{R}$,
\begin{equation}
\label{bound-on-M-tech}
\| M_{\pm}(x;\cdot) - I \|_{L^2} \leq C \left( \| r_+ \|_{L^2} + \| r_- \|_{L^2} \right).
\end{equation}
\end{lem}

\begin{proof}
By Proposition \ref{r-regularity},  if $r_{\pm} \in H^1(\mathbb{R})\cap L^{2,1}(\mathbb{R})$,
then $r(\lambda) \in L^2(\mathbb{R}) \cap L^{\infty}_z(\mathbb{R})$. Under these conditions,
it follows from the explicit expressions (\ref{F-1}) and (\ref{F-2}) that
$R(x;z) \tau_{1,2}(\lambda)$ belong to $L^2_z(\mathbb{R})$ for every $x \in \mathbb{R}$ and
there is a positive constant $C$ that only depends on $\| r_{\pm} \|_{L^{\infty}(\mathbb{R})}$
such that for every $x \in \mathbb{R}$,
\begin{equation}
\label{bound-on-R-in-M}
\| R(x;z) \tau_{1,2}(\lambda) \|_{L^2_z} \leq C \left( \| r_+ \|_{L^2} + \| r_- \|_{L^2} \right).
\end{equation}
By derivation above, the integral equation (\ref{RH-M}) for the projection operator $\mathcal{P}^-$
is obtained from two versions of the integral
equation (\ref{RH-3}) corresponding to $F_{1,2}(x;\lambda) := \mathcal{P}^- \left( R(x;z) \tau_{1,2}(\lambda) \right)(z)$.
Therefore, each element of $M_-(x;z)$ enjoys
the bound (\ref{bound-on-inverse-lemma}) for the corresponding
row vectors of the two versions of $F_{1,2}(x;z)$.
Combining the estimates (\ref{bound-on-inverse-lemma}) and (\ref{bound-on-R-in-M}),
we obtain the bound (\ref{bound-on-M-tech}).
\end{proof}

Before we continue, let us discuss the redundancy between solutions to the two versions
of the Riemann--Hilbert problems (\ref{jump-2}). By using equation (\ref{projection-int-type-1}) for the second column of $G_{\pm 1}$,
we obtain
\begin{equation} \label{projection-int-type-3-red}
2 i \lambda \left( \eta_{\pm}(x;z) - e_2 \right) =
\mathcal{P}^{\pm} \left( 2 i \lambda \left( M_-(x;\cdot) R(x;\cdot) \right)^{(2)} \right)(z), \quad  z \in \mathbb{R}.
\end{equation}
By using equation (\ref{projection-int-type-1}) for the first column of $G_{\pm 2}$,
we obtain
\begin{equation} \label{projection-int-type-2-red}
(2 i \lambda)^{-1} \left( \mu_{\pm}(x;z) - e_1 \right) =
\mathcal{P}^{\pm} \left( (2 i \lambda)^{-1} \left( M_-(x;\cdot) R(x;\cdot) \right)^{(1)} \right)(z), \quad  z\in \mathbb{R}.
\end{equation}
Unless equations (\ref{projection-int-type-3-red}) and (\ref{projection-int-type-2-red}) are redundant
in view of equations (\ref{projection-int-type-2}) and (\ref{projection-int-type-3}),
the two versions of the Riemann--Hilbert problems (\ref{RH-2}) may seem to be inconsistent.
In order to show the redundancy explicitly, we use the following result.

\begin{prop}
\label{prop-redundancy}
Let $f(\lambda) \in L^1_z(\mathbb{R}) \cap L^{\infty}_z(\mathbb{R})$ be even in $\lambda$
for all $\lambda \in \mathbb{R} \cup i \mathbb{R}$. Then
\begin{equation}
\label{projection-even}
\mathcal{P}^{\pm}_{\rm even}\left(\lambda f(\lambda) \right)(\lambda) = \lambda \mathcal{P}^{\pm}_{\rm even}(f)(\lambda),
\quad \lambda \in \mathbb{R} \cup i \mathbb{R},
\end{equation}
where
\begin{equation}
\label{integral-even}
\mathcal{P}^{\pm}_{\rm even}(f)(\lambda) := \left( \int_0^{+\infty} + \int_{+i\infty}^{i0} + \int_0^{-\infty} + \int_{-i \infty}^{i0} \right)
\frac{f(\lambda') d\lambda'}{\lambda' - (\lambda \pm i0)} \equiv \mathcal{P}^{\pm}(f(\lambda))(\lambda^2).
\end{equation}
Similarly, let $g(\lambda) \in L^1_z(\mathbb{R}) \cap L^2_z(\mathbb{R})$ be odd in $\lambda$
for all $\lambda \in \mathbb{R} \cup i \mathbb{R}$. Then
\begin{equation}
\label{projection-odd}
\mathcal{P}^{\pm}_{\rm odd}\left(\lambda g(\lambda) \right)(\lambda) = \lambda \mathcal{P}^{\pm}_{\rm odd}(g)(\lambda),
\quad \lambda \in \mathbb{R} \cup i \mathbb{R},
\end{equation}
where
\begin{equation}
\label{integral-odd}
\mathcal{P}^{\pm}_{\rm odd}(g)(\lambda) := \left( \int_0^{+\infty} + \int_{+i\infty}^{i0} + \int_{-\infty}^0 + \int_{i0}^{-i \infty} \right)
\frac{g(\lambda') d\lambda'}{\lambda' - (\lambda \pm i0)} \equiv \mathcal{P}^{\pm}(g(\lambda))(\lambda^2).
\end{equation}
\end{prop}

\begin{proof}
First, we note the validity of the definition (\ref{integral-even}) if $f(-\lambda) = f(\lambda)$:
\begin{eqnarray*}
\mathcal{P}^{\pm}(f(\lambda))(\lambda^2) & = & \int_{-\infty}^{\infty}
\frac{f(\lambda^{\prime}) 2 \lambda' d \lambda'}{(\lambda^{\prime})^2 - (\lambda^2 \pm i 0)} \\
& = &
\left( \int_0^{+\infty} + \int_{+i \infty}^{i0} \right) f(\lambda') \left[ \frac{1}{\lambda' - (\lambda \pm i0)} +
\frac{1}{\lambda' + (\lambda \pm i0)} \right] d \lambda' =: \mathcal{P}^{\pm}_{\rm even}(f)(\lambda).
\end{eqnarray*}
Then, relation (\ref{projection-even}) is established from the trivial result
$$
\left( \int_0^{+\infty} + \int_{+i\infty}^{i0} + \int_0^{-\infty} + \int_{-i \infty}^{i0} \right)
f(\lambda') d\lambda' = 0,
$$
which is justified if $f(\lambda) \in L^1_z(\mathbb{R})$ and even in $\lambda$.
The relation (\ref{projection-odd}) is proved similarly, thanks to the changes in the definition (\ref{integral-odd}).
\end{proof}

\begin{figure}[htbp] 
   \centering
 \begin{tikzpicture}

     \draw [-] (-4,1) -- (-4,-1);
     \draw[->](-4,-2)--(-4,-1);  	
    \draw [ <-] (-4,1) -- (-4,2)      
        node [above, black] {$\mbox{Im}(\lambda) $};              

    \draw[-](-6,0)--(-5,0);
    \draw[<-](-5,0)--(-4,0);
    \draw[->](-4,0)--(-3,0);
    \draw [-] (-3,0) -- (-2,0)      
        node [right, black] {$\mbox{Re}(\lambda) $};

     \draw [->] (4,1) -- (4,-1);
     \draw[-](4,-2)--(4,-1);  	
    \draw [ <-] (4,1) -- (4,2)      
        node [above, black] {$\mbox{Im}(\lambda) $};              

    \draw[->](2,0)--(3,0);
    \draw[-](3,0)--(4,0);
    \draw[->](4,0)--(5,0);
    \draw [-] (5,0) -- (6,0)      
        node [right, black] {$\mbox{Re}(\lambda) $};
 \end{tikzpicture}

   \caption{The left and right panels show the direction of contours used for $\mathcal{P}^{\pm}_{\rm even}$ and $\mathcal{P}^{\pm}_{\rm odd}$}
   \label{fig-contours}
\end{figure}
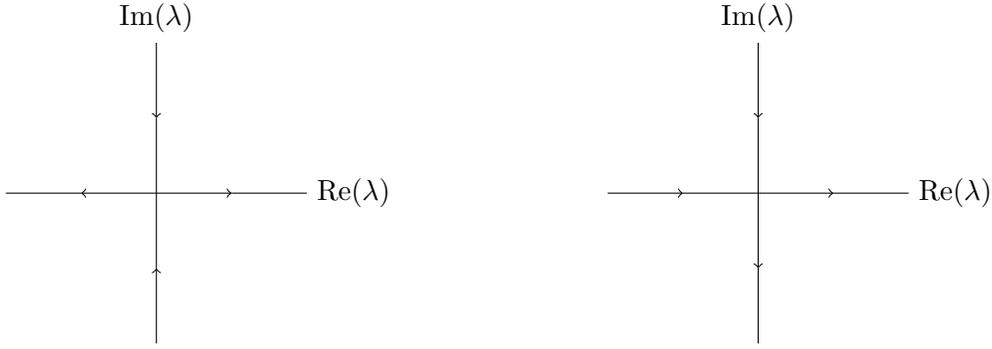

Figure \ref{fig-contours} shows the contours of integration used in the definitions of $\mathcal{P}^{\pm}_{\rm even}$ and $\mathcal{P}^{\pm}_{\rm odd}$
in (\ref{integral-even}) and (\ref{integral-odd}). The following corollary of Proposition \ref{prop-redundancy} 
specifies the redundancy between the two different versions of the Riemann--Hilbert problems (\ref{jump-2}).

\begin{cor}
\label{cor-redundancy}
Consider two unique solutions to the Riemann--Hilbert problems (\ref{jump-2}) in Corollary \ref{RH-general}.
Then, for every $x \in \mathbb{R}$, we have 
\begin{equation}
\label{G-redundancy}
G_{\pm 1}(x;\lambda) = 2i \lambda {\rm sign}(\lambda) G_{\pm 2}(x;\lambda),  \quad \lambda \in \mathbb{R} \cup i \mathbb{R},
\end{equation}
where the sign function returns the sign of either real or imaginary part of $\lambda$. 
\end{cor}

\begin{proof}
We note the relation $\tau_2^{-1}(\lambda) \tau_1(\lambda) = 2i \lambda I$, where $I$ is the identity $2$-by-$2$ matrix.
From here, the relation (\ref{G-redundancy}) follows for $\lambda \in \mathbb{R}^+ \cup i \mathbb{R}^+$. 
To consider the continuation of this relation to $\lambda \in \mathbb{R}^- \cup i \mathbb{R}^-$, we apply 
Proposition \ref{prop-redundancy} with the explicit parametrization of the contours of integrations 
as on Figure \ref{fig-contours}. We choose the even function $f$
and the odd function $g$ in the form
$$
f(\lambda) := \left( M_-(x;\lambda^2) R(x;\lambda^2) \right)^{(2)}, \quad
g(\lambda) := (2 i \lambda)^{-1} \left( M_-(x;\lambda^2) R(x;\lambda^2) \right)^{(1)}.
$$
Then, equation (\ref{projection-int-type-3-red}) follows from equation (\ref{projection-int-type-3}),
thanks to the relation (\ref{projection-even}), whereas equation (\ref{projection-int-type-2})
follows from equation (\ref{projection-int-type-2-red}) thanks to the relation (\ref{projection-odd}). 
Thus, the relation (\ref{G-redundancy}) is verified for every $\lambda \in \mathbb{R} \cup i \mathbb{R}$. 
To ensure that the integrations (\ref{integral-even}) and (\ref{integral-odd}) returns $\mathcal{P}^{\pm}$ 
for $\lambda f(\lambda)$ and $\lambda g(\lambda)$, the sign function is used in the relation (\ref{G-redundancy}).
\end{proof}

\begin{remark}
Corollary \ref{cor-redundancy} shows that the complex integration in the $z$ plane in the integral 
equations (\ref{projection-int-type-1}) has to be extended in two different ways in the $\lambda$ plane. 
For the first vector columns of the integral equation (\ref{projection-int-type-1}), we have to use the definition
(\ref{integral-odd}) for odd functions in $\lambda$, whereas for the second vector
columns of the integral equation (\ref{projection-int-type-1}), we have to use
the definition (\ref{integral-even}) for even functions in $\lambda$.
\label{remark-redundancy}
\end{remark}

Next, we shall obtain refined estimates on the solution to the integral equations (\ref{RH-M}).
We start with estimates on the scattering coefficients $r_+$ and $r_-$ obtained with the Fourier theory.

\begin{prop}
\label{prop-Fourier-2}
For every $x_0 \in \mathbb{R}^+$ and every $r_{\pm} \in H^1(\mathbb{R})$, we have
\begin{equation}
\label{Fourier-2}
\sup_{x \in (x_0,\infty)} \left\| \langle x \rangle \mathcal{P}^+ \left(\bar{r}_+(z) e^{-2i z x}\right) \right\|_{L^2_z} \leq \| r_+ \|_{H^1}
\end{equation}
and
\begin{equation}
\label{Fourier-3}
\sup_{x \in (x_0,\infty)} \left\| \langle x \rangle \mathcal{P}^- \left( r_-(z) e^{2i z x}\right) \right\|_{L^2_z} \leq \| r_- \|_{H^1},
\end{equation}
where $\langle x \rangle := (1+x^2)^{1/2}$. In addition, if $r_{\pm} \in H^1(\mathbb{R})$, then
\begin{equation}
\label{Fourier-2-inf}
\sup_{x \in \mathbb{R}} \left\| \mathcal{P}^+ \left(\bar{r}_+(z) e^{-2i z x}\right) \right\|_{L^{\infty}_z} \leq
\frac{1}{\sqrt{2}} \| r_+ \|_{H^1}
\end{equation}
and
\begin{equation}
\label{Fourier-3-inf}
\sup_{x \in \mathbb{R}}  \left\| \mathcal{P}^- \left( r_-(z) e^{2i z x}\right) \right\|_{L^{\infty}_z} \leq
\frac{1}{\sqrt{2}} \| r_- \|_{H^1}.
\end{equation}
Furthermore, if $r_{\pm} \in L^{2,1}(\mathbb{R})$, then
\begin{equation}
\label{Fourier-2-3}
\sup_{x \in \mathbb{R}}  \left\| \mathcal{P}^+ \left( z \bar{r}_+(z) e^{-2i z x}\right) \right\|_{L^2_z} \leq
\| z r_+(z) \|_{L_z^{2}},
\end{equation}
and
\begin{equation}
\label{Fourier-2-3a}
\sup_{x \in \mathbb{R}} \left\| \mathcal{P}^- \left( z r_-(z) e^{2i z x} \right) \right\|_{L^2_z} \leq
\| z r_-(z) \|_{L_z^{2}}.
\end{equation}
\end{prop}

\begin{proof}
Recall the following elementary result from the Fourier theory.
For a given function $r \in L^2(\mathbb{R})$, we use the Fourier transform
$\widehat{r} \in L^2(\mathbb{R})$ with the definition
$\widehat{r}(k) := \frac{1}{2\pi} \int_{\mathbb{R}} r(z) e^{-ikz} dz$, so that 
$$
\| r \|_{L^2}^2 = 2 \pi \| \hat{r} \|_{L^2}^2.
$$
Then, we have $r \in H^1(\mathbb{R})$ if and only if $\widehat{r} \in L^{2,1}(\mathbb{R})$.
Similarly, $r \in L^{2,1}(\mathbb{R})$ if and only if $\widehat{r} \in H^1(\mathbb{R})$.

In order to prove (\ref{Fourier-2}), we write explicitly
\begin{eqnarray}
\nonumber
\mathcal{P}^+ \left(\bar{r}_+(z) e^{-2i z x}  \right)(z) & = & \frac{1}{2\pi i} \lim_{\epsilon \downarrow 0}
\int_{\mathbb{R}} \frac{\overline{r}_+(s) e^{-2isx}}{s-(z+i\epsilon)} ds \\
\nonumber
& = & \frac{1}{2\pi i} \int_{\mathbb{R}} \widehat{\overline{r}_+}(k) \left(  \lim_{\epsilon \downarrow 0}
\int_{\mathbb{R}} \frac{e^{i (k-2x) s}}{s-(z+i\epsilon)} ds \right) dk \\
\label{residue-term}
& = & \int_{2x}^{\infty} \widehat{\overline{r}_+}(k) e^{i (k-2x) z} dk,
 \end{eqnarray}
where the following residue computation has been used:
\begin{equation}
\label{residue-computation}
\lim_{\epsilon \downarrow 0} \frac{1}{2\pi i} \int_{\mathbb{R}} \frac{e^{is(k-2x)}}{s-i\epsilon}ds =
\lim_{\epsilon \downarrow 0} \left\{ \begin{array}{l}
e^{-\epsilon(k-2x)}, \quad \mbox{\rm if} \;\; k - 2x > 0 \\
0, \quad \quad \quad \quad \;\; \mbox{\rm if} \;\; k - 2x < 0 \end{array} \right. = \chi(k-2x),
\end{equation}
with $\chi$ being the characteristic function. The bound (\ref{Fourier-2})
is obtained from the bound (\ref{Fourier-1}) of Proposition \ref{prop-Fourier-1}
for every $x_0 \in \mathbb{R}^+$:
$$
\sup_{x \in (x_0,\infty)} \left\| \langle x \rangle \int_{2x}^{\infty} \widehat{\overline{r}_+}(k) e^{i (k-2x) z} dk \right\|_{L^2_z}
\leq \sqrt{2\pi} \| \widehat{r}_+ \|_{L^{2,1}} = \| r_+ \|_{H^1}.
$$
Similarly, we use the representation (\ref{residue-term}) and obtain bound (\ref{Fourier-2-inf})
for every $x \in \mathbb{R}$:
\begin{eqnarray}
\| \mathcal{P}^+ \left(\bar{r}_+(z) e^{-2i z x}  \right)(z) \|_{L^{\infty}_z}
\leq \| \widehat{\overline{r}_+}(k) \|_{L^1_k} \leq \sqrt{\pi} \| \widehat{\overline{r}_+}(k) \|_{L^{2,1}_k}
\leq \frac{1}{\sqrt{2}}  \| r_+ \|_{H^1}.
 \end{eqnarray}
The bounds (\ref{Fourier-3}) and (\ref{Fourier-3-inf}) are obtained similarly from the representation
\begin{eqnarray*}
\mathcal{P}^- \left(r_-(z) e^{2i z x}  \right)(z) = \frac{1}{2\pi i} \lim_{\epsilon \downarrow 0}
\int_{\mathbb{R}} \frac{r_-(s) e^{2isx}}{s-(z-i\epsilon)} ds =  -\int_{-\infty}^{-2x} \widehat{r_-}(k) e^{i (k+2x) z} dk.
 \end{eqnarray*}
The bounds (\ref{Fourier-2-3}) and (\ref{Fourier-2-3a}) follow from
the bound (\ref{bound-projection-operator}) with $C_{p=2} = 1$ of Proposition \ref{prop-RHP}.
\end{proof}

We shall use the estimates of Lemma \ref{inverse-fredholm-cor} and Proposition \ref{prop-Fourier-2}
to derive useful estimates on the solutions to the Riemann--Hilbert problem (\ref{jump}).
By Lemma \ref{inverse-fredholm-cor}, these solutions on the real line can be written in the integral Fredholm form (\ref{RH-M}).
We only need to obtain estimates on the vector columns $\mu_- - e_1$ and $\eta_+ - e_2$.
From equation (\ref{projection-int-type-2}), we obtain
\begin{equation} \label{projection-int-type-4}
\mu_-(x;z) - e_1 = \mathcal{P}^- \left( r_-(z) e^{2i z x} \eta_+(x;z) \right)(z), \quad  z\in \mathbb{R},
\end{equation}
where we have used the following identities
\begin{eqnarray*}
(M_- R)^{(1)} = \Phi_{\infty}^{-1} (P_- R)^{(1)} = r_-(z) e^{2izx} \Phi_{\infty}^{-1} p_+ = r_-(z) e^{2izx} M_+^{(2)},
\end{eqnarray*}
which follow from the representations (\ref{r-definition}), (\ref{correspondence-M}), and (\ref{formula-P-M}),
as well as the scattering relation (\ref{linear-4}). From equation (\ref{projection-int-type-3}), we obtain
\begin{equation} \label{projection-int-type-5}
\eta_+(x;z) - e_2 = \mathcal{P}^+ \left( \bar{r}_+(z) e^{-2i z x} \mu_-(x;z) \right)(z), \quad  z \in \mathbb{R},
\end{equation}
where we have used the following identities
\begin{eqnarray*}
(M_- R)^{(2)} = \Phi_{\infty}^{-1} (P_- R)^{(2)} = \bar{r}_+(z) e^{-2izx} \Phi_{\infty}^{-1} m_+ = \bar{r}_+(z) e^{-2izx} M_-^{(1)},
\end{eqnarray*}
which also follow from the representations (\ref{r-definition}), (\ref{correspondence-M}), and (\ref{formula-P-M}).

Let us introduce the $2$-by-$2$ matrix
\begin{equation}
\label{RH-5aa}
M(x;z) = \left[\mu_-(x;z) - e_1, \quad \eta_+(x;z) - e_2 \right]
\end{equation}
and write the system of integral equations (\ref{projection-int-type-4}) and (\ref{projection-int-type-5}) in the matrix form
\begin{eqnarray}
\label{RH-5}
M - \mathcal{P}^+(M R_+) - \mathcal{P}^-(M R_-)  = F,
\end{eqnarray}
where
\begin{eqnarray}
\label{RH-5a}
R_+(x;z) = \left[ \begin{matrix} 0  & \bar{r}_+(z) e^{-2izx} \\ 0 & 0 \end{matrix} \right], \quad
R_-(x;z) = \left[ \begin{matrix} 0  & 0 \\ r_-(z) e^{2izx}  & 0 \end{matrix} \right]
\end{eqnarray}
and
\begin{eqnarray}
\label{RH-5b}
F(x;z) := \left[ e_2 \mathcal{P}^-(r_-(z) e^{2i z x}), \quad  e_1 \mathcal{P}^+(\bar{r}_+(z) e^{-2i z x}) \right].
\end{eqnarray}
The inhomogeneous term $F$  given by (\ref{RH-5b}) isestimated by Proposition \ref{prop-Fourier-2}.
The following lemma estimates solutions to the system of integral equations (\ref{RH-5}).

\begin{lem}
For every $x_0 \in \mathbb{R}^+$ and every $r_{\pm} \in H^1(\mathbb{R})$,
the unique solution to the system of integral equations (\ref{projection-int-type-4}) and (\ref{projection-int-type-5})
satisfies the estimates
\begin{equation}
\label{Fourier-2a}
\sup_{x \in (x_0,\infty)} \left\| \langle x \rangle  \mu_-^{(2)}(x;z) \right\|_{L^2_z} \leq C \| r_- \|_{H^1}
\end{equation}
and
\begin{equation}
\label{Fourier-3a}
\sup_{x \in (x_0,\infty)} \left\| \langle x \rangle \eta_+^{(1)}(x;z) \right\|_{L^2_z} \leq C \| r_+ \|_{H^1},
\end{equation}
where $C$ is a positive constant that depends on
$\| r_{\pm} \|_{L^{\infty}}$. Moreover, if
$r_{\pm} \in H^1(\mathbb{R}) \cap L^{2,1}(\mathbb{R})$,
then
\begin{equation}
\label{Fourier-2aa}
\sup_{x \in \mathbb{R}} \left\| \partial_x \mu_-^{(2)}(x;z) \right\|_{L^2_z} \leq
C \left( \| r_+ \|_{H^1 \cap L^{2,1}} + \| r_- \|_{H^1 \cap L^{2,1}} \right)
\end{equation}
and
\begin{equation}
\label{Fourier-3aa}
\sup_{x \in \mathbb{R}} \left\| \partial_x \eta_+^{(1)}(x;z) \right\|_{L^2_z} \leq
C \left( \| r_+ \|_{H^1 \cap L^{2,1}} + \| r_- \|_{H^1 \cap L^{2,1}} \right)
\end{equation}
where $C$ is another positive constant that depends on
$\| r_{\pm} \|_{L^{\infty}}$.
\label{cor-Fourier-2}
\end{lem}

\begin{proof}
Using the identity $\mathcal{P}^+ - \mathcal{P}^- = I$ following from
relations (\ref{plemelj}) and the identity
$$
R_+ + R_- = (I - R_+) R,
$$
which follows from the explicit form (\ref{jump-P}), we rewrite
the inhomogeneous equation (\ref{RH-5}) in the matrix form
\begin{equation}
\label{RH-7}
G - \mathcal{P}^-(G R) = F,
\end{equation}
where $G := M (I - R_+)$ is given explicitly from (\ref{RH-5aa}) and (\ref{RH-5a}) by
\begin{equation}
\label{RH-5ab}
G(x;z) = \left[ \begin{array}{cc} \mu_-^{(1)}(x;z) - 1 & \eta_+^{(1)}(x;z) - \bar{r}_+(z) e^{-2izx} (\mu_-^{(1)}(x;z) - 1) \\
\mu_-^{(2)}(x;z) & \eta_+^{(2)}(x;z) - 1 - \bar{r}_+(z) e^{-2izx} \mu_-^{(2)}(x;z) \end{array} \right].
\end{equation}

From the explicit expression (\ref{RH-5b}) for $F(x;z)$, we can see that
the second row vector of $F(x;z)$ and $F(x;z) \tau_1(\lambda)$ remains the same
and is given by $[ \mathcal{P}^-(r_-(z) e^{2i z x}), 0]$. From the explicit expressions (\ref{RH-5ab}),
the second row vector of $G(x;z) \tau_1(\lambda)$ is given by
$$
\left[ \mu_-^{(2)}(x;z),\;\; 2 i \lambda \left( \eta_+^{(2)}(x;z) - 1 - \bar{r}_+(z) e^{-2izx} \mu_-^{(2)}(x;z) \right) \right]
$$
Using bound (\ref{bound-on-inverse-lemma}) for the second row vector of $G(x;z) \tau_1(\lambda)$,
we obtain the following bounds for every $x \in \mathbb{R}$,
\begin{equation}
\label{bound-mu-minus-2}
\| \mu_-^{(2)}(x;z) \|_{L^2_z} \leq C \| \mathcal{P}^-(r_-(z) e^{2i z x}) \|_{L^2_z}
\end{equation}
and
\begin{equation}
\label{bound-eta-plus}
\| 2 i \lambda \left( \eta_+^{(2)}(x;z) - 1 - \bar{r}_+(z) e^{-2izx} \mu_-^{(2)}(x;z) \right) \|_{L^2_z}
\leq C \| \mathcal{P}^-(r_-(z) e^{2i z x}) \|_{L^2_z},
\end{equation}
where the positive constant $C$ only depends on $\| r_{\pm} \|_{L^{\infty}}$.
By substituting bound (\ref{Fourier-3}) of Proposition \ref{prop-Fourier-2} into (\ref{bound-mu-minus-2}),
we obtain bound (\ref{Fourier-2a}). Also note that since $|2 i \lambda \bar{r}_+(z)| = |r(\lambda)|$
and $r(\lambda) \in L^{\infty}_z(\mathbb{R})$, we also obtain from (\ref{bound-mu-minus-2})
and (\ref{bound-eta-plus}) by the triangle inequality,
\begin{equation}
\label{bound-eta-plusplus}
\| 2 i \lambda \left( \eta_+^{(2)}(x;z) - 1  \right) \|_{L^2_z} \leq
C \| \mathcal{P}^-(r_-(z) e^{2i z x}) \|_{L^2_z},
\end{equation}
where the positive constant $C$ still depends on $\| r_{\pm} \|_{L^{\infty}}$ only.

Similarly, from the explicit expression (\ref{RH-5b}) for $F(x;z)$, we can see that
the first row vector of $F(x;z)$ and $F(x;z) \tau_2(\lambda)$ remains the same
and is given by $[ 0, \mathcal{P}^+(\bar{r}_+(z) e^{-2i z x})]$.
From the explicit expressions (\ref{RH-5ab}), the first row vector of
$G(x;z) \tau_2(\lambda)$ is given by
$$
\left[ (2i\lambda)^{-1} (\mu_-^{(1)}(x;z) - 1), \;\; \eta_+^{(1)}(x;z) - \bar{r}_+(z) e^{-2izx} (\mu_-^{(1)}(x;z) - 1) \right]
$$
Using bound (\ref{bound-on-inverse-lemma}) for the first row vector of $G(x;z) \tau_2(\lambda)$,
we obtain the following bounds for every $x \in \mathbb{R}$,
\begin{equation}
\label{bound-mu-minus-1}
\| (2i\lambda)^{-1} (\mu_-^{(1)}(x;z) - 1) \|_{L^2_z} \leq C \| \mathcal{P}^+(\bar{r}_+(z) e^{-2i z x}) \|_{L^2_z}
\end{equation}
and
\begin{equation}
\label{bound-eta-plus-1}
\| \eta_+^{(1)}(x;z) - \bar{r}_+(z) e^{-2izx} (\mu_-^{(1)}(x;z) - 1) \|_{L^2_z} \leq C \| \mathcal{P}^+(\bar{r}_+(z) e^{-2i z x}) \|_{L^2_z},
\end{equation}
where the positive constant $C$ only depends on $\| r_{\pm} \|_{L^{\infty}}$.
Since $|2 i \lambda \bar{r}_+(z)| = |r(\lambda)|$
and $r(\lambda) \in L^{\infty}_z(\mathbb{R})$, we also obtain from (\ref{bound-mu-minus-1}) and (\ref{bound-eta-plus-1})
by the triangle inequality,
\begin{eqnarray}
\nonumber
\| \eta_+^{(1)}(x;z) \|_{L^2_z} & \leq &
\| (2 i \lambda) \bar{r}_+(z) e^{-2izx} (2i \lambda)^{-1} (\mu_-^{(1)}(x;z) - 1) \|_{L^2_z} + C \| \mathcal{P}^+(\bar{r}_+(z) e^{-2i z x}) \|_{L^2_z} \\
& \leq & C' \| \mathcal{P}^+(\bar{r}_+(z) e^{-2i z x}) \|_{L^2_z},
\label{bound-eta-plus-2}
\end{eqnarray}
where the positive constant $C'$ still depends on $\| r_{\pm} \|_{L^{\infty}}$ only.
By substituting bound (\ref{Fourier-2}) of Proposition \ref{prop-Fourier-2} into (\ref{bound-eta-plus-2}),
we obtain bound (\ref{Fourier-3a}).

In order to obtain bounds (\ref{Fourier-2aa}) and (\ref{Fourier-3aa}), we take derivative of
the inhomogeneous equation (\ref{RH-5}) in $x$ and obtain
\begin{eqnarray}
\label{RH-5c}
\partial_x M - \mathcal{P}^+ \left( \partial_x M \right) R_+ - \mathcal{P}^- \left( \partial_x M \right) R_-  = \tilde{F},
\end{eqnarray}
where
\begin{eqnarray*}
\tilde{F} & := & \partial_x F + \mathcal{P}^+ M \partial_x R_+ + \mathcal{P}^- M \partial_x  R_- \\
& = & 2i \left[ e_2 \mathcal{P}^-(z r_-(z) e^{2i z x}), \quad  e_1 \mathcal{P}^+(-z \bar{r}_+(z) e^{-2i z x}) \right] \\
&\phantom{t} & +
2i \left[ \begin{array}{cc} z r_-(z) \eta_+^{(1)}(x;z) e^{2izx} & -z \bar{r}_+(z) (\mu_-^{(1)}(x;z) - 1) e^{-2izx}  \\
z r_-(z) (\eta_+^{(2)}(x;z) - 1) e^{2izx} & -z \bar{r}_+(z) \mu_-^{(2)}(x;z) e^{-2izx} \end{array} \right].
\end{eqnarray*}
Recall that $\lambda r_-(z) \in L^{\infty}_z(\mathbb{R})$ by Proposition \ref{r-boundness}.
The second row vector of $\tilde{F}(x;z) \tau_1(\lambda)$ and the first row vector of $\tilde{F}(x;z) \tau_2(\lambda)$
belongs to $L^2_z(\mathbb{R})$, thanks to bounds (\ref{Fourier-2-3}) and (\ref{Fourier-2-3a}) of Proposition \ref{prop-Fourier-2},
as well as bounds (\ref{bound-on-M-tech}), (\ref{bound-eta-plusplus}), and (\ref{bound-mu-minus-1}). As a result,
repeating the previous analysis, we obtain
the bounds (\ref{Fourier-2aa}) and (\ref{Fourier-3aa}).
\end{proof}

\subsection{Reconstruction formulas}

We shall now recover the potential $u$ of the Kaup--Newell spectral problem (\ref{lax1}) from
the matrices $M_{\pm}$, which satisfy the integral equations (\ref{RH-M}).
This will gives us the map
\begin{equation}
H^1(\mathbb{R})  \cap L^{2,1}(\mathbb{R}) \ni (r_-,r_+) \mapsto u \in H^2(\mathbb{R}) \cap H^{1,1}(\mathbb{R}),
\end{equation}
where $r_-$ and $r_+$ are related by (\ref{r1-r2}).

Let us recall the connection formulas between the potential $u$ and the Jost functions
of the direct scattering transform in Section 2. By Lemma \ref{asympt-mod},
if $u \in H^2(\mathbb{R}) \cap H^{1,1}(\mathbb{R})$, then
\begin{equation}
\label{inverse-1}
\partial_x \left(\bar{u}(x) e^{\frac{1}{2i} \int_{\pm \infty}^x |u(y)|^2 dy} \right) = 2 i \lim_{|z| \to \infty} z m_{\pm}^{(2)}(x;z).
\end{equation}
On the other hand, by Lemma \ref{asympt-mod} and the representation (\ref{f}),
if $u \in H^2(\mathbb{R}) \cap H^{1,1}(\mathbb{R})$, then
\begin{equation}
\label{inverse-2}
u(x) e^{-\frac{1}{2i} \int_{\pm \infty}^x |u(y)|^2 dy} = -4  \lim_{|z| \to \infty} z p_{\pm}^{(1)}(x;z).
\end{equation}

We shall now study properties of the potential $u$ recovered by
equations (\ref{inverse-1}) and (\ref{inverse-2}) from properties of
the matrices $M_{\pm}$. The two choices in the reconstruction formulas (\ref{inverse-1}) and (\ref{inverse-2})
are useful for controlling the potential $u$ on the positive and negative half-lines.
We shall proceed separately with the estimates on the two half-lines.

\subsubsection{Estimates on the positive half-line}

By comparing (\ref{correspondence-M}) with (\ref{definitions-M}), we rewrite
the reconstruction formulas (\ref{inverse-1}) and (\ref{inverse-2}) for
the choice of $m_+^{(2)}$ and $p_+^{(1)}$ as follows:
\begin{equation}
\label{inverse-1a}
\partial_x \left(\bar{u}(x) e^{\frac{1}{2i} \int_{+\infty}^x |u(y)|^2 dy} \right)
= 2 i e^{-\frac{1}{2i} \int_{+\infty}^x |u(y)|^2 dy}  \lim_{|z| \to \infty} z \mu_-^{(2)}(x;z)
\end{equation}
and
\begin{equation}
\label{inverse-2a}
u(x) e^{-\frac{1}{2i} \int_{+\infty}^x |u(y)|^2 dy} = -4 e^{\frac{1}{2i} \int_{+\infty}^x |u(y)|^2 dy}
 \lim_{|z| \to \infty} z \eta_+^{(1)}(x;z)
\end{equation}

Since $r_{\pm} \in H^1(\mathbb{R})  \cap L^{2,1}(\mathbb{R})$, we have $R(x;\cdot) \in L^1(\mathbb{R}) \cap L^2(\mathbb{R})$
for every $x \in \mathbb{R}$, so that the asymptotic limit (\ref{limit-Cauchy-operator}) in Proposition \ref{prop-RHP} 
is justified since $M_-(x;\cdot) - I \in L^2(\mathbb{R})$ by Lemma \ref{inverse-fredholm-cor}. Therefore, 
we use the solution representation (\ref{RH-M-complex}) and rewrite 
the reconstruction formulas (\ref{inverse-1a}) and (\ref{inverse-2a}) in the explicit form
\begin{eqnarray}
\nonumber
e^{\frac{1}{2i} \int_{+\infty}^x |u(y)|^2 dy} \partial_x \left(\bar{u}(x) e^{\frac{1}{2i} \int_{+\infty}^x |u(y)|^2 dy} \right)
& = & -\frac{1}{\pi} \int_{\mathbb{R}} r_-(z) e^{2 i z x} \left[ \eta_-^{(2)}(x;z)
+ \bar{r}_+(z) e^{-2izx} \mu_-^{(2)} (x;z) \right] dz \\ \label{inverse-3}
& = & -\frac{1}{\pi} \int_{\mathbb{R}} r_-(z) e^{2 i z x} \eta_+^{(2)} (x;z) dz
\end{eqnarray}
and
\begin{equation}
\label{inverse-4}
u(x) e^{i \int_{+\infty}^x |u(y)|^2 dy} = \frac{2}{\pi i}  \int_{\mathbb{R}} \bar{r}_+(z) e^{-2 i z x} \mu_-^{(1)}(x;z) dz.
\end{equation}
where we have used the jump condition (\ref{jump}) for the second equality in (\ref{inverse-3}).

If $r_+, r_- \in H^1(\mathbb{R})$, then the reconstruction formulas (\ref{inverse-3}) and (\ref{inverse-4})
recover $u$ in class $H^{1,1}(\mathbb{R}^+)$. Furthermore, if $r_+,r_-\in L^{2,1}(\mathbb{R})$, then $u$ is in class $H^2(\mathbb{R}^+)$.

\begin{lem} \label{mu-estimate}
Let $r_{\pm}\in H^1(\mathbb{R})\cap L^{2,1}(\mathbb{R})$ such that the inequality (\ref{smallness-defocusing})
is satisfied. Then, $u \in H^2(\mathbb{R}^+) \cap H^{1,1}(\mathbb{R}^+)$ satisfies the bound
\begin{equation} \label{mu-estimate-1}
  \| u \|_{H^2(\mathbb{R}^+) \cap H^{1,1}(\mathbb{R}^+)} \leq C \left( \|r_+\|_{H^1 \cap L^{2,1}} + \|r_-\|_{H^1 \cap L^{2,1}}\right),
\end{equation}
where $C$ is a positive constant  that depends on $\| r_{\pm} \|_{H^1 \cap L^{2,1}}$.
\end{lem}

\begin{proof}
We use the reconstruction formula (\ref{inverse-4}) rewritten as follows:
\begin{eqnarray}
\nonumber
u(x) e^{i \int_{+\infty}^x |u(y)|^2 dy} & = & \frac{2}{\pi i} \int_{\mathbb{R}} \bar{r}_+(z) e^{-2 i z x} dz \\
& \phantom{t} & +
\frac{2}{\pi i} \int_{\mathbb{R}} \bar{r}_+(z) e^{-2 i z x}
\left[ \mu_-^{(1)}(x;z) - 1 \right] dz. \label{repres-u-r}
\end{eqnarray}
The first term is controlled in $L^{2,1}(\mathbb{R})$
because $\overline{r}_+$ is in $H^1(\mathbb{R})$
and its Fourier transform $\widehat{\overline{r}_+}$is in $L^{2,1}(\mathbb{R})$. To control the second term
in $L^{2,1}(\mathbb{R}^+)$, we denote
$$
I(x) :=  \int_{- \infty}^{\infty} \bar{r}_+(z) e^{-2 i z x} \left[ \mu_-^{(1)}(x;z) - 1  \right] dz,
$$
use the inhomogeneous equation (\ref{projection-int-type-4}), and integrate by parts to obtain
\begin{eqnarray*}
I(x) & = & -\int_{\mathbb{R}} r_-(z) \eta_+^{(1)}(x;z) e^{2i z x}
\mathcal{P}^+ \left( \overline{r}_+(z) e^{-2izx} \right)(z) dz.
\end{eqnarray*}
By bounds (\ref{Fourier-2}) in Proposition \ref{prop-Fourier-2}, bound (\ref{Fourier-3a}) in Lemma \ref{cor-Fourier-2}, and
the Cauchy--Schwarz inequality, we have for every $x_0 \in \mathbb{R}^+$,
\begin{eqnarray*}
\sup_{x \in (x_0,\infty)} |\langle x \rangle^2 I(x)| & \leq & \| r_- \|_{L^{\infty}}
\sup_{x \in (x_0,\infty)} \| \langle x \rangle\eta_+^{(1)}(x;z) \|_{L^2_z}
\sup_{x \in (x_0,\infty)} \left\|  \langle x \rangle \mathcal{P}^+ \left( \overline{r}_+(z) e^{-2izx} \right) \right\|_{L^2_z} \\
& \leq & C \| r_+ \|_{H^1}^2,
\end{eqnarray*}
where the positive constant $C$ only depends on $\| r_{\pm} \|_{L^{\infty}}$.
By combining the estimates for the two terms with the triangle inequality, we obtain the bound
\begin{equation}
\label{bound-1-u}
\| u \|_{L^{2,1}(\mathbb{R}^+)} \leq C \left(1 + \| r_+ \|_{H^1}\right)\|r_+\|_{H^1}.
\end{equation}

On the other hand, the reconstruction formula (\ref{inverse-3}) can be rewritten in the form
\begin{eqnarray}
\nonumber
e^{\frac{1}{2i} \int_{+\infty}^x |u(y)|^2 dy} \partial_x \left(\bar{u}(x) e^{\frac{1}{2i} \int_{+\infty}^x |u(y)|^2 dy} \right)
& = & -\frac{1}{\pi}  \int_{\mathbb{R}} r_-(z) e^{2 i z x} dz \\
& \phantom{t} & -\frac{1}{\pi} \int_{\mathbb{R}} r_-(z) e^{2 i z x} \left[ \eta_+^{(2)} (x;z) - 1 \right] dz.
\label{repres-uu-r}
\end{eqnarray}
Using the same analysis as above yields the bound
\begin{eqnarray}
\label{bound-2-u}
\left\| \partial_x \left(\bar{u} e^{\frac{1}{2i} \int_{+\infty}^x |u(y)|^2 dy} \right) \right\|_{L^{2,1}(\mathbb{R}^+)}
\leq C \left(1 + \| r_- \|_{H^1} \right) \|r_-\|_{H^1},
\end{eqnarray}
where $C$ is another positive constant that depends on $\| r_{\pm} \|_{L^{\infty}}$.
Combining bounds (\ref{bound-1-u}) and (\ref{bound-2-u}), we set $v(x) := u(x) e^{-\frac{1}{2i} \int_{+\infty}^x |u(y)|^2 dy}$ and obtain
\begin{equation} \label{mu-estimate-1-v}
\| v \|_{H^{1,1}(\mathbb{R}^+)} \leq C \left( \|r_+\|_{H^1} + \|r_-\|_{H^1}\right),
\end{equation}
where  $C$ is a new positive constant  that depends on $\| r_{\pm} \|_{H^1}$.
Since $|v(x)|=|u(x)|$ and $H^1(\mathbb{R})$ is embedded into $L^6(\mathbb{R})$,
the estimate (\ref{mu-estimate-1-v}) implies the bound
\begin{equation} \label{mu-estimate-1-add}
  \| u \|_{H^{1,1}(\mathbb{R}^+)} \leq C \left( \|r_+\|_{H^1} + \|r_-\|_{H^1}\right),
\end{equation}
where $C$ is a positive constant  that depends on $\| r_{\pm} \|_{H^1}$.

In order to obtain the estimate $u$ in $H^2(\mathbb{R}^+)$ and complete the proof
of the bound (\ref{mu-estimate-1}), we differentiate $I$ in $x$, substitute
the inhomogeneous equation (\ref{projection-int-type-4}) and its $x$ derivative,
and integrate by parts to obtain
\begin{eqnarray*}
I'(x) & = &  - 2 i \int_{- \infty}^{\infty} z \bar{r}_+(z) e^{-2 i z x} \left[ \mu_-^{(1)}(x;z) - 1  \right] dz
+ \int_{- \infty}^{\infty} \bar{r}_+(z) e^{-2 i z x} \partial_x \mu_-^{(1)}(x;z) dz \\
& = &  2 i \int_{- \infty}^{\infty} r_-(z) \eta_+^{(1)}(x;z) e^{2izx} \mathcal{P}^+( z \bar{r}_+(z) e^{-2 i z x})(z)  dz \\
& \phantom{t} & - 2 i \int_{- \infty}^{\infty} z r_-(z) \eta_+^{(1)}(x;z) e^{2izx} \mathcal{P}^+( \bar{r}_+(z) e^{-2 i z x})(z)  dz \\
& \phantom{t} & -\int_{- \infty}^{\infty} r_-(z) \partial_x \eta_+^{(1)}(x;z) e^{2izx} \mathcal{P}^+( \bar{r}_+(z) e^{-2 i z x})(z)  dz.
\end{eqnarray*}
Using bounds (\ref{Fourier-2}), (\ref{Fourier-2-inf}) and (\ref{Fourier-2-3}) in Proposition \ref{prop-Fourier-2},
bounds (\ref{Fourier-3a}) and (\ref{Fourier-3aa}) in Lemma \ref{cor-Fourier-2}, as well as
the Cauchy--Schwarz inequality, we have for every $x_0 \in \mathbb{R}^+$,
\begin{eqnarray*}
\sup_{x \in (x_0,\infty)} | \langle x \rangle I'(x)| & \leq & 2 \| r_- \|_{L^{\infty}}
\sup_{x \in (x_0,\infty)} \| \langle x \rangle \eta_+^{(1)}(x;z) \|_{L^2_z}
\sup_{x \in (x_0,\infty)} \left\|  \mathcal{P}^+ \left( z \overline{r}_+(z)  e^{-2izx} \right)\right\|_{L^2_z} \\
& \phantom{t} & + 2 \| z r_- \|_{L^2}
\sup_{x \in (x_0,\infty)} \| \langle x \rangle \eta_+^{(1)}(x;z) \|_{L^2_z}
\sup_{x \in (x_0,\infty)} \left\|  \mathcal{P}^+ \left( \overline{r}_+(z) e^{-2izx}\right) \right\|_{L^{\infty}_z}\\
& \phantom{t} & + \| r_- \|_{L^{\infty}}
\sup_{x \in (x_0,\infty)} \| \partial_x \eta_+^{(1)}(x;z) \|_{L^2_z}
\sup_{x \in (x_0,\infty)} \left\| \langle x \rangle \mathcal{P}^+ \left( \overline{r}_+(z) e^{-2izx} \right)\right\|_{L^2_z}\\
& \leq & C \| r_- \|_{H^1 \cap L^{2,1}}  \| r_+ \|_{H^1 \cap L^{2,1}}
\left( \| r_+ \|_{H^1 \cap L^{2,1}} + \| r_- \|_{H^1 \cap L^{2,1}} \right),
\end{eqnarray*}
where $C$ is a positive constant that only depends on $\| r_{\pm} \|_{L^{\infty}}$.
This bound on $\sup_{x \in \mathbb{R}^+} |\langle x \rangle I'(x)|$ is sufficient
to control $I'$ in $L^2(\mathbb{R}^+)$ norm and hence the derivative of (\ref{repres-u-r}) in $x$.
Using the same analysis for the derivative of (\ref{repres-uu-r}) in $x$ yields similar estimates.
The proof of the bound (\ref{mu-estimate-1}) is complete.
\end{proof}

By Lemma \ref{mu-estimate}, we obtain the existence of the mapping
\begin{equation}
\label{mapping-r-u}
H^1(\mathbb{R})  \cap L^{2,1}(\mathbb{R}) \ni (r_-,r_+) \mapsto u \in H^2(\mathbb{R}^+) \cap H^{1,1}(\mathbb{R}^+).
\end{equation}
We now show that this map is Lipschitz.

\begin{cor}
\label{cor-mu-estimate}
Let $r_{\pm}, \tilde{r}_{\pm} \in H^1(\mathbb{R}) \cap L^{2,1}(\mathbb{R}) $ satisfy $\| r_{\pm} \|_{H^{1}\cap L^{2,1}},
\| \tilde{r}_{\pm} \|_{H^{1}\cap L^{2,1}} \leq \rho$
for some $\rho > 0$. Denote the corresponding potentials by $u$ and $\tilde{u}$ respectively.
Then, there is a positive $\rho$-dependent constant $C(\rho)$ such that
\begin{equation} \label{mu-estimate-1-lipschitz}
  \| u - \tilde{u} \|_{H^2(\mathbb{R}^+) \cap H^{1,1}(\mathbb{R}^+)} \leq C(\rho) \left( \|r_+ - \tilde{r}_+\|_{H^1 \cap L^{2,1}}
  + \|r_- - \tilde{r}_- \|_{H^1 \cap L^{2,1}}\right).
\end{equation}
\end{cor}

\begin{proof}
By the estimates in Lemma \ref{mu-estimate},
if $r_{\pm} \in H^1(\mathbb{R}) \cap L^{2,1}(\mathbb{R})$, then the quantities
$$
v(x) := u(x) e^{i \int_{+\infty}^x |u(y)|^2 dy} \quad \mbox{\rm and} \quad
w(x) := \left( \partial_x u(x) + \frac{i}{2} |u(x)|^2 u(x) \right) e^{i \int_{+\infty}^x |u(y)|^2 dy},
$$
are defined in function space $H^1(\mathbb{R}^+) \cap L^{2,1}(\mathbb{R}^+)$.
Lipschitz continuity of the corresponding mappings
follows from the reconstruction formula (\ref{repres-u-r}) and (\ref{repres-uu-r})
by repeating the same estimates in Lemma \ref{mu-estimate}. Since $|v| = |u|$,
we can write
$$
u - \tilde{u} = (v - \tilde{v}) e^{-i  \int_{+\infty}^x |v(y)|^2 dy} +
\tilde{v} \left( e^{-i  \int_{+\infty}^x |v(y)|^2 dy}  - e^{-i  \int_{+\infty}^x |\tilde{v}(y)|^2 dy}   \right).
$$
Therefore, Lipschitz continuity of the mapping $(r_+,r_-) \mapsto v \in H^1(\mathbb{R}^+) \cap L^{2,1}(\mathbb{R}^+)$
is translated to Lipschitz continuity of the mapping
$(r_+,r_-) \mapsto u \in H^1(\mathbb{R}^+) \cap L^{2,1}(\mathbb{R}^+)$. Using a similar
representation for $\partial_x u$ in terms of $v$ and $w$, we obtain Lipschitz continuity of
the mapping (\ref{mapping-r-u}) with the bound (\ref{mu-estimate-1-lipschitz}).
\end{proof}

\subsubsection{Estimates on the negative half-line}

Estimates on the positive half-line were found from the reconstruction formulas \eqref{inverse-1a} and \eqref{inverse-2a},
which only use estimates of vector columns $\mu_-$ and $\eta_+$, as seen in \eqref{inverse-3} and \eqref{inverse-4}.
By comparing (\ref{correspondence-M}) with (\ref{definitions-M}),  we can rewrite
the reconstruction formulas (\ref{inverse-1}) and (\ref{inverse-2}) for
the lower choice of $m_-^{(2)}$ and $p_-^{(2)}$ as follows:
\begin{equation}
\label{inverse-1b}
\partial_x \left(\bar{u}(x) e^{\frac{1}{2i} \int_{-\infty}^x |u(y)|^2 dy} \right)
= 2 i e^{-\frac{1}{2i} \int_{+\infty}^x |u(y)|^2 dy} a_{\infty} \lim_{|z| \to \infty} z \mu_+^{(2)}(x;z)
\end{equation}
and
\begin{equation}
\label{inverse-2b}
u(x) e^{-\frac{1}{2i} \int_{-\infty}^x |u(y)|^2 dy} = -4  e^{\frac{1}{2i} \int_{+\infty}^x |u(y)|^2 dy}
\bar{a}_{\infty} \lim_{|z| \to \infty} z  \eta_-^{(1)}(x;z),
\end{equation}
where $a_{\infty} := \lim_{|z| \to \infty} a(z) = e^{\frac{1}{2i} \int_{\mathbb{R}} |u(y)|^2 dy}$.
If we now use the same solution representation (\ref{RH-M-complex})
in the reconstruction formulas (\ref{inverse-1b}) and (\ref{inverse-2b}),
we obtain the same explicit expressions (\ref{inverse-3}) and (\ref{inverse-4}).
On the other hand, if we rewrite the Riemann--Hilbert problem (\ref{jump})
in an equivalent form, we will be able to find nontrivial representation formulas for $u$,
which are useful on the negative half-line. To do so,
we need to factorize the scattering matrix $R(x;z)$ in an equivalent form.

Let us consider the scalar Riemann--Hilbert problem
\begin{equation}
\label{scalar-RH}
\left\{ \begin{array}{l} \delta_+(z) - \delta_-(z) = \bar{r}_+(z) r_-(z) \delta_-(z), \quad z \in \mathbb{R}, \\
\delta_{\pm}(z) \to 1 \quad \mbox{\rm as} \quad |z| \to \infty,\end{array} \right.
\end{equation}
and look for analytic continuations of functions $\delta_{\pm}$ in $\mathbb{C}^{\pm}$.
The solution to the scalar Riemann--Hilbert problem (\ref{scalar-RH})
and some useful estimates  are reported in the following two propositions,
where we recall from (\ref{r1-r2-other}) that
$$
\left\{ \begin{array}{lr} 1+\overline{r}_+(z)r_-(z)=1 + |r(\lambda)|^2 \geq 1, & z \in \mathbb{R}^+, \\
1+\overline{r}_+(z)r_-(z) = 1 - |r(\lambda)|^2 \geq c_0^2 > 0, & z \in \mathbb{R}^-,\end{array} \right.
$$
where the latter inequality is due to (\ref{smallness-defocusing}).

\begin{prop} \label{delta-analytic}
Let $r_{\pm}\in H^1(\mathbb{R})\cap L^{2,1}(\mathbb{R})$ such that the inequality (\ref{smallness-defocusing})
is satisfied. There exists unique analytic functions $\delta_{\pm}$ in $\mathbb{C}^{\pm}$ of the form
\begin{equation} \label{delta}
\delta(z) = e^{\mathcal{C} \log(1+\overline{r}_+r_- )}, \quad z\in \mathbb{C}^{\pm},
\end{equation}
which solve the scalar Riemann--Hilbert problem (\ref{scalar-RH}) and which have the limits
\begin{equation}
\label{delta-limits}
\delta_{\pm}(z) = e^{\mathcal{P}^{\pm}\log(1+\overline{r}_+r_-)}, \quad z \in \mathbb{R},
\end{equation}
as $z \in \mathbb{C}^{\pm}$ approaches to a point on the real axis by any non-tangential contour in $\mathbb{C}^{\pm}$.
\end{prop}

\begin{proof}
First, we prove that $\log(1+\overline{r}_+r_-) \in L^1(\mathbb{R})$. Indeed,
since $r_{\pm} \in L^{2,1}_z(\mathbb{R}) \cap L^{\infty}(\mathbb{R})$,
we have $\overline{r}_+ r_- \in L^1(\mathbb{R})$. Furthermore,
it follows from the representation (\ref{r-definition}) as well as from
Propositions \ref{r-regularity} and \ref{r-boundness} that
$$
\langle z\rangle |r(\lambda)| \leq |r(\lambda)| + \frac{1}{2} |\lambda| |r_-(z)| \leq C, \quad z \in \mathbb{R},
$$
where $C$ is a positive constant. Therefore,
$$
\log(1+|r(\lambda)|^2)\leq \log(1+C^2\langle z\rangle^{-2}), \quad z \in \mathbb{R}^+, \quad \lambda \in \mathbb{R},
$$
so that $\log(1+\overline{r}_+r_-) \in L^1(\mathbb{R}^+)$. On the other hand,
it follows from the inequality (\ref{smallness-defocusing}) that
$$
|\log(1-|r(\lambda)|^2)| \leq -\log(1-C^2\langle z\rangle^{-2}), \quad z \in \mathbb{R}^-, \quad \lambda \in \mathbb{R},
$$
so that $\log(1+\overline{r}_+r_-) \in L^1(\mathbb{R}^-)$.

Thus, we have $\log(1+\overline{r}_+r_-) \in L^1(\mathbb{R})$. It also follows from the above estimates that
$\log(1+\overline{r}_+r_-) \in L^{\infty}(\mathbb{R})$. By H\"{o}lder inequality, we hence obtain $\log(1+\overline{r}_+r_-) \in L^2(\mathbb{R})$.
By Proposition \ref{prop-RHP} with $p = 2$, the expression (\ref{delta}) defines unique analytic functions in $\mathbb{C}^{\pm}$,
which recover the limits (\ref{delta-limits}) and the limits at infinity: $\lim_{|z| \to \infty} \delta_{\pm}(z) = 1$. Finally, since
$\mathcal{P}^+-\mathcal{P}^-=I$, we obtain
$$
\delta_+(z) \delta_-^{-1}(z) = e^{\log(1+\overline{r}_+(z) r_-(z))} = 1+\overline{r}_+(z)r_-(z), \quad z \in \mathbb{R},
$$
so that $\delta_{\pm}$ given by (\ref{delta}) satisfy the scalar Riemann--Hilbert problem (\ref{scalar-RH}).
\end{proof}

\begin{prop}
\label{r-delta-property}
Let $r_{\pm}\in H^1(\mathbb{R})\cap L^{2,1}(\mathbb{R})$ such that the inequality (\ref{smallness-defocusing})
is satisfied.  Then, $\delta_+ \delta_- r_{\pm} \in H^1(\mathbb{R})\cap L^{2,1}(\mathbb{R})$.
\end{prop}

\begin{proof}
We first note that $\mathcal{P}^+ + \mathcal{P}^- = -i\mathcal{H}$ due to the projection
formulas (\ref{plemelj}), where $\mathcal{H}$ is the Hilbert transform. Therefore,
we write
$$
\delta_+ \delta_- = e^{-i \mathcal{H}\log(1+\overline{r}_+r_-)}.
$$
Since $\log(1+\overline{r}_+r_-) \in L^2(\mathbb{R})$, we have
$\mathcal{H}\log(1+\overline{r}_+r_-) \in L^2(\mathbb{R})$ being a real-valued function.
Therefore, $|\delta_+(z) \delta_-(z)| = 1$ for almost every $z \in \mathbb{R}$.
Then, $\delta_+ \delta_- r_{\pm} \in L^{2,1}(\mathbb{R})$ follows from $r_{\pm} \in L^{2,1}(\mathbb{R})$.

It remains to show that $\partial_z \delta_+ \delta_- r_{\pm} \in L^2(\mathbb{R})$. To do so,
we shall prove that $\partial_z \mathcal{H}\log(1+\overline{r}_+r_-) \in L^2(\mathbb{R})$.
Due to Parseval's identity and the fact $\| \mathcal{H} f \|_{L^2} = \| f \|_{L^2}$ for every
$f \in L^2(\mathbb{R})$, we obtain
\begin{eqnarray*}
\|\partial_z \mathcal{H}\log(1+\overline{r}_+r_-)\|_{L^2} = \|\partial_z \log(1+\overline{r}_+r_-)\|_{L^2}.
\end{eqnarray*}
The right-hand side is bounded since
$\partial_z \log(1+\overline{r}_+r_-)=\frac{\partial_z(\overline{r}_+r_-)}{1+\overline{r}_+r_-} \in L^2(\mathbb{R})$
under the conditions of the proposition. The assertion $\partial_z \delta_+ \delta_- r_{\pm} \in L^2(\mathbb{R})$ is proved.
\end{proof}

Next, we factorize the scattering matrix $R(x;z)$ in an equivalent form:
\begin{align*}
\left[\begin{matrix} \delta_-(z) & 0 \\  0 & \delta_-^{-1}(z) \end{matrix}\right]
\left[ I + R(x;z) \right] \left[\begin{matrix} \delta_+^{-1}(z) & 0 \\  0 & \delta_+(z) \end{matrix}\right]
= \left[\begin{matrix}1 & \delta_-(z) \delta_+(z) \overline{r}_+(z) e^{-2izx} \\
\overline{\delta}_+(z) \overline{\delta}_-(z) r_-(z) e^{2izx} &  1 + \overline{r}_+(z) r_-(z) \end{matrix}\right],
\end{align*}
where we have used $\delta_-^{-1}\delta_+^{-1}=\overline{\delta_-\delta_+}$. Let us now define new jump matrix
$$
\tilde{R}_{\delta}(x;z) := \left[ \begin{matrix} 0 & \overline{r}_{+,\delta}(z) e^{-2 i x z} \\ r_{-,\delta}(z) e^{2 i x z} &
\overline{r}_{+,\delta}(z)  r_{-,\delta}(z)  \end{matrix} \right],
$$
associated with new scattering data
$$
r_{\pm,\delta}(z) := \overline{\delta}_+(z) \overline{\delta}_-(z) r_{\pm}(z).
$$
By Proposition \ref{r-delta-property}, we have $r_{\pm,\delta} \in H^1(\mathbb{R})\cap L^{2,1}(\mathbb{R})$
similarly to the scattering data $r_{\pm}$.

By using the functions $M_{\pm}(x;z)$ and $\delta_{\pm}(z)$, we define functions
\begin{equation}
\label{M-delta}
M_{\pm, \delta}(x;z) := M_{\pm}(x;z) \left[\begin{matrix} \delta_{\pm}^{-1}(z) & 0 \\  0 & \delta_{\pm}(z) \end{matrix}\right].
\end{equation}
By Proposition \ref{delta-analytic},
the new functions $M_{\pm, \delta}(x;\cdot)$ are analytic in $\mathbb{C}^{\pm}$
and have the same limit $I$ as $|z|\rightarrow \infty$.
On the real axis, the new functions satisfy the jump condition
associated with the jump matrix $\tilde{R}_{\delta}(x;z)$.
All together, the new Riemann--Hilbert problem
\begin{equation}
\label{jump-tilde}
\left\{ \begin{array}{l}
M_{+,\delta}(x;z) - M_{-,\delta}(x;z) = M_{-,\delta}(x;z) \tilde{R}_{\delta}(x;z), \quad z \in \mathbb{R}, \\
\lim_{|z| \to \infty} M_{\pm, \delta}(x;z) = I, \end{array} \right.
\end{equation}
follows from the previous Riemann--Hilbert problem (\ref{jump}).
By Corollary \ref{RH-general} and analysis preceding Lemma \ref{inverse-fredholm-cor}, the Riemann--Hilbert problem (\ref{jump-tilde})
admits a unique solution, which is given by the Cauchy operators in the form:
\begin{equation}\label{RH-1-tilde}
M_{\pm,\delta}(x;z) = I + \mathcal{C} \left( M_{-,\delta}(x;\cdot) \tilde{R}_{\delta}(x;\cdot)\right)(z), \quad z\in \mathbb{C}^{\pm}.
\end{equation}

Let us denote the vector columns of $M_{\pm, \delta}$ by $M_{\pm, \delta} = [\mu_{\pm,\delta},\eta_{\pm, \delta}]$.
What is nice in the construction of $M_{\pm, \delta}$ that
$$
\lim_{|z| \to \infty} z \mu_{\pm,\delta}^{(2)}(x;z) =  \lim_{|z| \to \infty} z \mu_{\pm}^{(2)}(x;z) \quad
\mbox{\rm and} \quad \lim_{|z| \to \infty} z \eta_{\pm,\delta}^{(1)}(x;z) =
\lim_{|z| \to \infty} z \eta_{\pm}^{(1)}(x;z).
$$

Since $r_{\pm, \delta} \in H^1(\mathbb{R}) \cap L^{2,1}(\mathbb{R})$,
we have $\tilde{R}_{\delta}(x;\cdot) \in L^1(\mathbb{R}) \cap L^2(\mathbb{R})$ for every $x \in \mathbb{R}$,
so that the asymptotic limit (\ref{limit-Cauchy-operator}) in Proposition \ref{prop-RHP}
is justified for the integral representation (\ref{RH-1-tilde}).
As a result, the reconstruction formulas (\ref{inverse-1b}) and (\ref{inverse-2b}) can be rewritten in the explicit form:
\begin{eqnarray}
e^{\frac{1}{2i} \int_{+\infty}^x |u(y)|^2 dy}  \partial_x \left(\bar{u}(x) e^{\frac{1}{2i} \int_{+\infty}^x |u(y)|^2 dy} \right)
= -\frac{1}{\pi} \int_{\mathbb{R}} r_-(z) e^{2 i z x} \eta_{-,\delta}^{(2)}(x;z) dz \label{inverse-3b}
\end{eqnarray}
and
\begin{eqnarray}
\nonumber
u(x) e^{i \int_{+\infty}^x |u(y)|^2 dy} & = & \frac{2}{\pi i}
\int_{\mathbb{R}} \bar{r}_{+,\delta}(z) e^{-2 i z x} \left[  \mu_{-,\delta}^{(1)}(x;z) + r_{-,\delta}(z) e^{2 i z x} \eta_{-,\delta}^{(1)}(x;z) \right] dz \\
& = & \frac{2}{\pi i} \int_{\mathbb{R}}  \bar{r}_{+,\delta}(z) e^{-2 i z x} \mu_{+,\delta}^{(1)} (x;z) dz,\label{inverse-4b}
\end{eqnarray}
where we have used the first equation of the Riemann--Hilbert problem (\ref{jump-tilde}) for the second equality in (\ref{inverse-4b}).

The reconstruction formulas (\ref{inverse-3b}) and (\ref{inverse-4b}) can be studied similarly to the analysis in
the previous subsection. First, we obtain the system of integral
equations for vectors $\mu_{+,\delta}$ and $\eta_{-,\delta}$ from projections of the solution
representation (\ref{RH-1-tilde}) to the real line:
\begin{eqnarray}
\mu_{+,\delta}(x;z)  & = & e_1 + \mathcal{P}^+ \left( r_{-,\delta} e^{2 i z x} \eta_{-,\delta}(x;\cdot) \right)(z), \label{RH-6-tilde} \\
\eta_{-,\delta}(x;z) & = & e_2 + \mathcal{P}^- \left( \bar{r}_{+,\delta} e^{-2 i z x} \mu_{+,\delta}(x;\cdot) \right)(z). \label{RH-6a-tilde}
\end{eqnarray}
The integral equations above can be written as
\begin{equation}\label{projection-int-negative}
G_{\delta} - \mathcal{P}^-(G_{\delta} R_{\delta}) = F_{\delta},
\end{equation}
where
$$
G_{\delta}(x;z) := \left[ \mu_{+,\delta}(x;z)-e_1, \eta_{-,\delta}(x;z) - e_2 \right]
\left[\begin{matrix} 1 & 0 \\-r_{-,\delta}(z) e^{2izx} &1 \end{matrix}\right]
$$
and
$$
F_{\delta}(x;z) := \left[ e_2\mathcal{P}^+(r_{-,\delta}(z) e^{2izx}), \; e_1\mathcal{P}^-(\overline{r}_{+,\delta}(z) e^{-2izx})\right].
$$
The estimates of Proposition \ref{prop-Fourier-2}, Lemma \ref{cor-Fourier-2}, Lemma \ref{mu-estimate},
and Corollary \ref{cor-mu-estimate}  apply to the system of integral equations (\ref{RH-6-tilde}) and (\ref{RH-6a-tilde})
with the only change: $x_0 \in \mathbb{R}^+$ is replaced by $x_0 \in \mathbb{R}^-$ because the operators $\mathcal{P}^+$ and
$\mathcal{P}^-$ swap their places in comparison with the system (\ref{RH-7}).
As a result, we extend the statements of Lemma \ref{mu-estimate} and Corollary \ref{cor-mu-estimate}
to the negative half-line. This construction yields existence and Lipschitz continuity of the mapping 
\begin{equation}
\label{mapping-r-u-minus}
H^1(\mathbb{R})  \cap L^{2,1}(\mathbb{R}) \ni (r_-,r_+) \mapsto u \in H^2(\mathbb{R}^-) \cap H^{1,1}(\mathbb{R}^-).
\end{equation}

\begin{lem} \label{mu-estimate-minus}
Let $r_{\pm}\in H^1(\mathbb{R})\cap L^{2,1}(\mathbb{R})$ such that the inequality (\ref{smallness-defocusing})
is satisfied. Then, $u \in H^2(\mathbb{R}^-) \cap H^{1,1}(\mathbb{R}^-)$ satisfies the bound
\begin{equation} \label{mu-estimate-1-minus}
  \| u \|_{H^2(\mathbb{R}^-) \cap H^{1,1}(\mathbb{R}^-)} \leq C
  \left( \|r_{+,\delta}\|_{H^1 \cap L^{2,1}} + \|r_{-,\delta}\|_{H^1 \cap L^{2,1}}\right),
\end{equation}
where $C$ is a positive constant  that depends on $\| r_{\pm,\delta} \|_{H^1 \cap L^{2,1}}$.
\end{lem}

\begin{cor}
\label{cor-mu-estimate-minus}
Let $r_{\pm}, \tilde{r}_{\pm} \in H^1(\mathbb{R}) \cap L^{2,1}(\mathbb{R}) $ satisfy $\| r_{\pm} \|_{H^{1}\cap L^{2,1}},
\| \tilde{r}_{\pm} \|_{H^{1}\cap L^{2,1}} \leq \rho$
for some $\rho > 0$. Denote the corresponding potentials by $u$ and $\tilde{u}$ respectively.
Then, there is a positive $\rho$-dependent constant $C(\rho)$ such that
\begin{equation} \label{mu-estimate-1-lipschitz-minus}
  \| u - \tilde{u} \|_{H^2(\mathbb{R}^-) \cap H^{1,1}(\mathbb{R}^-)} \leq C(\rho) \left( \|r_+ - \tilde{r}_+\|_{H^1 \cap L^{2,1}}
  + \|r_- - \tilde{r}_- \|_{H^1 \cap L^{2,1}}\right).
\end{equation}
\end{cor}

\begin{remark}
Since Corollaries \ref{cor-mu-estimate} and \ref{cor-mu-estimate-minus} yield Lipschitz continuity of the mappings (\ref{mapping-r-u}) 
and (\ref{mapping-r-u-minus}) for every $r_{\pm}$, $\tilde{r}_{\pm}$ in a ball of a fixed (but possibly large) radius $\rho$,
the mappings (\ref{mapping-r-u}) and (\ref{mapping-r-u-minus}) are one-to-one for every $r_{\pm}$ in the ball.  \label{remark-inverse}
\end{remark}

\section{Proof of Theorem \ref{main}}

Thanks to the local well-posedness theory in \cite{TF1,TF2} and
the weighted estimates in \cite{H-O-1,H-O-2}, there exists a local solution
$u(t,\cdot)\in H^2(\mathbb{R})\cap H^{1,1}(\mathbb{R})$ to the Cauchy problem (\ref{dnls})
with an initial data $u_0\in H^2(\mathbb{R})\cap H^{1,1}(\mathbb{R})$ for $t\in [0, T]$ for some finite $T>0$.

For every $t \in [0,T]$, we define fundamental solutions
$$
\psi(t,x;\lambda) := e^{-i2\lambda^4t-i\lambda^2x}\varphi_{\pm}(t,x;\lambda)
$$
and
$$
\psi(t,x;\lambda) := e^{i2\lambda^4t+i\lambda^2x}\phi_{\pm}(t,x;\lambda)
$$
to the Kaup--Newell spectral problem (\ref{laxeq1}) and the time-evolution problem (\ref{laxeq2})
associated with the potential $u(t,x)$ that belongs to $C([0,T],H^2(\mathbb{R})\cap H^{1,1}(\mathbb{R}))$.
By Corollaries \ref{Jost-original} and \ref{lemma-Jost-original},
the bounded Jost functions $\varphi_{\pm}(t,x;\lambda)$ and $\psi_{\pm}(t,x;\lambda)$
have the same analytic property in $\lambda$ plane
and satisfy the same boundary conditions
$$
\left\{ \begin{array}{l} \varphi_{\pm}(t,x;\lambda) \rightarrow e_1 \\
\phi_{\pm}(t,x;\lambda) \rightarrow e_2 \end{array} \right. \quad \mbox{\rm as} \quad
x \rightarrow \pm\infty
$$
for every $t\in [0,T]$. From linear independence of two solutions
to the Kaup--Newell spectral problem (\ref{laxeq1}), the bounded Jost functions
satisfy the scattering relation
\begin{equation}\label{time-zero-combination}
\varphi_-(t,x;\lambda) = a(\lambda) \varphi_+(t,x;\lambda) + b(\lambda) e^{2i\lambda^2x + 4 i \lambda^4 t} \phi_+(t,x;\lambda), \quad
x \in \mathbb{R}, \quad \lambda \in \mathbb{R} \cup i \mathbb{R},
\end{equation}
where the scattering coefficients $a(\lambda)$ and $b(\lambda)$ are independent of $(t,x)$
due to the fact that the matrices of the linear system (\ref{laxeq1}) and (\ref{laxeq2}) have zero trace.
Indeed, in this case, the Wronskian determinants are independent of $(t,x)$, so that we have
\begin{eqnarray*}
a(\lambda) & = & W(\varphi_-(t,x;\lambda) e^{-i2\lambda^4t-i\lambda^2 x},\phi_+(t,x;\lambda)e^{i2\lambda^4t+i\lambda^2x}) =
W(\varphi_-(0,0;\lambda),\phi_+(0,0;\lambda)),  \\
b(\lambda) & = & W(\varphi_+(t,x;\lambda)e^{-i2\lambda^4t-i\lambda^2x},\varphi_-(t,x;\lambda)e^{-i2\lambda^4t-i\lambda^2x})
= W(\varphi_+(0,0;\lambda),\varphi_-(0,0;\lambda)).
\end{eqnarray*}
By Lemma \ref{regularity-a-b} and assumptions on zeros of $a$ in the $\lambda$ plane,
we can define the time-dependent scattering data
\begin{equation}
\label{r-definition-time}
r_+(t;z) = -\frac{b(\lambda) e^{4i \lambda^4 t}}{2 i \lambda a(\lambda)}, \quad
r_-(t;z) = \frac{2i \lambda b(\lambda)e^{4i \lambda^4 t}}{a(\lambda)}, \quad z \in \mathbb{R},
\end{equation}
so that the scattering relation (\ref{time-zero-combination}) becomes equivalent to
the first scattering relation in (\ref{linear}). Thus, we define
\begin{equation}
r_{\pm}(t;z) = r_{\pm}(0;z) e^{4 i z^2 t},
\label{time-evolution}
\end{equation}
where $r_{\pm}(0;\cdot)$ are initial spectral data found from the initial condition
$u(0,\cdot)$ and the direct scattering transform in Section 2. By Lemma \ref{regularity-a-b} and Corollary \ref{cor-regularity-a-b},
under the condition that $u_0 \in H^2(\mathbb{R}) \cap H^{1,1}(\mathbb{R})$ admits no resonances of the linear
equation (\ref{laxeq1}), the scattering data $r_{\pm}(0;\cdot)$ is defined in $H^1(\mathbb{R}) \cap L^{2,1}(\mathbb{R})$
and is a Lipschitz continuous function of $u_0$.

Now the time evolution (\ref{time-evolution}) implies that
$r_{\pm}(t;\cdot)$ remains in $H^1(\mathbb{R}) \cap L^{2,1}(\mathbb{R})$ for every $t \in [0,T]$.
Indeed, we have
$$
\| r_{\pm}(t;\cdot) \|_{L^{2,1}} = \| r_{\pm}(0;\cdot) \|_{L^{2,1}} \quad
\mbox{\rm and} \quad
\| \partial_z r_{\pm}(t;\cdot) + 4 i t z r_{\pm}(t;\cdot) \|_{L^2} = \| \partial_z r(0;\cdot) \|_{L^2}.
$$
Hence, $r(t;\cdot) \in H^1(\mathbb{R}) \cap L^{2,1}(\mathbb{R})$ for every $t \in [0,T]$.
Moreover, the constraint (\ref{smallness-defocusing}) and the
relation (\ref{r1-r2}) remain valid for every $t \in [0,T]$.

The potential $u(t,\cdot)$ is recovered from the scattering data $r_{\pm}(t;\cdot)$
with the inverse scattering transform in Section 4.
By Lemmas \ref{mu-estimate}, \ref{mu-estimate-minus} and Corollaries \ref{cor-mu-estimate}, \ref{cor-mu-estimate-minus},
the potential $u(t,\cdot)$ is defined in $H^2(\mathbb{R}) \cap H^{1,1}(\mathbb{R})$ for every $t \in [0,T]$
and is a Lipschitz continuous function of $r(t;\cdot)$. Thus, for every $t\in [0, T)$
we have proved that
\begin{eqnarray}
\nonumber
\|u(t,\cdot)\|_{H^2 \cap H^{1,1}} & \leq & C_1 \left( \| r_+(t;\cdot) \|_{H^1 \cap L^{2,1}} + \| r_-(t;\cdot) \|_{H^1 \cap L^{2,1}} \right) \\
\nonumber
& \leq & C_2 \left( \| r_+(0;\cdot) \|_{H^1 \cap L^{2,1}} + \| r_-(0;\cdot) \|_{H^1 \cap L^{2,1}} \right) \\
& \leq & C_3 \| u_0 \|_{H^2 \cap H^{1,1}},
\label{bound-on-u}
\end{eqnarray}
where the positive constants $C_1$, $C_2$, and $C_3$ depends on $\| r_{\pm}(t;\cdot) \|_{H^1 \cap L^{2,1}}$,
$(T,\| r_{\pm}(0;\cdot) \|_{H^1 \cap L^{2,1}})$, and $(T,\| u_0 \|_{H^2 \cap H^{1,1}})$ respectively.
Moreover, the map $H^2(\mathbb{R}) \cap H^{1,1}(\mathbb{R}) \ni u_0 \mapsto u \in C([0,T],H^2(\mathbb{R}) \cap H^{1,1}(\mathbb{R}))$
is Lipschitz continuous.

Since $\| r(t;\cdot) \|_{H^1}$ may grow at most linearly in $t$ and constants $C_1,C_2,C_3$ in (\ref{bound-on-u})
depends polynomially on their respective norms, we have
\begin{eqnarray}
\|u(t,\cdot)\|_{H^2 \cap H^{1,1}} \leq C(T) \| u_0 \|_{H^2 \cap H^{1,1}}, \quad t \in [0,T],
\label{bound-on-u}
\end{eqnarray}
where the positive constant $C(T)$ (that also depends on $\| u_0 \|_{H^2 \cap H^{1,1}}$)
may grow at most polynomially in $T$ but it remains finite for every $T > 0$.
From here, we derive a contradiction on the assumption that the local solution $u \in C([0,T],H^2(\mathbb{R})\cap H^{1,1}(\mathbb{R}))$
blows up in a finite time. Indeed, if there exists a maximal existence tim $T_{\rm max} > 0$ such that
$\lim_{t \uparrow T_{\rm max}} \| u(t;\cdot) \|_{H^2 \cap H^{1,1}} = \infty$,
then the bound (\ref{bound-on-u}) is violated as $t \uparrow T$, which is impossible. Therefore,
the local solution $u \in C([0,T],H^2(\mathbb{R})\cap H^{1,1}(\mathbb{R}))$ can be continued globally in
time for every $T > 0$. This final argument yields the proof of Theorem \ref{main}.

Figure \ref{IST} illustrates the proof of Theorem \ref{main} and summarizes the main ingredients of our results.

\begin{figure}[htbp] 
   \centering
   \begin{tikzpicture}
  \node at (-4,3) (u0) {$u_0(x)$};
  \node at (4,3) (r0) {$r_{\pm}(z)$};
  \node at ( 4,1.5) (r) {$r_{\pm}(z)e^{i4z^2t}$};
  \node at (-4,1.5) (u)  {$u(t,x)$};


  \draw[dashed] (5,3.3) to (3,3.3);
  \draw[dashed] (3,3.3) to (3,1);
  \draw[dashed] (3,1) to (5,1);
  \draw[dashed] (5,1) to (5, 3.3);

 \draw[->] (u0) to node[above]{injective and Lipschitz} (r0);
  \draw[->] (r0) to node{} (r);
  \draw[->] (r) to node[below] {injective and Lipschitz} (u);
  \draw[->, dashed] (u0) to node{} (u);
\end{tikzpicture}
   \caption{The scheme behind the proof of Theorem \ref{main}.}
   \label{IST}
\end{figure}
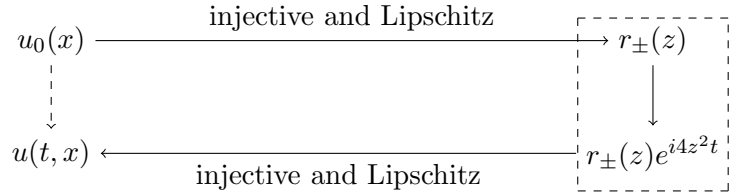

\vspace{0.5cm}

\noindent{\bf Acknowledgement.} The authors are indebted to C. Sulem for bringing this
problem to their attention and for useful discussion. We also thank P. Miller for pointing
out some errors in the earlier draft. D.P. is supported by the NSERC Discovery grant.
Y.S. is supported by the McMaster graduate scholarship.

\end{document}